\documentclass[11pt]{amsart}
\usepackage{mathrsfs}
\usepackage{threeparttable}
\usepackage{amsfonts}
\usepackage{amsmath}
\usepackage{amssymb}
\usepackage{enumerate}
\usepackage{hyperref}
\usepackage{latexsym}
\usepackage{color}
\usepackage{enumitem}
\usepackage{enumerate}
\usepackage{etoolbox} 
\usepackage{ragged2e}
\usepackage{array}
\usepackage{booktabs}
\usepackage{mathabx}
\usepackage{cite}
\usepackage{appendix}
\usepackage{longtable}
\usepackage{multirow}

\def\S{\mathrm{S}}

\def\div{\,\,\big|\,\,}
\def\Ga{\Gamma}
\def\O{{\rm O}}
\def\qed{\hfill $\Box$} \def\demo{\noindent{\it Proof.}~}

\newcommand\ZZ{\mathrm C} \newcommand\D{\mathrm{D}}
\newcommand\A{\mathrm{A}} 
\newcommand\Aut{\mathrm{Aut}} \newcommand\Inn{\mathrm{Inn}} \newcommand\Out{\mathrm{Out}}
\newcommand\Cay{\mathrm{Cay}}   
\newcommand\soc{\mathrm{soc}}
\newcommand\C{\mathbf{C}}
\newcommand\N{\mathbf{N}} \newcommand\Z{\mathbf{Z}}

\newcommand\GL{\mathrm{GL}} \newcommand\SL{\mathrm{SL}}  \newcommand\PGL{\mathrm{PGL}}  \newcommand\PSL{\mathrm{PSL}}
   \newcommand\PSU{\mathrm{PSU}}
  
 \newcommand\PSp{\mathrm{PSp}}

\newcommand\PGammaL{\mathrm{P\Gamma L}}

 \newcommand\F{\mathrm{F}} \newcommand\G{\mathrm{G}}
  
\newcommand\M{\mathrm{M}} \newcommand\Q{\mathrm{Q}}

\newtheorem{theorem}{Theorem}[section]
\newtheorem{corollary}[theorem]{Corollary}
\newtheorem{remark}[theorem]{Remark}
\newtheorem{lemma}[theorem]{Lemma}
\newtheorem{definition}[theorem]{Definition}
\newtheorem{proposition}[theorem]{Proposition}
\theoremstyle{definition}
\newtheorem{example}[theorem]{Example}
\newtheorem{problem}[theorem]{Problem}
\newtheorem{question}[theorem]{Question}
\newtheorem{observation}[theorem]{Observation}

 \def\Syl{\hbox{\rm Syl}}
\def\lg{\langle} \def\rg{\rangle}

\makeatletter
\patchcmd{\@settitle}
  {\@setauthors}
  {\@setauthors\@setaddresses}
  {}{}
\patchcmd{\@maketitle}
  {\@setauthors}
  {\@setauthors\@setaddresses}
  {}{}

\makeatother

\title[Subgroup perfect codes]{Characterizing Finite Groups via Subgroup Perfect Codes}
\thanks{Corresponding author: Wei Meng}
\thanks{Correspondence to: mlwhappyhappy@163.com}
\thanks{The project was partially supported by NSF of Guangxi (2025GXNSFAA069013) and the NNSF of China (12571362).}
\thanks{Email:  libb@st.gxu.edu.cn (B.B. Li),
  lijjhx@gxu.edu.cn (J.J. Li),
  mlwhappyhappy@163.com(W. Meng),
  haoyu@gxu.edu.cn (H. Yu).}

\author[B.B Li, J.J. Li, W. Meng\and H. Yu] 
{Bin Bin Li\textsuperscript{a},
Jing Jian Li\textsuperscript{a},
Wei Meng\textsuperscript{b}
and Hao Yu\textsuperscript{a},
}

\address{%
  \begin{minipage}{\textwidth}%
  \centering%
  \textsuperscript{a}School of Mathematics \& Guangxi University \\%
  Nanning, Guangxi 530004, P. R. China. \\
  \textsuperscript{b}School of Mathematics and Computing Science, Guilin University of Electronic Technology,Guilin 541004, Guangxi, P. R. China.
  \end{minipage}%
}

\begin{document}

\maketitle
\begin{abstract}
A perfect code in a graph $\Gamma = (V, E)$ is a subset $C$ of $V$ such that no two vertices in $C$ are adjacent and every vertex in $V \setminus C$ is adjacent to exactly one vertex in $C$. A subgroup $H$ of a group $G$ is called a subgroup perfect code of $G$ if it is a perfect code in some Cayley graph of $G$. In this paper, we study the set $\Delta(G)$ of conjugacy classes of nontrivial subgroup perfect codes of $G$, with a focus on its relation to $|\pi(G)|$, the number of prime divisors of $|G|$. We prove that $|\Delta(G)| \ge |\pi(G)|$ with only three exceptional families, which leads to the natural question: when is this bound attained or nearly attained?
We completely classify finite groups $G$ satisfying $|\Delta(G)| = |\pi(G)|$ and $|\Delta(G)| = |\pi(G)| + 1$, and we further characterize all non-solvable groups with $|\Delta(G)| \le 6$. Our approach is based on the classification of primitive groups of odd degree, as well as the classification of primitive groups of square-free degree.

\vskip 5pt

\end{abstract}
\maketitle

\parskip 5pt

\section{Introduction}
In this paper, all groups and graphs are finite, and all graphs are simple and undirected.
Let $\Gamma$ be a graph with vertex set $V$.
A perfect code in a graph $\Gamma$ is an independent subset $C$ of $V$ such that every vertex in $V \setminus C$ is adjacent to exactly one vertex in $C$.
The study of perfect codes in graphs originated in coding theory \cite{biggs}. Classical examples such as Hamming and Golay codes\cite{vL} occur in Hamming graphs, which are Cayley graphs of elementary abelian groups. This motivates a natural generalization to perfect codes in arbitrary Cayley graphs. Let $G$ be a group and $S \subseteq G \setminus \{1\}$ an inverse-closed subset. The Cayley graph $\Cay(G, S)$ has vertex set $G$, and two distinct vertices $x$ and $y$ are adjacent if and only if $yx^{-1} \in S$.

Perfect codes in Cayley graphs have been studied extensively (see, for example, \cite{ds,f,zhou}). A particularly active line of research asks when a given subgroup of a group can be realized as a perfect code in some Cayley graph of the group. The formal notion of a subgroup perfect code was introduced by Huang et al. \cite{h}: a subgroup $H$ of $G$ is called a \textit{subgroup perfect code} if $H$ is a perfect code in some Cayley graph of $G$. In that work, the authors gave a necessary and sufficient condition for a normal subgroup to be a perfect code. Their result motivated efforts to characterize subgroup perfect codes, leading to classifications in families such as dihedral groups \cite{h} and generalized quaternion groups \cite{m}, as well as studies involving maximal subgroups of finite simple groups \cite{z}.

Subgroup perfect codes in Cayley graph are important not only in coding theory and graph theory, but also in group theory.
A \textit{factorization of a group $G$ into two factors} is a pair $(A,B)$ of subsets of $G$ such that every element of $G$ can be uniquely written as $ab$ with $a \in A$ and $b \in B$, so that $G = AB$.
If $1 \in A \cap B$, we call it a \textit{tiling} of $G$.
As shown in \cite{h}, if $(A,B)$ is a tiling with $A^{-1} = A$, then $B$ is a perfect code in $\Cay(G,\,A \setminus \{1\})$; conversely, any perfect code in a Cayley graph induces a tilling.
Now let $B$ be a subgroup perfect code of $G$.
There exists an inverse-closed subset $A$ of $G$ such that $(A,B)$ is a factorization of $G$.
In particular, if $A$ is also a subgroup, the factorization $G = AB$ is called an \textit{exact factorization}.
Exact factorizations of permutation groups play a special role, owing to a classical observation: for a transitive group action on a set $\Omega$, a subgroup $A$ acts regularly on $\Omega$ if and only if $G$ admits an exact factorization $G = A G_{\alpha}$ for some point stabilizer $G_{\alpha}$.
For recent developments on exact factorizations, see \cite{her,gi,lx,ww}.

A subgroup perfect code $H$ of a group $G$ is called \textit{nontrivial} if $1<H<G$.
Recently, an interesting problem is to explore the connections between the number of subgroup perfect codes in a group relates to its structure.
For instance, Ma et al. \cite{m} characterized groups in which every subgroup is a subgroup perfect code, showing that such groups contain no element of order 4. Conversely, Chen et al. \cite{ch} proved that groups possessing no nontrivial subgroup perfect codes are precisely cyclic groups of odd prime order, cyclic 2-groups, or generalized quaternion 2-groups. Furthermore, Khaefi et al. \cite{kh} established that any group with exactly two nontrivial subgroup perfect codes must be abelian.

Since the property of being a subgroup perfect code is preserved under conjugation, a natural and refined measure is to count these subgroups up to conjugacy. This leads to an interesting question in group theory:

\begin{problem}
How is the global structure of a group $G$ determined by the number of conjugacy classes of its nontrivial subgroup perfect codes?
\end{problem}
In this paper, we always let $\pi(G)$ denote the set of all prime divisors of the order of a group $G$ and $\Delta(G)$ the set of all conjugacy classes of nontrivial subgroup perfect codes of $G$.

We now introduce a new perspective rooted in the Sylow subgroups of $G$.  A fundamental result by Zhang \cite{z} shows that every Sylow subgroup of $G$ is its a subgroup perfect code. If $G$ is not a $p$-group, its Sylow subgroups for distinct primes are non-conjugate and nontrivial. This yields at least $|\pi(G)|$ distinct conjugacy classes of such subgroups.
 What is the relationship between the sizes of $\pi(G)$ and $\Delta(G)$? Our first theorem answers this question. For two groups $H$ and $K$, we denote by $H{:}K$ the semidirect product of $H$ by $K$, and by $H.K$ an extension of $H$ by $K$. Let $\Q_{2^n}$ denote the generalized quaternion group of order $2^n$, where $n \ge 3$ is an integer.

\begin{theorem}\label{main}
Let $G$ be a finite group and $P$ be a Sylow $2$-subgroup of $G$. Then $|\Delta(G)| \ge |\pi(G)|$ unless $G \cong \ZZ_p$, $\ZZ_{2^m}$ or $\Q_{2^{m+1}}$, where $p$ is a prime and $m \ge 2$.
Moreover, equality holds if and only if $G$ belongs to one of the cases listed in Table \ref{tab:m=pi}.

\begin{table}[ht]
\centering
\caption{}
\begin{tabular}{ccc}
\toprule
$G$ & $|\Delta(G)|$ & Conditions \\
\midrule
$\ZZ_{p^2}$ & $1$ & $p$ odd prime \\
$\ZZ_q{:}\ZZ_p$ & $2$ & $p, q$ odd primes, $p < q$ \\
$\ZZ_p {:} \Q_{2^n}$ & $2$ & $p$ odd prime, $n \geq 3$ \\
$\ZZ_p {:} \ZZ_{2^n}$ & $2$ & $p$ odd prime, $n \geq 1$ \\
$\Q_8{:}\ZZ_3$ & $2$ & \\
$\SL(2,3).\ZZ_2$ & $2$ & $P\cong \Q_{16}$\\
\bottomrule
\end{tabular}
\label{tab:m=pi}
\end{table}
\end{theorem}
\begin{remark} Theorem \ref{main} shows that any finite group with $|\Delta(G)| = |\pi(G)|$ is solvable.
This equality condition is best possible in the sense that it cannot be relaxed to $|\Delta(G)| = |\pi(G)| + 1$: there exist non-solvable groups with $|\Delta(G)| = |\pi(G)| + 1$, as shown in
Example \ref{exm:4} and Lemma \ref{exm:5}.
\end{remark}

Having established the inequality $|\Delta(G)| \ge |\pi(G)|$ and characterized the equality case $|\Delta(G)| = |\pi(G)|$ in Table \ref{tab:m=pi}, we now investigate the next possible value: when does $|\Delta(G)| = |\pi(G)| + 1$ hold?

For a group $G$, recall that $\O_2(G)$ denote the maximal normal $2$-subgroup of $G$, i.e., the intersection of all Sylow $2$-subgroups of $G$.

\begin{theorem}\label{main2}
Let $G$ be a finite group and $P$ be a Sylow $2$-subgroup of $G$. Suppose that $|\Delta(G)| = |\pi(G)| + 1$.
\begin{enumerate}[label=(\arabic*), leftmargin=*, font=\normalfont, itemsep=0.5em]
    \item $G$ is non-solvable if and only if the following holds:
    \begin{enumerate}[label=(\roman*), font=\normalfont]
        \item $P \cong \Q_{2^n}$ for some integer $n\ge 3$ and $\O_2(G)>1$ is a cyclic subgroup;

        \item $G/\O_2(G) \in \{\PSL(2,5),\PGL(2,5), \PGL(2,7), \PGL(2,9),\PSL(2,17), \PGL(2,17),\allowbreak \PSL(2,p), \PGL(2,p)\}$, where $p$ is a Fermat prime satisfying $p+1 = 2p_2 p_3$ for primes $p_2 < p_3 < p$.
    \end{enumerate}
    \item If $G$ is solvable, then for some positive integer $n$ and prime $p\ge 3$, one of the following holds:
          \begin{enumerate}[label=(\roman*), font=\normalfont]
                \item $G\cong \ZZ_{p^3}$;

                \item $|G|=2^n $ and $\Z(G)$ is cyclic, all involutions of $G$ split into two conjugacy classes;

                \item $|G|=2^n p^a$ for $a\in\{1,2\}$, $G/\O_2(G)\cong \ZZ_p^a{:}(P/\O_2(G))$, and the involutions of $G$ form a single conjugacy class; moreover, if $a=2$, then $P\cong \ZZ_{2^n}$ or $\Q_{2^n}$.

         \end{enumerate}
\end{enumerate}
\end{theorem}

The following result is a direct consequence of Theorem \ref{main2}.
\begin{corollary}\label{1.3cor}
Let $G$ be a finite non-solvable group with $|\pi(G)|=5$. Then $|\Delta(G)|\ge 7$.
\end{corollary}

Note that $\A_5$ is the smallest non-abelian simple group and has seven nontrivial subgroup conjugacy classes. Since $\A_5$ has no element of order $4$, by \cite[Theorem 1.1]{m}, every subgroup of $\A_5$ is a subgroup perfect code. Consequently, $\A_5$ has seven conjugacy classes of nontrivial subgroup perfect codes. This leads to a natural question:
\begin{question}\label{que}
For a finite group, if it has at most six conjugacy classes of nontrivial subgroup perfect codes, is the group necessarily solvable?
\end{question}

To answer the Question \ref{que}, Corollary \ref{cor} shows that for an arbitrary finite group $G$ with at most $3$ conjugacy classes of nontrivial subgroup perfect codes, it must be solvable and $|\pi(G)|\leq 2$.
Let $n_{2'}$ denote the odd part of a positive integer $n$. The following result is to answer the Question \ref{que}.

\begin{theorem}\label{main3}
Assume that $G$ is a finite group and $|\Delta(G)|\le 6$. If $G$ is non-solvable, then one of the following holds:
\begin{enumerate}[font=\normalfont]
\item $G/\O_2(G)\cong {^2\!B_2}(8)$;

\item $\soc(G/\O_2(G))\cong \PSL(2,q)$,
where $q\in \{5, 7, 9, 17, p\}$ and $p$ is an odd prime such that $p\notin \{5, 7, 17\}$ and $(p^2-1)_{2'}$ is a product of two distinct primes.
\end{enumerate}
\end{theorem}

\begin{remark} In Theorem \ref{main3}, if $|\Delta(G)|\le 3$, then $G$ must be solvable (see Corollary \ref{cor}) and if $4\leqslant|\Delta(G)|\leqslant6$, then the group $G$ may be solvable or non-solvable (see Examples \ref{exm:sol}--\ref{exm:4} and Lemmas \ref{exm:5}--\ref{exm:6}).
\end{remark}

For a solvable group, there is a strict lower bound on the number of conjugacy classes of nontrivial subgroup perfect codes.

\begin{theorem}\label{sol}
Let $G$ be a finite solvable group. Then $|\Delta(G)| \ge 2^{|\pi(G)|} - 2$, with equality if and only if the odd part $|G|_{2'}$ is square-free and the Sylow $2$-subgroup of $G$ is a cyclic or generalized quaternion $2$-group.
\end{theorem}

Following this introduction, we present the necessary notations, definitions, and preliminary results in Section \ref{sec2}.  The proofs of our main results (Theorems \ref{main}, \ref{main2}, \ref{main3} and \ref{sol}) appear in Sections \ref{sec3}, \ref{sec4}, \ref{sec5} and \ref{Theorem 1.7}, respectively.

\section{Preliminaries}\label{sec2}

This section collects fundamental results on subgroup perfect codes, standard group-theoretic tools, and the theory of primitive permutation groups of square-free degree as well as related simple groups.

Notations and terminologies used in the paper are standard and can be found in \cite{hu1}. For example, for a group $G$, let $\Syl_p(G)$ denote the set of all Sylow $p$-subgroups of $G$ for each $p\in \pi(G)$, and for a subgroup $H$ of $G$, let $[H]$ denote the conjugacy class of $H$ in $G$. Denote $\Z(G)$ the center of group $G$. Let $\ZZ_n$ denote the cyclic group of order $n$, and let $E_q$ denote the elementary abelian $p$-group of order $q = p^f$, where $p$ is a prime.

\subsection{Subgroup perfect codes}
We begin with fundamental criteria for subgroup perfect codes, notably Propositions \ref{odd-index}, \ref{chen}, and \ref{chen2}.

\begin{proposition}[{\cite[Theorem 3.5]{z0}}]\label{odd-index}
Let $G$ be a group and $H$ a subgroup of $G$. If either the order of $H$ is odd or the index of $H$ in $G$ is odd, then $H$ is a subgroup perfect code of $G$.
\end{proposition}

\begin{proposition}[{\cite[Theorem 1.2]{ch}}] \label{chen}
Let $G$ be a group and $H\leqslant G$. Then the following are equivalent$:$
\begin{itemize}[font=\normalfont]
\item [(1)]  $H$ is a subgroup perfect code of $G$;

\item [(2)]  there exists an inverse-closed right transversal of $H$ in $G$;

\item [(3)]  for each $x \in G$ such that $x^2\in  H$ and $|H|/|H \cap H^x|$ is odd, there exists $y \in Hx$ such that $y^2 = 1$;

\item [(4)]  for each $x\in G$ such that $HxH = Hx^{-1}H$ and $|H|/|H \cap H^x|$ is odd, there exists $y \in Hx$ such that $y^2 = 1$.
\end{itemize}
\end{proposition}

\begin{proposition}[{\cite[Theorem 1.2]{z}}]\label{zhang2}
Let $G$ be a finite group and $H\leq G$.
Let $Q\in\Syl_2(H)$ and $P\in\Syl_2(\N_G(Q))$.
Then $H$ is a subgroup perfect code of $G$ if and only if $Q$ is a subgroup perfect code of $P$.
\end{proposition}

\begin{corollary}\label{2pc}
Let $G$ be a finite group and $H$ a $2$-subgroup of $G$. Let $P\in \Syl_2(\N_G(H))$.
Then $H$ is a subgroup perfect code of $G$ if and only if for each $x \in P$ such that $x^2\in  H$, there exists $y \in Hx$ such that $y^2 = 1$.
\end{corollary}
\demo By Proposition \ref{zhang2}, $H$ is a subgroup perfect code of $G$ if and only if $H$ is a subgroup perfect code of $P$. Furthermore, by Proposition \ref{chen}, $H$ is a subgroup perfect code of $P$ if and only if for each $x \in P$ such that $x^2\in  H$, there exists $y \in Hx$ such that $y^2 = 1$. Hence, Corollary \ref{2pc} holds. \qed

\begin{proposition}[{\cite[Theorem 1.3]{ch}}]\label{chen2}
Let $G$ be a group of composite order. Then the following are equivalent:
\begin{enumerate} [font=\normalfont]
    \item[(a)] $G$ has no nontrivial proper subgroup as a subgroup perfect code;
    \item[(b)] $G$ is a $2$-group with a unique involution;
    \item[(c)] $G$ is either a cyclic $2$-group or a generalized quaternion $2$-group.
\end{enumerate}
\end{proposition}

\begin{lemma}\label{cg}
Let $G$ be a finite group whose Sylow $2$-subgroup is cyclic or generalized quaternion. A subgroup $H$ of $G$ is a subgroup perfect code if and only if $H$ has odd order or odd index.
\end{lemma}
\begin{proof}
By Proposition \ref{odd-index}, it suffices to prove that if $H$ is a nontrivial subgroup perfect code of $G$, then $H$ has odd order or odd index. Suppose for contradiction that $H$ has both even order and even index.
Let $Q \in \Syl_2(H)$ and $R \in \Syl_2(\N_G(Q))$. By the Sylow theorems, there exists a Sylow $2$-subgroup $P$ of $G$ containing $R$. Proposition \ref{zhang2} applied to $P$ shows that $Q$ is a subgroup perfect code of $P$. Noting that $P$ is a cyclic 2-subgroup or a generalized quaternion 2-subgroup, since $Q \neq 1$, applying Proposition \ref{chen2}, one yields that $P$ has no nontrivial proper subgroup as a subgroup perfect code, and so $Q = P$. But then $H$ has odd index, contradicting our assumption.
Therefore, $H$ has odd order or odd index and Lemma \ref{cg} holds.
\end{proof}

\begin{proposition}[{\cite[Lemma 3.3]{kh}}]\label{kf}
Let $G$ be a $2$-group with more than one involution. Then the following holds:
\begin{enumerate} [font=\normalfont]
\item[(a)] there exists a nontrivial subgroup perfect code $K$ of $G$, which is isomorphic to a cyclic or a generalized quaternion $2$-group.

\item[(b)] if $\{H_i \mid i = 1, \dots, n\}$ is the set of all nontrivial proper subgroup perfect codes of $G$,
then there exists no involution in $G \setminus \left( \bigcup_{i=1}^n H_i \right)$.
\end{enumerate}
\end{proposition}

In fact, statement (a) can be generalized to an arbitrary group of even order, as in the following lemma.

\begin{lemma}\label{rm}
Let $G$ be a finite group of even order. Then there exists a proper nontrivial subgroup perfect code of $G$, which is isomorphic to a cyclic or a generalized quaternion $2$-group.
\end{lemma}
\demo Suppose for contradiction that no cyclic or generalized quaternion $2$-subgroup of $G$ is a subgroup perfect code. Noting that $G$ has even order, one yields that $G$ must contain a cyclic $2$-subgroup. Among all cyclic or generalized quaternion 2-subgroups of $G$, take $L$ to be one of maximal order. The assumption implies that $L$ fails to be a subgroup perfect code of $G$.
According to the contrapositive of Corollary \ref{2pc}, we can conclude that there is a nontrivial $2$-element $g \in \N_G(L)$ such that $g^2 \in L$ and the coset $Lg$ contains no involution. Then $M:=\langle L, g \rangle$ is a $2$-group with a unique involution and $L< M$.

According to the equivalence of (b) and (c) in Proposition \ref{chen2}, we deduce that $M$ is either a cyclic 2-group or a generalized quaternion 2-group.
This contradicts the maximality of $L$. Therefore, there exists a nontrivial subgroup perfect code of $G$ that is isomorphic to a cyclic or a generalized quaternion $2$-group and Lemma \ref{rm} holds.
\qed

\subsection{Standard group-theoretic tools}
We recall standard results on Hall subgroups in solvable groups.

\begin{definition}[{\cite[p. 659, Definition 1.2]{hu1}}]
Let $\pi$ be a set of primes and $\pi'$ the set of all primes not in $\pi$.
\begin{enumerate}[label=\textup{(\arabic*)}]
\item A group $G$ is called a \emph{$\pi$-group} if every prime divisor of $|G|$ lies in $\pi$.
\item A subgroup $H$ of $G$ is called a \emph{$\pi$-Hall subgroup} if $H$ is a $\pi$-group and $|G:H|$ is divisible by no prime in $\pi$, $\text{i.e., } (|H|, |G:H|) = 1$.
\end{enumerate}
\end{definition}

In particular, a subgroup $H$ is a \textit{Hall subgroup} of $G$ if it is a $\pi$-Hall subgroup for some $\pi \subseteq \pi(G)$, equivalently, $(|H|, |G:H|) = 1$.

\begin{proposition}[{\cite[p. 662, Hauptsatz 1.8]{hu1}}]\label{Hall}
If $G$ is solvable, then $G$ possesses $\pi$-Hall subgroups for every subset $\pi$ of $\pi(G)$. Any two $\pi$-Hall subgroups of $G$ are conjugate in $G$, and every $\pi$-subgroup of $G$ is contained in a $\pi$-Hall subgroup of $G$.
\end{proposition}

The following two propositions are standard results on group automorphisms.

\begin{proposition}[{\cite[p.189, Lemma 4.1]{gorenstein} and \cite[p.79]{zhu}}]\label{aut}
The automorphism groups of the cyclic $2$-groups and the generalized quaternion groups are given as follows.
\begin{enumerate}[font=\normalfont]
 \item For $n \geq 2$, $\Aut(\ZZ_{2^n}) \cong \ZZ_2 \times \ZZ_{2^{n-2}}$.

 \item $\Aut(\Q_8) \cong \S_4$, while $\Aut(\Q_{2^n})$ is a $2$-group for $n \geq 4$.
\end{enumerate}
\end{proposition}

\subsection{Primitive groups of odd and square-free degree}
The proof of Theorems \ref{main} and \ref{main2} relys on the classification of primitive permutation groups of square-free degree.
\begin{proposition}[{\cite[Theorem 1]{Li}}]\label{square-free}
Let $G$ be a primitive permutation group on a finite set $\Omega$, and let $x \in \Omega$.
Assume that $|G|_{2'}$ is square-free and $|\Omega|$ is odd. The following holds.
\begin{enumerate}[font=\normalfont]
\item If the stabilizer $G_x$ is a $2$-group, then $G$ is isomorphic to one of the following:
\begin{enumerate}[font=\normalfont]
    \item $\ZZ_p{:}\ZZ_{2^d}$, where $p \geq 3$ is a prime, $d\ge 0$ and $2^d \mid (p - 1)$;
    \item $\PSL(2, p)$ or $\PGL(2, p)$, where $p \geq 5$ is a prime satisfying that either $p+1$ or $p-1$ is a power of $2$.
\end{enumerate}
\item If $|G_x|_{2'}$ is a prime, then one of the following holds:
\begin{enumerate}[font=\normalfont]
    \item $G\cong \ZZ_p{:}\ZZ_d$, where prime $p\geq 3$ and positive integer $d \mid (p - 1)$;
    \item $\soc(G)\cong \PSL(2, p)$, where $p \geq 5$ is a prime;
    \item $\soc(G)\cong {^2\!B}_2(q)$, where $q=2^{2f+1}$ and integer $f\ge 1$.
\end{enumerate}
\end{enumerate}
\end{proposition}

To handle the $\PSL(2,p)$ and $\PGL(2,p)$ cases from Proposition~\ref{square-free}, we need information about their maximal subgroups.

\begin{proposition}[{\cite[Tables 8.1 and 8.2]{bray}}]\label{bray}
For $q = p^f \geq 5$ with $p$ an odd prime and integer $f\ge 1$,
the maximal subgroups of $\PSL(2,q)$ and $\PGL(2,q)$ are given in Table \ref{psl_pgl_compare}.
\end{proposition}

%

\begin{table}[htbp]
\centering
\caption{Maximal subgroups of $\PSL(2,q)$ and $\PGL(2,q)$}
\label{psl_pgl_compare}
\begin{tabular}{|l|l|}
\hline
\multicolumn{1}{|c|}{$\PSL(2,q)$} & \multicolumn{1}{c|}{$\PGL(2,q)$} \\
\hline
$\ZZ_p^f{:}\ZZ_{(q-1)/2}$ & $\ZZ_p^f{:}\ZZ_{q-1}$ \\
\hline
$\D_{q-1}$ ($q \geq 13$) & $\D_{2(q-1)}$ ($q > 5$) \\
\hline
$\D_{q+1}$ ($q \neq 7,9$) & $\D_{2(q+1)}$ \\
\hline
$\PGL(2,q_0)$ ($q = q_0^2$) & $\PGL(2,q_0)$ ($q = q_0^r$) \\
\hline
$\PSL(2,q_0)$ ($q = q_0^r$, $r$ odd prime) & $\PSL(2,q)$ \\
\hline
$\S_4$ ($q = p \equiv \pm 1 \pmod{8}$) & $\S_4$ ($q \equiv \pm 3 \pmod{8}$) \\
\hline
$\A_5$ ($q = p \equiv \pm 1 \pmod{10}$ or $q = p^2$ with $p \equiv \pm 3 \pmod{10}$) & \\
\hline
$\A_4$ ($q = p \equiv \pm 3, 5, \pm 13 \pmod{40}$) & \\
\hline
\end{tabular}
\end{table}

The following elementary facts from number theory will be used later in the proofs.

\begin{lemma}\label{num}
Let $m\geq 2$. If both $2^{m}+1$ and $2^{m-1}+1$ are prime numbers, then $m=2$.
\end{lemma}
\demo For $m=2$, we have that $2^2+1=5$ and $2^{2-1}+1=3$, which are both prime numbers. Now suppose $m\geqslant 3$. If $m$ is odd, then $2^m+1=(2+1)(2^{m-1}-2^{m-2}+\dots+1)$ is composite, a contradiction. If $m$ is even, then $m-1$ is odd and $m-1 \geqslant 3$, so similarly $2^{m-1}+1$ is composite, a contradiction again. Hence, $m=2$ and Lemma \ref{num} holds. \qed

\begin{lemma}\label{num2}
Let $p \ge 5$ be an odd prime. Then neither $p^2-1$ nor $p^2+1$ is a power of $2$.
\end{lemma}
\proof  Write $p^2-1 = (p-1)(p+1)$.
If $p^2-1 = 2^m$ for some $m \ge 1$, then both $p-1$ and $p+1$ are powers of $2$. Since they differ by $2$, the only possibility is $\{p-1,p+1\} = \{2,4\}$, giving $p=3$. This contradicts the assumption that $p \ge 5$. Hence $p^2-1$ is not a power of $2$.

Now consider $p^2+1$.
As $p\geqslant 5$ is an odd prime, write $p = 2k+1$ with $k \ge 2$. Then
$p^2+1 = (2k+1)^2 + 1 = 4k^2 + 4k + 2$.
Thus $p^2+1 \equiv 2 \pmod{4}$.
Any integer that is a power of $2$ greater than $2$ is divisible by $4$, so $p^2+1$ cannot be a power of $2$.
\qed

Recall that a prime $p$ is a Mersenne prime if $p=2^k-1$ with $k$ is a prime.

\begin{lemma}\label{num3}
Let $p\ge 5$ be a prime. Then the following statements hold.
\begin{enumerate} [font=\normalfont]
\item [(1)] If $p+1$ is a power of $2$ and $\frac{p-1}{2}$ is a power of an odd prime, then $p=7$.

\item [(2)] If $p-1$ is a power of $2$ and $\frac{p+1}{2}$ is a power of an odd prime, then $p=5$ or $17$.
\end{enumerate}
\end{lemma}
\proof (1) Suppose that $p+1$ is a power of 2 and $\frac{p-1}{2}$ is a power of an odd prime. Then $p=2^k-1$ with $k\ge 3$ an integer. Further, since $p$ is a prime, one yields that $k$ is a prime, and hence $p$ is a Mersenne prime. If $k=3$, then $p=7$ and so $\frac{7-1}{2}=3$, as desired.
Now suppose that $k\ge 5$. Then $\frac{p-1}{2}=2^{k-1}-1=(2^{(k-1)/2}-1)(2^{(k-1)/2}+1)$. However, since $(2^{(k-1)/2}-1,2^{(k-1)/2}+1)=1$, this contradicts the assumption that $\frac{p-1}{2}$ is a power of an odd prime.

(2) Suppose that $p-1$ is a power of $2$ and $(p+1)/2$ is a power of an odd prime.
Let $p-1=2^n$ for integer $n\ge 2$. Thus $p=2^n+1$ and $(p+1)/2=2^{n-1}+1$ is a power of an odd prime. Let $2^{n-1}+1=q^t$ for some prime $q\ge 3$ and integer $t$.
Noting that one of $p-1, p$ or $p+1$ is divisible by $3$, since $p-1$ is a power of $2$ and $p\ge 5$ is prime, we deduce that $3\mid (p+1)$. Thus $q=3$ and $2^{n-1} + 1=3^t$, so $3^t-1=2^{n-1}$.
Suppose that $t$ is even. Then $t=2t_1$ for some positive integer $t_1$, and consequently both $(3^{t_1}-1)$ and $(3^{t_1}+1)$ are powers of $2$, which implies that $(3^{t_1}-1)\mid (3^{t_1}+1)$. Now that $(3^{t_1}-1) \mid (3^{t_1}+1-(3^{t_1}-1))=2$, we have $(3^{t_1}-1)=2$ and $t_1=1$. Thus $p=2\cdot 3^2-1=17$, as desired.
Next assume that $t$ is odd. Then $(3^{t-1}+3^{t-2}+\cdots+1)$ is odd. Since  $3^t-1=(3-1)(3^{t-1}+3^{t-2}+\cdots+1)=2^{n-1}$ is a power of $2$, it follows that the odd factor  $(3^{t-1}+3^{t-2}+\cdots+1)$ must equal $1$.  Therefore $t-1=0$, i.e., $t=1$. Then $p=2\cdot 3-1=5$, as desired.
\qed

\section{Examples}\label{examples}
In this chapter, we will present several examples that demonstrate a close relationship between the structure of a group and the number of conjugacy classes of its subgroup perfect codes. Here, we always let $G$ be a finite group and  $\Delta(G)$ be the set of conjugacy classes of nontrivial subgroup perfect codes of $G$.

\begin{example}\label{exm:sol}
Let $G=\ZZ_{p^n}$ with $p\geqslant 3$ a prime and $n$ a positive integer. Then $|\Delta(G)|=n-1$.
\end{example}

\begin{example}\label{exm:4}
Let $G=\SL(2,5)$.
Since a Sylow $2$-subgroup of $G$ is isomorphic to $\Q_8$, Lemma \ref{cg} implies that every subgroup perfect code of $G$ either has odd order or odd index. Let $H$ be a nontrivial subgroup perfect code of $G$.
Suppose that $H$ has odd order. Since $H\cong H/(H\cap \Z(G))\cong H\Z(G)/\Z(G)\leqslant G/\Z(G)\cong\PSL(2,5)$, by Proposition \ref{bray}, one yields that $|H|\div 60$  and $G/\Z(G)$ contains no subgroup of order 15, and hence $H \cong \ZZ_3$ or $\ZZ_5$, i.e., $H$ is a Sylow $3$-subgroup or Sylow $5$-subgroup of $G$.
Now, assume that $H$ has odd index in $G$. Since $\Z(G)\cong \ZZ_2$, we have that $\Z(G)\leqslant H$ and $H/\Z(G)\leqslant G/\Z(G)\cong\PSL(2,5)$, and so $60/|H/Z(G)|$ is odd. The Atlas \cite[p.2]{atlas} shows that $H/\Z(G)$ is isomorphic to $\A_4$ or $\ZZ_2^2$, both of which has only one conjugacy class in $\PSL(2, 5)$.
In particular, if $H/\Z(G)\cong \ZZ_2^2$, then $H$ is a Sylow $2$-subgroup of $G$. In the following, we assume that $H/\Z(G)\cong \A_4$. Let $H_1\le G$ be another subgroup of odd order that is not a Sylow $2$-subgroup of $G$.
According to the foregoing discussion, one yields that $H_1/\Z(G)$ is conjugate to $H/\Z(G)\cong\A_4$. So $H_1$ is conjugate to $H$.
Based on the previous discussion, we conclude that $|\Delta(G)|=4$.
\end{example}

\begin{lemma}\label{exm:5}
Let $G=\SL(2,p)$, where $p>5$ is a prime such that $p-1$ is a power of 2 and $(p+1)/2=3r$ where $r>3$ is a prime. Then $|\Delta(G)|=5$.
\end{lemma}
\demo Taking $p = 257$, we have that $257-1 = 2^8$ and $257+1 =2 \cdot 3\cdot 43$, where 43 is a  prime. Thus such a prime $p$ exists.
We first show that every maximal subgroup $A$ of $G$ is the preimage of a maximal subgroup of $\PSL(2,p)$. Note that $G$ is perfect and $\Z(G)\cong \ZZ_2$. We claim that any maximal subgroup $A$ of $G$ must contain $\Z(G)$. Otherwise, the maximality of $A$ implies that $G = A \times \Z(G)$, and consequently $G' \le A$, contradicting the fact that $G' = G$. Therefore, $A/\Z(G)$ is maximal in $G/\Z(G) \cong \PSL(2,p)$, as claimed.

Let $P$ be a Sylow $2$-subgroup of $G$. By \cite[p.143]{Carter}, $P$ is isomorphic to a generalized quaternion $2$-group. By Lemma \ref{cg}, every subgroup perfect code of $G$ has odd order or odd index.
Since $(p+1)/2=3r$ and prime $r>3$, it follows that $p-1 > 2^4$. Consequently $|G| / |2.\S_4|=p(p+1)(p-1)/48$ is even, and by Proposition \ref{bray}, the only proper subgroups of $G$ of odd index are its Sylow $2$-subgroups, which are only one conjugacy class.

Assume that $|H|$ is odd. Since $|G|=p(p-1)(p+1)=3pr|P|$ and $G/\Z(G)\cong \PSL(2,p)$, Proposition \ref{bray} implies that $|H|\in \{3, p, r, 3r\}$. If $|H|\in\{3, p, r\}$, then $H$ is a Sylow subgroup of $G$.

Suppose that $|H|=3r$. By Proposition \ref{bray}, $H$ is contained in a maximal subgroup $M$ of $G$ such that  $M/\Z(G)\cong \D_{p+1}$.  Let $H_1$ be another subgroup of order $3r$ in $G$. Applying Proposition \ref{bray} again, there exists a maximal subgroup $M_1$ of $G$ such that $H_1\le M_1$ and $M_1/\Z(G)\cong \D_{p+1}$.
Since all subgroups of $G/\Z(G)$ that are isomorphic to $\D_{p+1}$ form a single conjugacy class, it follows that $M$ and $M_1$ are conjugate in $G$.
There exists an element $g\in G$ such that $M_1^g=M$.
Since $\Z(G)\cong \ZZ_2$ and $M/\Z(G)$ is solvable, it follows that $M$ is solvable.
Since both $H$ and $H_1^g$ are $\{3,r\}$-Hall subgroups of $M$, Proposition \ref{Hall} implies that $H$ and $H_1^g$ are conjugate in $M$, and consequently $(H_1^g)^x=H_1^{gx}=H$ for some $x\in M$.  Therefore, all subgroups of order $3r$ are conjugate in $G$.
In summery, $|\Delta(G)|=5$.
\qed

Using Magma\cite{magma}, we find that there exists a proper central extension of $\PGL(2,23)$ whose Sylow $2$-subgroups are isomorphic to $\Q_{32}$.

\begin{lemma}\label{exm:6}
Let $G= \ZZ_2.\PGL(2,23)$ be a proper central extension of $\PGL(2,23)$ that its Sylow $2$-subgroups are isomorphic to $\Q_{32}$. Then $|\Delta(G)|=6$.
\end{lemma}
\demo
Let $H$ be a nontrivial subgroup perfect code of $G$. Since the Sylow $2$-subgroup of $G$ is $\Q_{32}$, Lemma \ref{cg} implies that every subgroup perfect code of $G$ either has odd order or odd index. Let $Z\le \Z(G)$ such that $Z\cong\ZZ_2$ and $G/Z\cong \PGL(2,23)$.

\textbf{Claim 1.} The subgroups of $G$ of odd order fall into exactly four conjugacy classes.

Suppose that $|H|$ is odd. Then $H\cong HZ/Z\leqslant G/Z\cong \PGL(2,23)$. It follows from Proposition \ref{bray} that $|H|\in \{3,11,23,253\}$. If $|H|\in \{3,11,23\}$, then since $|G|=2^5 \cdot 3 \cdot 11 \cdot 23$, it follows that $H$ is a Sylow subgroup of $G$.

Assume that $|H|=253$. By Proposition \ref{bray}, every subgroup of $G/Z$ of order $253$ lies in a maximal subgroup of $G/Z$ that is isomorphic to $\ZZ_{23}{:}\ZZ_{22}$. By the Atlas\cite[p.15]{atlas}, the subgroups of $\PGL(2,23)$ isomorphic to $\ZZ_{23}{:}\ZZ_{22}$ form a single conjugacy class.
Let $H_1$ be another subgroup of $G$ of order $253$. Then there exist subgroups $M$ and $M_1$ of $G$ such that $MZ/Z$ and $M_1Z/Z$ are maximal subgroups of $G/Z$, both isomorphic to $\ZZ_{23}{:}\ZZ_{22}$, such that $HZ/Z\leqslant MZ/Z$ and $H_1Z/Z\leqslant M_1Z/Z$, which follows that $H\le HZ\le MZ$ and $H_1\le H_1Z\le M_1Z$ respectively.
Since $MZ/Z$ and $M_1Z/Z$ are conjugate in $G/Z$, there exists an element $g\in G$ such that $(M Z)^g=(M Z)^{Zg}=M_1 Z$. Since $Z\cong \ZZ_2$, it follows that $MZ$ is solvable and $H$ is a $\{11,23\}$-Hall subgroup of $MZ$. In particular, $H_1^{g^{-1}}$ is also a $\{11,23\}$-Hall subgroup of $MZ$. By Proposition \ref{Hall}, there exists an element $x\in MZ$ such that $(H_1^{g^{-1}})^x=H_1^{g^{-1}x}=H$, that is, $H_1$ is conjugate to $H$ in $G$. Therefore, all subgroups of $G$ of order $253$ form a single conjugacy class. Hence Claim 1 holds.

\textbf{Claim 2.} The nontrivial subgroups of odd index in $G$ fall into two conjugacy classes.

Suppose that $|G:H|$ is odd. Then $Z< H$. Let $P$ be a Sylow $2$-subgroup of $G$. Since $|P|=32$, Proposition \ref{bray} implies that $H/Z$ is either isomorphic to $P/Z$ or to $\D_{48}$. By the Atlas\cite[p.15]{atlas}, the subgroups of $\PGL(2,23)$ isomorphic to $\D_{48}$ are conjugate in $\PGL(2,23)$.
Without loss of generality, we assume that $H/Z\cong \D_{48}$.
Let $H_1$ be another subgroup of odd index of $G$.
By the arbitrariness of $H$, we get that $Z<H_1$, and $H_1/Z\cong P/Z$ or $\D_{48}$.
If $H_1/Z\cong\D_{48}$, then since all subgroups of $\PGL(2,23)$ isomorphic to $\D_{48}$ lie in a single conjugacy class, it follows that $H_1$ is conjugate to $H$.
Otherwise $H_1$ is conjugate to $P$. Hence Claim 2 holds.

Combining both Claims, we find that $|\Delta(G)|=6$ and Example \ref{exm:6} holds.
\qed

\section{The Proof of Theorem \ref{main}}\label{sec3}
In this section, we will prove Theorem \ref{main} through some technical lemmas. Here, we always let $G$ be a finite group with $\pi(G) = \{p_1, p_2, \dots, p_n\}$, where $n$ is a positive integer and $p_1< p_2< \dots<p_n$. For each $p_i \in \pi(G)$, set $P_i \in \Syl_{p_i}(G)$.

\begin{lemma}\label{gepi}
The condition $|\Delta(G)|\ge |\pi(G)|$ holds if and only if $G\ncong\ZZ_p$, $\ZZ_{2^m}$ or $\Q_{2^{m+1}}$, where $p$ is a prime and $m\geqslant 2$ is an integer.
\end{lemma}

\begin{proof}
Let $p\in \pi(G)$ and let $P$ be a Sylow $p$-subgroup of $G$. Then $P$ either has odd order or odd index. Assume first that $|\pi(G)|\ge 2$. Then $1<P<G$.
By Proposition \ref{odd-index}, we have that $[P]\in \Delta(G)$.
Since $p$ is arbitrary, we conclude that $|\Delta(G)|\ge |\pi(G)|$.

Now assume that $|\pi(G)|=1$. It suffices to show that $|\Delta(G)|=0$ if and only if $G\in \{\ZZ_p, \ZZ_{2^m}, \Q_{2^{m+1}}\}$ where $p$ is a prime and integer $m\geqslant 2$.
Suppose that $|\Delta(G)|=0$, i.e., $G$ has no nontrivial subgroup perfect code.
By Proposition \ref{chen2}, this is equivalent to $G\in \{\ZZ_p, \ZZ_{2^m}, \Q_{2^{m+1}}\}$, where $p$ is a prime and integer $m\geqslant 2$.

Therefore, the condition $|\Delta(G)|\ge |\pi(G)|$ holds if and only if $G\notin \{\ZZ_p, \ZZ_{2^m}, \Q_{2^{m+1}}\}$, where $p$ is prime and integer $m\geqslant 2$, and hence Lemma \ref{gepi} holds.
\end{proof}

Next, we will consider the conditions under which the equality sign holds. First, we consider the minimal case, namely when it is equal to 1.

\begin{lemma}\label{lemma1}
Suppose that $\pi(G) = \{p_1\}$. Then $|\Delta(G)| = 1$ if and only if $G$ is isomorphic to $\ZZ_{p_1^2}$, where $p_1$ is an odd prime.
\end{lemma}
\demo Assume that $\pi(G) = \{p_1\}$ and $|\Delta(G)|= 1$, that is $|\pi(G)|=|\Delta(G)|$. Lemma \ref{gepi} shows that $G\ncong\ZZ_{p_1}$, $\ZZ_{2^m}$ or $\Q_{2^{m+1}}$ for $m\geqslant 2$ integer.
We claim that $p_1\neq 2$.
Suppose for contradiction that $p_1=2$.
Applying Proposition \ref{kf} (a), one yields that there exists a nontrivial subgroup perfect code $H$ of $G$ that is isomorphic to a cyclic or a generalized quaternion 2-group. Further, according to Proposition \ref{kf} (b), we conclude that the conjugacy class $[H]$ contains all involutions of $G$. However, since $H$ contains a unique involution as it is a cyclic or a generalized quaternion $2$-group, we can conclude that all involutions of $G$ are conjugate.
Note that the center $\Z(G)$ is nontrivial and contains an involution, say $x$. Then the conjugacy class of $x$ is $\{x\}$, which implies that $x$ is the only involution in $G$, a contradiction. Therefore, the claim holds, i.e., $p_1$ is an odd prime.
Assume that $|G|=p_1^k$ for $k \geqslant 2$ integer. Since $|\Delta(G)|=1$, by Proposition \ref{odd-index}, one yields that $k=2$. Note that every group of order $p_1^2$ is isomorphic to either $\ZZ_{p_1^2}$ or $\ZZ_{p_1} \times \ZZ_{p_1}$. Since both direct factors of $\ZZ_{p_1} \times \ZZ_{p_1}$ are subgroup perfect codes of $\ZZ_{p_1} \times \ZZ_{p_1}$ and are not conjugate, it follows that $G \cong \ZZ_{p_1^2}$.

Conversely, suppose that $G \cong \ZZ_{p_1^2}$, where $ p_1$ is an odd prime.
Since the group $\ZZ_{p_1^2}$ contains a unique nontrivial subgroup of order $p_1$, which is a subgroup perfect code of $G$ by Proposition \ref{odd-index}. Hence $ |\Delta(G)| = 1 $. Therefore, Lemma \ref{lemma1} holds.
\qed

In what follows, we assume that $|\pi(G)| \geqslant 2$. By Proposition~\ref{odd-index}, the conjugacy class $[P_i]$ lies in $\Delta(G)$ for all $i = 1, \dots, n$. Furthermore, applying Proposition \ref{odd-index} again, every subgroup of $P_i$ (with $p_i \neq 2$) is a subgroup perfect code of $G$. It follows that if $|\Delta(G)| = |\pi(G)|$, then $|G|_{2'}$ must be square-free; that is, $|G|=p_1p_2\dots p_n$ if $p_1\neq 2$ or $|G|=2^mp_2\dots p_n$ where $m$ is a positive integer.

\begin{lemma}\label{lemma2}
Suppose that $|\pi(G)|\ge 2$ and $p_1\neq 2$. Then $|\Delta(G)|= |\pi(G)|$ if and only if $G\cong \ZZ_{p_2}{:}\ZZ_{p_1}$ where $p_1<p_2$ are odd primes.
\end{lemma}
\demo Suppose that $|\pi(G)|\ge 2$, $p_1\neq 2$ and $|\Delta(G)|= |\pi(G)|$. According to the argument in the previous paragraph, it can be concluded that $|G|=p_1p_2\dots p_n$ and $P_i \cong \ZZ_{p_i}$ for all $i= 1, \dots, n$. Since $|G|$ is odd, $G$ is solvable, so Propositions \ref{odd-index} and \ref{Hall} yield that $n = 2$, i.e., $|G|=p_1p_2$.
It follows from the assumption $p_1<p_2$ and the Sylow theorems that $P_2 \lhd G$, and hence $G \cong \ZZ_{p_2}{:}\ZZ_{p_1}$, as desired.

Conversely, suppose that $G \cong \ZZ_{p_2}{:}\ZZ_{p_1}$, where $p_1 < p_2$ are odd primes. Every nontrivial subgroup of $G$ is either a Sylow $p_1$-subgroup of $G$ or a Sylow $p_2$-subgroup of $G$.
Thus, $G$ has exactly two conjugacy classes of nontrivial subgroups. By Proposition \ref{odd-index}, both of these conjugacy classes belong to $\Delta(G)$. Therefore, $|\Delta(G)| = |\pi(G)|$ and Lemma \ref{lemma2} holds.
\qed

\begin{lemma}\label{lemma3}
Suppose that $|\pi(G)|\ge 2$ and $p_1=2$. Then $|\Delta(G)|= |\pi(G)|$ if and only if
$G\cong \ZZ_p{:}\ZZ_{2^m}$, $\ZZ_p{:}\Q_{2^m}$, $\Q_8{:}\ZZ_3$ or $\SL(2,3).\ZZ_2$, where $p=p_2$, $m$ is a positive integer, and in the case $G\cong \SL(2,3).\ZZ_2$, its Sylow $2$-subgroups are isomorphic to $\Q_{16}$.
\end{lemma}
\demo Let's prove the sufficiency of this lemma first. Assume that $G\cong \ZZ_p{:}\ZZ_{2^m}$, $\ZZ_p{:}\Q_{2^m}$, $\Q_8{:}\ZZ_3$, or  $\SL(2,3).\ZZ_2$, where $p\ge 3$ is a prime, $m$ is a positive integer and in the last case,  the Sylow $2$-subgroups of $G$ are isomorphic to $\Q_{16}$.
Since all Sylow $2$-subgroups of $G$ are cyclic or generalized quaternion, Lemma \ref{cg} implies that every subgroup perfect code of $G$ has odd order or odd index. Given the possible structures of $G$, any such subgroup must be either a Sylow $2$-subgroup or a Sylow $p_2$-subgroup of $G$ where $p_2=3,p$. Consequently, $|\Delta(G)| = |\pi(G)|$.

Next, we prove the necessity of this lemma. Assume that $p_1=2$, $|\pi(G)|\ge 2$ and $|\Delta(G)|= |\pi(G)|$. According to the previous paragraph of Lemma \ref{lemma2}, we can see that $|G|=2^mp_2\dots p_n$ and $\Delta(G) = \{[P_1], \dots, [P_n]\}$. Since $p_1=2$, $P_1$ is a Sylow 2-subgroup of $G$. Let $M$ be a maximal subgroup of $G$ containing $P_1$. Since $|G:M|$ is odd, by Proposition \ref{odd-index}, $M$ is a subgroup perfect code of $G$ and hence $[M] \in \Delta(G)$.
Given that $\Delta(G) = \{ [P_1], [P_2], \dots, [P_n] \}$, we have that $[M] = [P_1]$ as $P_1 \leq M$. Thus $P_1=M$, which is a maximal subgroup of $G$. Recall that $\Syl_2(G)$ denotes the set of all Sylow $2$-subgroups of $G$. Consider the conjugacy action of $G$ on $\Syl_2(G)$. It is clear that the kernel of this action is $\O_2(G)= \bigcap_{g \in G} P_1^g$. Let $\overline{G}=G/\O_2(G)$. Then $\overline{G}$ acts transitively on $\Syl_2(G)$, with point stabilizer $\overline{G}_{P_1} = P_1/\O_2(G)$.
Since $P_1$ is maximal in $G$, the quotient $P_1 / \O_2(G)$ is maximal in $\overline{G}$. Thus, $ \overline{G}$ is a primitive permutation group on $\Syl_2(G)$ of odd degree, with a point stabilizer $P_1 / \O_2(G)$  isomorphic to a $2$-group. The Sylow theorems shows that $|\Syl_2(G)|\mid |G|_{2'}$, combining with the assumption that $|G|_{2'}$ is square-free, $|\Syl_2(G)|$ is square-free. Applying Proposition \ref{square-free} (1), we get that either $\overline{G} \cong \ZZ_q{:}\ZZ_{2^d}$, where $q \geqslant 3$ is a prime and $d\ge 0$ is an integer satisfying $2^d \mid q-1$, or $\overline{G} \cong \PSL(2,p)$ or $\PGL(2,p)$, where $p \geqslant 5$ is a prime and either $p-1$ or $p+1$ is a power of $2$.

\noindent\textbf{(1)} Suppose that $\overline{G}\cong\ZZ_q{:}\ZZ_{2^d}$, where $q \geqslant 3$ is a prime and $d\ge 0$ is an integer satisfying $2^d \mid q-1$.
Then $p_2=q$, $\pi(G)=\{2,q\}$ and $\Delta(G)=\{[P_1], [P_2]\}$.

Assume that $d=0$. Then $G/\O_2(G)=\overline{G}\cong \ZZ_q$ for prime $q\ge 3$, and hence  $P_1=\O_2(G)\lhd G$ and $P_2\cong\ZZ_q$, so that $G=P_1{:} P_2$. By Lemma \ref{rm}, there exists a nontrivial subgroup perfect code $H$ of $G$ that is isomorphic to either a cyclic $2$-group or a generalized quaternion $2$-group. Since $[H]\in\Delta(G)=\{[P_1], [P_2]\}$ and $(|H|, |P_2|) = 1$, it follows that $[H]=[P_1]$ and so $P_1\cong\ZZ_{2^m}$ or $\Q_{2^m}$, where $m$ is a positive integer.
If $P_1 \cong \ZZ_{2^m}$ for some positive integer $m$, then Proposition \ref{aut} (1) implies that $\Aut(P_1)$ is a $2$-group, and applying the N/C-Theorem, we can conclude that $G\cong\ZZ_{p_2}\times \ZZ_{2^m}$, and of course we can also write it in the form of $G\cong\ZZ_{p_2}{:}\ZZ_{2^m}$, as desired. Now assume that $P_1\cong \Q_{2^m}$. By Proposition \ref{aut} (2), $\Aut(\Q_{2^m})$ is a $2$-group for $m \geqslant 4$, while $\Aut(\Q_8) \cong \S_4$.
Since $G=P_1{:}P_2$, $P_2/(P_2\cap \C_G(P_1))\cong P_2 \C_G(P_1)/\C_G(P_1)\leqslant G/\C_G(P_1)$, applying the N/C-Theorem, one yields that $G/\C_G(P_1)\lesssim\Aut(P_1)$ and $P_2/ (P_2\cap \C_G(P_1))$ is isomorphic to a subgroup of $\Aut(\Q_{2^m})$. If $m \geqslant 4$, then $\Aut(\Q_{2^m})$ is a 2-subgroup. Since $|P_2|$ is odd, it follows that $P_2=P_2\cap \C_G(P_1)$, i.e, $P_2\le\C_G(P_1)$ and $P_2$ centralizes $P_1$, and consequently $G\cong\ZZ_{p_2}\times\Q_{2^m}$, as desired.
Now suppose that $m=3$. Then $P_1\cong \Q_8$ and $\Aut(P_1)=\Aut(\Q_8)\cong\S_4$, in particular, $P_2/ (P_2\cap \C_G(P_1))\lesssim\S_4$. We deduce from $|P_2|$ is odd that $P_2/P_2\cap\C_G(P_1)=1$ or $P_2/P_2\cap\C_G(P_1)\cong\ZZ_3$.  If the former holds, i.e., $P_2=P_2\cap\C_G(P_1)$, then $G\cong \ZZ_{p_2}\times\Q_8$. For the latter case, that is $P_2/P_2\cap\C_G(P_1)\cong\ZZ_3$, since $P_2\cong\ZZ_q$, one yields that $P_2\cap\C_G(P_1)=1$ and $P_2 \cong \ZZ_3$. Consequently $G\cong\Q_8{:}\ZZ_3$, as desired.

Now suppose that $d\ge 1$. Then $\overline{G}\cong\ZZ_q{:}\ZZ_{2^d}$ has even order. By Lemma \ref{rm}, $G$ has a nontrivial subgroup perfect code $H$ that is a cyclic or a generalized quaternion 2-group.  Since $\Delta(G)=\{ [P_1], [P_2] \}$ and $(|H|,|P_2|)=1$, we conclude that $[H]=[P_1]$.
Therefore, $P_1$ is isomorphic to $\ZZ_{2^m}$ or $\Q_{2^{m}}$ for some positive integer $m$. Note that $\O_2(G)\le P_1$ and $P_2/(P_2\cap\C_G(\O_2(G)))\cong P_2 \C_G(\O_2(G))/\C_G(\O_2(G))\le G/\C_G(\O_2(G))$. Suppose that $\O_2(G)\cong\C_{2^l}$ for positive integer $l\le m$.
Applying the N/C-Theorem and Proposition \ref{aut} (1), we can conclude that $P_2/(P_2\cap\C_G(\O_2(G)))\lesssim G/\C_G(\O_2(G))\lesssim\Aut(\O_2(G))$ is a 2-subgroup and hence, $P_2\le \C_G(\O_2(G))$ and $\O_2(G){:}P_2=\O_2(G)\times P_2$. It follows from $\overline{G}\cong\ZZ_q{:}\ZZ_{2^d}$ that $G=\O_2(G).(\ZZ_q{:}\ZZ_{2^d})=(\O_2(G)\times\ZZ_q).\ZZ_{2^d}$.  Since $P_2$ is characteristic in $\O_2(G)\times P_2$ and $\O_2(G)\times P_2 \lhd G$, we can conclude that $P_2\lhd G$. Thus $G=P_2{:}P_1$, that is, $G\cong\ZZ_q{:}\ZZ_{2^m}$ or $G\cong \ZZ_q{:}\Q_{2^{m}}$.
Now suppose that $\O_2(G)$ is noncyclic. Then $P_1\cong \Q_{2^{m}}$ and $|P_1:\O_2(G)|=2$. Applying the N/C-Theorem and Proposition \ref{aut} (2), we obtain that $P_2/(P_2\cap\C_G(\O_2(G)))\lesssim G/\C_G(\O_2(G))\lesssim\Aut(\O_2(G))$, which is isomorphic to either a 2-subgroup or $\S_4$ for $m=4$. It follows from $P_2\cong\ZZ_q$ for $q\geqslant 3$ prime that either $P_2\leqslant\C_G(\O_2(G))$ or $P_2\cap\C_G(\O_2(G))=1$. For the former case, that is $P_2\le \C_G(\O_2(G))$, then  $\O_2(G){:}P_2=\O_2(G)\times P_2$ and since $P_2$ is characteristic in $\O_2(G)\times P_2$ and $\O_2(G)\times P_2 \lhd G$, one yields that $P_2\lhd G$, and consequently $G=P_2{:}P_1\cong\ZZ_q{:}\ZZ_{2^m}$ or $G\cong \ZZ_q{:}\Q_{2^{m}}$, as desired. In what follows, we assume that the latter case holds, i.e., $P_2\cap \C_{G}(\O_2(G))=1$. Then $\O_2(G)\cong \Q_8$ with $q=3$, and $P_2\ntriangleleft \O_2(G){:}P_2$. Consequently $\O_2(G){:}P_2\cong \SL(2,3)$.
Since $|P_1:\O_2(G)|=2$ and $G=P_1{:}P_2$, one yields that $|G:(\O_2(G){:}P_2)|=2$, $G \cong \SL(2,3).\ZZ_2$ and $P_1\cong \Q_{16}$. Therefore, $G\cong \ZZ_q{:}\ZZ_{2^m}, \ZZ_q{:}\Q_{2^m}$, or $\SL(2,3).\ZZ_2$, where $q\ge 3$ is a prime.

In conclusion, $G\cong \ZZ_p{:}\ZZ_{2^m}$, $\ZZ_p{:}\Q_{2^m}$, $\Q_8{:}\ZZ_3$, or $\SL(2,3).\ZZ_2$, where $p=p_2$, $m$ is a positive integer, and in the case $G\cong \SL(2,3).\ZZ_2$, its Sylow $2$-subgroup is isomorphic to $\Q_{16}$, as desired.

\noindent\textbf{(2)} Assume that $\overline{G}\cong \PSL(2,p)$ or $\PGL(2,p)$, where $p \ge 5$ is a prime such that one of $p-1$ and $p+1$ is a power of $2$. Next, we will prove our conclusion by analyzing the situation that $p-1$ is a power of 2 or $p+1$ is a power of 2.

First, suppose that $p-1=2^k$ for some integer $k\ge 2$.
Then $p + 1 = 2(2^{k-1} + 1)$. By Proposition \ref{bray}, the group $\PGL(2,p)$ contains a maximal subgroup isomorphic to $\D_{2(p+1)}$, while $\PSL(2,p)$ contains a maximal subgroup isomorphic to $\D_{p+1}$. Consequently, $\overline{G}$ contains a cyclic subgroup of odd order $ (p+1)/2 = 2^{k-1} + 1$, that is, there exists a subgroup $H$ of $G$ containing $ \O_2(G)$ such that $H/\O_2(G)\cong \ZZ_{(p+1)/2}$ ($H$ is a preimage of $\ZZ_{(p+1)/2}$ in $G$). Since $\O_2(G)$ is a 2-group and $ (p+1)/2$ is odd, it follows that $H$ contains a subgroup $Q$ such that $Q\cong \ZZ_{(p+1)/2}$. Proposition \ref{odd-index} implies that $[Q] \in \Delta(G)$.
Recall that $\Delta(G) = \{ [P_1], [P_2], \dots, [P_n] \}$, where $P_1\in \Syl_2(G)$ and each $P_i$ for $i \geqslant 2$ is isomorphic to $\ZZ_{p_i}$. There exists an integer $i$ such that $[Q]=[P_i]$, which means that $Q \cong \ZZ_{p_i}$. Consequently $2^{k-1} + 1=(p+1)/2=p_i$ is a prime. On the other hand, since $p=2^k+1$ is prime, Lemma \ref{num} yields that $k=2$, and hence $p=5$. Thus, $\overline{G} \cong \PSL(2,5)$ or $\overline{G} \cong \PGL(2,5)$.
Suppose that $\overline{G} \cong \PGL(2,5)$. Then $|\overline{G}|=120=2^3\cdot3\cdot5$ and $\Delta(G)=\{ [P_1], [P_2], [P_3] \}$. By Proposition \ref{bray}, $\overline{G}$ has a maximal subgroup which is isomorphic to $\S_4$. Thus there exists a subgroup $R$ of $G$ containing $\O_2(G)$ such that $R / \O_2(G) \cong \S_4$. Since $|\overline{G}|=120$, we have that $|G : R|$ is odd.
By Proposition \ref{odd-index}, we get that $[R] \in \Delta(G)$, which implies that $R\cong P_i\cong \ZZ_{p_i}$ for some $i\in \{1,2,3\}$. This contradicts the assumption that $R / \O_2(G) \cong \S_4$.
Now suppose that $\overline{G} \cong \PSL(2,5)$. By the Atlas\cite[p.2]{atlas}, the maximal subgroups of $ \overline{G}$ (up to conjugacy) are isomorphic to $\S_3$, $\D_{10}$ and $\A_4$, none of which is a $2$-group. However, the Sylow $2$-subgroup $P_1 / \O_2(G)$ is maximal in $\overline{G}$,  a contradiction.

Assume that $p+1=2^k$ for some integer $k\ge 3$.
Then $p - 1 = 2(2^{k-1} - 1)$. Since $k \geqslant 3$, we have that $2^{k-1} - 1 > 1$. By  Proposition \ref{bray}, the group $\PGL(2,p)$ contains a maximal subgroup which is isomorphic to $ \ZZ_p{:} \ZZ_{p-1}$, while $\PSL(2,p)$ contains one isomorphic to $\ZZ_p{:}\ZZ_{(p-1)/2}$. Consequently, there exists a subgroup $H$ of $G$ containing $\O_2(G)$ such that $H / \O_2(G) \cong \ZZ_p {:}\ZZ_{(p-1)/2}$.
Since $H/\O_2(G)$ has odd order, it follows that $H$ is solvable. By Proposition \ref{Hall}, $H$ contains a subgroup $Q$ with $Q \cong \ZZ_p{:}\ZZ_{(p-1)/2}$.
Proposition \ref{odd-index} implies that $[Q] \in \Delta(G)$, yielding that $Q\cong P_i$ for some $i\in \{1,\dots,n\}$. This contradicts the assumption that $Q \cong \ZZ_p{:}\ZZ_{(p-1)/2}$.

Therefore, $\overline{G}$ is isomorphic to neither $\PSL(2,p)$ nor $\PGL(2,p)$. Finally, we proved the correctness of Lemma \ref{lemma3}.\qed

With the results currently available, we now prove Theorem \ref{main}.

{\bf Proof of Theorem \ref{main}.}
Let $G$ be a finite group. By Lemma \ref{gepi}, we have that $|\Delta(G)| \ge |\pi(G)|$ unless $G$ is isomorphic to one of the groups $\ZZ_p$, $\ZZ_{2^m}$, or $\Q_{2^{m+1}}$, where $p$ is a prime and $m \ge 2$.
Combining Lemmas \ref{lemma1}, \ref{lemma2}, and \ref{lemma3}, we conclude that $|\Delta(G)| = |\pi(G)|$ if and only if $G$ is isomorphic to one of the groups listed in Table \ref{tab:m=pi}. So we proved the correctness of Theorem \ref{main}.
\qed

The proof technique of Theorem \ref{main} also enables us to establish the following result concerning groups that have at most three conjugacy classes of nontrivial subgroup perfect codes.

\begin{corollary}\label{cor}
Let $G$ be a finite group with at most $3$ conjugacy classes of nontrivial subgroup perfect codes. Then $G$ is solvable and $|\pi(G)|\leq 2$.
\end{corollary}
\proof Assume that $G$ is a finite group with $|\Delta(G)|\le 3$.
Suppose for contradiction that $G$ is non-solvable. Then $ |\pi(G)| \ge 3$.
By Proposition \ref{odd-index}, $G$ has at least $|\pi(G)|$ non-conjugate Sylow subgroups that are subgroup perfect codes of $G$, which implies that $|\Delta(G)|\ge |\pi(G)|$.
Since $|\pi(G)| \ge 3$ and $|\Delta(G)|\le 3$, we get that $|\Delta(G)|=|\pi(G)|$.
By Theorem \ref{main}, we conclude that $\pi(G)\leq 2$, a contradiction. Therefore, $G$ is solvable. Moreover, by Propositions \ref{odd-index} and \ref{Hall}, we can conclude that $ |\pi(G)| \leq 2$. \qed

\section{The proof of Theorem \ref{main2}}\label{sec4}

In this section, let $G$ be a finite group with $\pi(G) = \{p_1, p_2, \dots, p_n\}$, where $p_1< p_2< \dots<p_n$, and $\Delta(G)$ the set of all conjugacy classes of nontrivial subgroup perfect codes of $G$. For each $p_i \in \pi(G)$, set $P_i \in \Syl_{p_i}(G)$.

We now establish several preliminary results and technical lemmas before proceeding to the proof of Theorem \ref{main2}.

\begin{proposition}\cite[p.160, Corollary 5.14]{isaacs}\label{pb}
Let $p$ be the smallest prime divisor of $|G|$, and let $P \in \Syl_p(G)$.
If $P$ is cyclic, then $G$ has a normal $p$-complement.
\end{proposition}

The following proposition can be read off from \cite[Theorem 1]{ls}(for details, see Appendix \ref{Appendix}).

\begin{proposition}\label{pi+1}
Let $G$ be a primitive permutation group on a finite set $\Omega$ of odd degree, and let
$|G| = 2^{a_1} p_{2}^{a_{2}} \cdots p_{n}^{a_{n}}$, where each $a_i \ge 1$ and $a_{2}+\dots+a_{n} \le n$. Let $\alpha \in \Omega$. Suppose that $\G_{\alpha}$ is a $2$-group. Then one of the following holds:
\begin{enumerate} [font=\normalfont]

            \item $G\cong \ZZ_p {:} \G_{\alpha}$ or $G\cong \ZZ_p^2 {:} \G_{\alpha}$, where prime $p \geqslant 3$;

            \item $G\cong \PGL(2,9), \M_{10}, \PGammaL(2,9)$;

            \item $G\cong \PSL(2, p)$ or $\PGL(2,p)$, where prime $p\ge 5$, one of $p-1$ or $p+1$ is a power of $2$.
\end{enumerate}
\end{proposition}

Now we consider an interesting proposition on generalized quaternion $2$-group. Let $G\cong Q_{2^n}$ be a generalized quaternion $2$-group of order $2^n$ for $n \geqslant 3$. Then $G=\langle x, y \mid x^{2^{n-1}}=1,\, x^{2^{n-2}}=y^2,\, y^{-1}xy = x^{-1} \rangle$.
Suppose that $K$ is a nontrivial normal subgroup of $G$ such that $G/K$ is noncyclic. From the structure of group $G$, we can deduce that $K$ is cyclic or generalized quaternion. For the latter case, i.e., $K$ is a generalized quaternion 2-group, we conclude that $K$ can be expressed in the following form: $K=\lg x^iy, x^k y\rg=\lg x^{k-i}, x^k y\rg$ with some integers $i$ and $k$. The generating relation of group $G$ shows that $G=\langle x\rangle\langle x^ky\rangle$ and $g= x^j (x^k y)^l$ for all $g\in G$ and some integers $j$ and $l$. Then $gK=x^j (x^k y)^l K=x^j K$, which implies that $G/K=\lg x \rg K/K\cong \lg x \rg/(\lg x \rg\cap K)$ is cyclic, a contradiction. Now that $K$ is cyclic. Summing up the above discussion, we can draw the following small observation.

\begin{observation}\label{observation}
If $G\cong Q_{2^n}$ for $n \geqslant 3$ and $1\neq K\lhd G$ such that $G/K$ is noncyclic, then $K$ is cyclic.
\end{observation}

\begin{lemma}\label{non-solvable}
Suppose that $|\Delta(G)| = |\pi(G)| + 1$. Then $G$ is non-solvable if and only if the following holds:
    \begin{enumerate}[label=(\roman*), font=\normalfont]
        \item $P_1 \cong \Q_{2^{a_1}}$ for some integer $a_1\ge 3$ and $\O_2(G)>1$ is a cyclic subgroup;
        \item $G/\O_2(G) \in \{\PSL(2,5),\; \PGL(2,5),\; \PGL(2,7),\; \PGL(2,9),\; \PSL(2,17),\; \PGL(2,17), \; \allowbreak \PSL(2,p)$,\; $\PGL(2,p)\}$, where $p$ is a Fermat prime satisfying $p+1 = 2p_2 p_3$ for primes $p_2 < p_3 < p$.
    \end{enumerate}
\end{lemma}
\demo
\textbf{(I).} First, let's prove the necessity. Suppose that $|\Delta(G)| = |\pi(G)| + 1$ and $G$ is non-solvable. Then $|\pi(G)|\geqslant3$. Let $|G|=p_1^{a_1} p_2^{a_2} \dots p_n^{a_n}$ and let $P_i\in \Syl_{p_i}(G)$, where $p_1<p_2<\dots<p_n$ are primes, $a_i\ge 1$ and $n\geqslant3$ are integers. Since $G$ is non-solvable, it follows that $p_1=2$.
By Proposition \ref{odd-index}, we can conclude that $\Delta(G)\supset\{[P_1], [P_2], \dots, [P_n]\}$ and  $a_2+\dots+a_n\le n$ as $|\Delta(G)|=|\pi(G)|+1$.
Let $M$ be a maximal subgroup of $G$ such that $P_1 \le M$. Applying Proposition \ref{odd-index} again, we conclude that $[M]\in \Delta(G)$.
We claim that $P_1=M$ or $P_1$ is a maximal subgroup of $M$.
Assume for contradiction that there exists a maximal subgroup $M_1$ of $M$ such that $P_1<M_1<M$. Proposition \ref{odd-index} implies that $[M_1]\in \Delta(G)$ and consequently $$\Delta(G)\supseteq \{[P_1], [P_2], \dots, [P_n], [M_1], [M]\},$$
contradicting the assumption that $|\Delta(G)|=|\pi(G)|+1$. Thus $P_1=M$ or $P_1$ is a maximal subgroup of $M$, and hence the claim holds. Let $K = \bigcap_{g \in G} M^g$.

{\textbf{Claim 1.} If $P_1$ is a maximal subgroup of $M$, then $G/K\in \{\PSL(2,5), \PGL(2,5)\}$, $P_1\cong \Q_{2^{a_1}}$ with $a_1\ge 3$, and $K=\O_2(G)>1$ is cyclic.

Suppose that $P_1$ is a maximal subgroup of $M$. Since $|\Delta(G)|=|\pi(G)|+1$ and $[M]\in \Delta(G)$, we can conclude that
$$ \Delta(G)=\{[P_1], [P_2], \dots, [P_n], [M]\}.$$
Proposition \ref{odd-index} implies that $a_2+\dots+a_n= n-1$ and hence $|G|_{2'}$ is square-free, in particular, $|G:M|$ is odd and square-free. It follows from Lemma \ref{rm} that $P_1$ is a cyclic or a generalized quaternion 2-group. If $P_1$ is cyclic, then Proposition \ref{pb} gives that $G$ has a normal $2$-complement with odd order and so $G$ is solvable contradicting the insolvability of $G$. Thus $P_1\cong \Q_{2^{a_1}}$ for $a_1\ge 3$. By Lemma \ref{cg}, every nontrivial subgroup perfect code of $M$ has odd order or odd index. Proposition \ref{odd-index} shows that these subgroups are also subgroup perfect codes of $G$, which forces that $|\Delta(M)|=|\pi(M)|$.
Theorem \ref{main} implies that $M$ is solvable and satisfies that $|M|_{2'}$ is a prime.

Noting that $K = \bigcap_{g \in G} M^g$, one yields that $K$ is solvable as $M$ is solvable.
We claim that $K$ is a 2-group and $K=\O_2(G)$. Assume for a contradiction that $K$ is not a 2-group. Then $|K|$ has an odd prime divisor, namely $p_i$ where $i\in\{2,\dots,n\}$. Since $|\pi(G)|\ge 3$ as $G$ is non-solvable, there exists another odd prime $p_j\in \pi(G)$ such that $p_j\ne p_i$. Given that $K\lhd G$, $K$ is solvable and $P_j$ is cyclic, we obtain that $K P_j$ is a solvable subgroup of $G$. By Proposition \ref{Hall}, $K P_j$ contains a $\{p_i, p_j\}$-Hall subgroup, say $H$. Then Proposition \ref{odd-index} implies that $[H]\in \Delta(G)$ and hence $|\Delta(G)|\ge |\pi(G)|+2$, contradicting the assumption that $|\Delta(G)|= |\pi(G)|+1$.
Therefore, $K$ is a normal 2-group and so $K\le \O_2(G)=\bigcap_{g\in G} P_1^g$. Moreover, since $P_1\le M$ and $K=\bigcap_{g\in G} M^g$, we conclude that $K=\O_2(G)$, as desired.

Consider the right multiplication action of $G$ on the set $\Omega := \{ Mg \mid g \in G \}$. The kernel of this action is $K$ and the stabilizer group $G_x\cong M$ for each $x\in \Omega$. We see that $G/K$ is a primitive permutation group of odd degree. Let $L/K = \soc(G/K)$ and $x\in\Omega$. Then $L\unlhd G$ and $(G/K)_x=G_x/K\cong M/K$. Since $|G|_{2'}$ is square-free and $|M|_{2'}$ is a prime, Proposition \ref{square-free} (2) implies that one of the following holds:
\begin{enumerate}[font=\normalfont]
    \item $G/K\cong \ZZ_p{:}\ZZ_d$, where prime $p \geq 3$ and positive integer $d\mid (p-1)$;
    \item $L/K\cong {^2\!B}_2(q)$, where $q=2^{2f+1}$ and integer $f\ge 1$;
    \item $L/K\cong \PSL(2, p)$, where $p \geq 5$ is a prime.
\end{enumerate}
It follows from $G$ is non-solvable, $\ZZ_p{:}\ZZ_d$ is solvable, and $K$ a 2-group that $G/K\ncong \ZZ_p{:}\ZZ_d$.

Now assume that $L/K\cong {^2\!B}_2(q)$, where $q=2^{2f+1}$ and integer $f\ge 1$. Since $L\lhd G$, $K=\O_2(G)$ and $P_1\in\Syl_2(G)$, one yields that $K\leqslant P_1$ and $(P_1\cap L)/K$ is a Sylow $2$-subgroup of $L/K\cong {^2\!B}_2(q)$. Since $P_1$ is a generalized quaternion $2$-group, $(P_1\cap L)/K$ is a cyclic, dihedral or generalized quaternion $2$-group, which contradicts the fact that the Sylow $2$-subgroups of ${^2\!B}_2(q)$ are isomorphic to $E_q.E_q$. Therefore, this case does not occur.

Suppose that $L/K\cong \PSL(2,p)$, where $p\geqslant 5$ is an odd prime.
Noting that $K$ is a 2-subgroup of $G$ and $L/K=\soc(G/K)$, we yield that $|G|_{2'}=p(p-1)_{2'}(p+1)_{2'}$. It deduces from $|\pi(G)|\ge 3$ that one of $p-1$ or $p+1$ has an odd prime divisor.
Assume that $p-1$ has an odd prime divisor, say $p_i$ for $i\in\{2,3,\dots,n\}$. Since $L/K\cong \PSL(2,p)$, it follows from Proposition \ref{bray} that $L$ has a subgroup $M_1$ containing $K$ with $M_1/K \cong \ZZ_p{:}\ZZ_{(p-1)/2}$. Since $K$ is a 2-group and $\ZZ_p{:}\ZZ_{(p-1)/2}$ is solvable, it follows that $M_1$ is solvable.  Proposition \ref{Hall} implies that $M_1$ contains a $\{p_i, p\}$-Hall subgroup $H$. By Proposition \ref{odd-index}, we conclude that $[H]\in \Delta(G)$ and hence $|\Delta(G)|\ge |\pi(G)|+2$, which contradicts the assumption that $|\Delta(G)|=|\pi(G)|+1$.
Therefore, $p-1$ is a power of 2. Since $|\pi(G)|\ge 3$ and $|G|_{2'}=p(p-1)_{2'}(p+1)_{2'}$, it follows that $p+1$ is not a power of 2.
Suppose that $p+1$ has at least two distinct odd prime divisors, namely $p_i$ and $p_j$ for $i,j\in \{2,3,\cdots,n\}$.
By Proposition \ref{bray}, $L$ has a subgroup $M_2$ containing $K$ with $M_2/K \cong \D_{p+1}$. Since $K$ is a 2-group and $\D_{p+1}$ is solvable, it follows that $M_2$ is solvable.
Proposition \ref{Hall} implies that $M_2$ contains a $\{p_i, p_j\}$-Hall subgroup $H$. By Proposition \ref{odd-index}, we conclude that $[H]\in \Delta(G)$ and hence $|\Delta(G)|\ge |\pi(G)|+2$, which contradicts the assumption that $|\Delta(G)|=|\pi(G)|+1$. Therefore $p+1$ has only one odd prime divisor.
However $p-1\geqslant 4$ is 2-power shows that $4\div(p-1)$ and $p+1=2(p+1)_{2'}$.  Since $|G|_{2'}$ is square-free, we conclude that $(p+1)_{2'}$ is a prime, and by Lemma \ref{num3} (2), we have that $p=5$. Consequently, $G/K\in \{\PSL(2,5), \PGL(2,5)\}$.
In particular, since $K = \O_2(G)$ and $P_1$ is a Sylow $2$-subgroup of $G$, $P_1/K$ is a Sylow $2$-subgroup of $G/K$. Note that $P_1 \cong \Q_{2^{a_1}}$ and the Sylow $2$-subgroups of $\PSL(2,5)$ and $\PGL(2,5)$ are isomorphic to $\ZZ_2^2$ and $\D_8$, respectively. By Observation \ref{observation}, we conclude that $K$ is a nontrivial cyclic group. It follows from $G/K \in \{\PSL(2,5), \PGL(2,5)\}$, $P_1 \cong \Q_{2^{a_1}}$ with $a_1 \ge 3$ and $K = O_2(G) > 1$ is cyclic that Claim 1 holds.

{\textbf{Claim 2.} If $P_1=M$ is a maximal subgroup of $G$, then $G/K\in \{\PGL(2,7)$, $\PGL(2,9)$, $\PSL(2,17)$, $\PGL(2,17)$, $\allowbreak\PSL(2,p)$, $\allowbreak \PGL(2,p)\}$, where $K=\O_2(G)>1$ is cyclic, $P_1\cong \Q_{2^{a_1}}$, $p$ is a Fermat prime such that $p+1 = 2p_2p_3$.

Suppose that $P_1=M$ is a maximal subgroup of $G$.
Consider the right multiplication action of $G$ on the set $\Omega := \{ P_1^g \mid g \in G \}$. Then the kernel of this action is $K = \O_2(G)$ and $G/K$ is a primitive permutation group of odd degree $|\Omega|$, in particular, for every $ x \in \Omega$, $G_x\cong P_1$ is a 2-group and $(G/K)_x=G_x/K$. By Proposition \ref{pi+1}, one of the following holds:

\begin{enumerate}[label=(\alph*), font=\normalfont]
    \item $G/K\cong \ZZ_p {:} (P_1/K)$ or $\ZZ_p^2 {:} (P_1/K)$, where prime $p \geqslant 3$;

   \item $G/K\cong \PGL(2,9), \M_{10}, \PGammaL(2,9)$;

   \item $G/K\cong \PSL(2, p)$ or $\PGL(2, p)$, where prime $p\ge 5$, and $p-1$ or $p+1$ is a power of $2$.
\end{enumerate}

Since $G$ is non-solvable, $K$ and $P_1$ are 2-groups, one yields that $G/K\ncong \ZZ_p{:}(P_1/K)$ and $G/K\ncong \ZZ_p^2{:}(P_1/K)$, that is, the case (a) cannot occur. Assume that the case (b) holds, that is $G/K\cong \PGL(2,9)$, $\M_{10}$, or $\PGammaL(2,9)$. Since $9 \mid |G|$, we get that $G$ contains a subgroup $H$ isomorphic to $\ZZ_3$ that is not a Sylow $3$-subgroup of $G$.
By Proposition \ref{odd-index} and the assumption that $|\Delta(G)|=|\pi(G)|+1$, we have that
$\Delta(G)= \{[P_1], [P_2], \ldots, [P_n], [H]\}$.
Lemma \ref{rm} implies that $P_1\cong \ZZ_{2^{a_1}}$ or $\Q_{2^{a_1}}$.
Since $G$ is non-solvable, Proposition \ref{pb} yields that $P_1\cong \Q_{2^{a_1}}$, which implies that $P_1/K$ is cyclic, dihedral or generalized quaternion.
The Sylow 2-subgroups of $\M_{10}$ and $\mathrm{P\Gamma L}(2,9)$ are isomorphic to $\mathrm{QD}_{16}$ and $\ZZ_8{:}(\ZZ_2\times \ZZ_2)$ respectively. Hence $G/K$ is neither isomorphic to $\M_{10}$ nor to $\PGammaL(2,9)$, and so $G/K\cong \PGL(2,9)$.
Given that $K=\O_2(G)$ and a Sylow $2$-subgroup of $\PGL(2,9)$ is $\D_{16}\cong P_1/K$, applying  Observation \ref{observation}, we yield that $K$ is a nontrivial cyclic group. Therefore, $G/K\cong \PGL(2,9)$, $P_1\cong \Q_{2^{a_1}}$ and $K=\O_2(G)>1$ is cyclic.

Finally, we assume that the case (c) holds, i.e., $G/K\cong \PSL(2, p)$ or $G/K\cong \PGL(2, p)$, where prime $p\ge 5$, and $p-1$ or $p+1$ is a power of $2$.
We assert that $3\leqslant|\pi(G)|\leqslant 4$. Suppose for contradiction that $|\pi(G)|\ge 5$. Since one of $p-1$ and $p+1$ is a power of $2$, without loss of generality, we assume that $p-1$ is a power of 2. Then $|G|_{2'}=p(p-1)_{2'}(p+1)_{2'}=p(p+1)_{2'}$. Combining with $|\pi(G)|\ge 5$, we conclude that $p=p_n$ and $p_2p_3p_4 \mid (p+1)$.
Since $\soc(G/K)\cong \PSL(2,p)$, Proposition \ref{bray} implies that $G$ has a subgroup $M_1$ containing $K$ such that $M_1/K \cong \D_{p+1}$. Since $K=\O_2(G)$ is a 2-group and $\D_{p+1}$ is solvable, it follows that $M_1$ is solvable.  By Proposition \ref{Hall}, for any distinct $i, j \in \{2,3,4\}$, the group $M_1$ contains a $\{p_i, p_j\}$-Hall subgroup $H_{ij}$. Proposition \ref{odd-index} implies that $[H_{ij}]\in \Delta(G)$ and consequently $\Delta(G)\supseteq \{[P_1], [P_2], \dots, [P_n], [H_{23}], [H_{24}], [H_{34}]\}$.
This contradicts the assumption that $|\Delta(G)|=|\pi(G)|+1$ as $\pi(G)=\{p_1, p_2, \cdots, p_n\}$. Hence the assertion holds, that is $3\le |\pi(G)|\le 4$. Next, let us analyze the cases $|\pi(G)| = 3$ and $|\pi(G)| = 4$ one by one.

Suppose that $|\pi(G)|=3$. Since $|G|_{2'}=p(p-1)_{2'}(p+1)_{2'}$, we have that $p_3=p$.
Assume first that $p_2 \mid (p-1)$. The assumption that $p-1$ or $p+1$ is a power of 2 implies that  $p+1$ is a power of 2.
Since $|\pi(G)|=3$, it follows that $\frac{p-1}{2}$ is a power of $p_2$.  Lemma \ref{num3} (1) implies that $p=7$, and consequently $G/K\cong \PSL(2,7)$ or $\PGL(2,7)$.
By Proposition \ref{bray}, $G$ has a subgroup $M_1$ containing $K$ such that $M_1/K \cong \ZZ_7 {:} \ZZ_3$. Since $M_1/K$ is solvable and $K$ is a 2-group, it follows that $M_1$ is solvable. Applying Propositions \ref{Hall} and \ref{odd-index}, one yields that $M_1$ has a $\{3, 7\}$-Hall subgroup $H$ such that $[H] \in \Delta(G)$, and therefore, $\Delta(G) = \{[P_1], [P_2], \dots, [P_n], [H]\}$.
Applying Lemma \ref{rm}, we have that $P_1\cong \ZZ_{2^{a_1}}$ or $\Q_{2^{a_1}}$. Since $G$ is non-solvable, it follows from Proposition \ref{pb} that $P_1\cong \Q_{2^{a_1}}$.
Moreover, since the Sylow 2-subgroups of $\PGL(2,7)$ are known to be a dihedral, by Observation \ref{observation}, we have that $K$ is a nontrivial cyclic group.
By Proposition \ref{bray}, $\PSL(2,7)$ has no Sylow 2-subgroup that is maximal. Deduced from $P_1/K$ is a maximal subgroup of $G/K$, we can draw the conclusion that $G/K\ncong \PSL(2,7)$, and so $G/K\cong \PGL(2,7)$.

Now assume that $p_2 \mid (p+1)$. Then  $p-1$ is a power of 2. Since $|\pi(G)|=3$, it follows that $(p+1)/2$ is a power of $p_2$. By Lemma \ref{num3} (2), we obtain that $p = 5$ or $17$. Since neither group $\PSL(2,5)$ nor $\PGL(2,5)$ has a maximal subgroup isomorphic to a 2-group (see the Atlas\cite[p.2]{atlas} for example), and $P_1/K$ is a maximal subgroup of $G/K$, we see that $p\neq 5$ and $L/K\ncong \PSL(2,5)$.
Thus $p=17$ and $G/K\cong \PSL(2,17)$ or $\PGL(2,17)$. It follows that $9\mid |G|$, and hence $G$ contains a subgroup $H$ isomorphic to $\ZZ_3$ that is not a Sylow $3$-subgroup of $G$. Applying Proposition \ref{odd-index}, one yields that $[H]\in \Delta(G)$ and hence $\Delta(G)=\{[P_1], [P_2], \dots, [P_n], [H]\}$. Further, by Lemma \ref{rm}, one yields that $P_1\cong \ZZ_{2^{a_1}}$ or $\Q_{2^{a_1}}$.
It deduces from the assumption that $G$ is non-solvable and Proposition \ref{pb} that $P_1\cong \Q_{2^{a_1}}$. On the other hand, since the Sylow 2-subgroups of $\PSL(2,17)$ and $\PGL(2,17)$ are known to be dihedral, applying Observation \ref{observation}, we conclude that $K$ must be a nontrivial cyclic 2-group.

Next, we suppose that $|\pi(G)|=4$.
Since $|G|_{2'}=p(p-1)_{2'}(p+1)_{2'}$ and $K$ is a 2-group, we conclude that $p_4=p$.
Note that $p-1$ or $p+1$ is a power of 2. Assume that $p+1$ is a power of 2. Then $p_2p_3\mid (p-1)$. Since $\soc(G/K)\cong \PSL(2,p)$, Proposition \ref{bray} implies that $G$ has a subgroup $M_1$ containing $K$ such that  $M_1/K\cong \ZZ_p{:}\ZZ_{(p-1)/2}$. Since $K$ is a 2-group and $\ZZ_p{:}\ZZ_{(p-1)/2}$ is solvable, it follows that $M_1$ is solvable.
By Proposition \ref{Hall}, $M_1$ contains a $\{p, p_2\}$-Hall subgroup $H_1$ and a $\{p, p_3\}$-Hall subgroup $H_2$.
By Proposition \ref{odd-index}, we conclude that $[H_1], [H_2] \in \Delta(G)$. This shows that $\Delta(G) = \{[P_1], [P_2], [P_3], [P_4], [H_1], [H_2]\}$, which contradicts the assumption that $|\Delta(G)|=|\pi(G)|+1$.
So in what follows, we may assume that $p-1$ is a power of 2. Then $p$ is a Fermat prime.
Since $|\pi(G)|=4$ and $|G|_{2'}=(p(p^2-1))_{2'}$, we have that $p_2p_3\mid (p+1)$. Since $\soc(G/K)\cong \PSL(2,p)$, Proposition \ref{bray} implies that $G$ has a subgroup $M_2$ containing $K$ such that $M_2/K\cong \D_{p+1}$.
Since $K$ is a $2$-group and $M_2/K$ is solvable, it follows that $M_2$ is solvable.
By Proposition \ref{Hall}, there exists a $\{p_2, p_3\}$-Hall subgroup $H$ of $M_2$. Proposition \ref{odd-index} implies that $[H]\in \Delta(G)$ and hence $\Delta(G)= \{[P_1], [P_2], [P_3], [P_4], [H]\}$. Applying Proposition \ref{odd-index} again, we conclude that $|G|_{2'}$ is square-free, which means that $p+1=2p_2p_3$. Moreover, Lemma \ref{rm} shows that $P_1 \cong \ZZ_{2^{a_1}}$ or $\Q_{2^{a_1}}$. Deduced from $G$ is non-solvable, Proposition \ref{pb} implies that $P_1\cong \Q_{2^{a_1}}$.
Since the Sylow 2-subgroups of $\PSL(2,p)$ and $\PGL(2,p)$ are known to be dihedral, by Observation \ref{observation}, one deduces that $K$ must be a nontrivial cyclic 2-group.

In conclusion, $G/K\in \{\PGL(2,7), \PGL(2,9), \PSL(2,17), \PGL(2,17), \allowbreak\PSL(2,p), \allowbreak \PGL(2,p)\},$ where $K=\O_2(G)$ is cyclic, $P_1\cong \Q_{2^{a_1}}$ for $a_1\geqslant3$ integer, $p$ is a Fermat prime such that $p+1 = 2p_2p_3$. Hence Claim 2 holds. Combining with Claims 1 and 2, we can conclude that the necessity of Lemma \ref{non-solvable} holds.

\textbf{(II).} Next, we prove the sufficiency.  Assume that
$G/\O_2(G) \in \{\PSL(2,5)$, $\PGL(2,5)$, $\allowbreak \PGL(2,7)$, $\allowbreak \PGL(2,9)$, $\PSL(2,17)$, $\allowbreak \PGL(2,17)$, $\allowbreak\PSL(2,p)$, $\allowbreak \PGL(2,p)\}$, $\O_2(G)>1$ is cyclic, and a Sylow $2$-subgroup $P_1$ of $G$ is isomorphic to $\Q_{2^{a_1}}$ for integer $a_1\geqslant3$, where $p$ is a Fermat prime satisfying $p+1 = 2p_2 p_3$ for primes $p_2 < p_3 < p$. In what follows, we will analyze each case one by one.

(1) Assume that $G/\O_2(G) \cong \PSL(2,p)$ or $\PGL(2,p)$,
$P_1 \cong \Q_{2^{a_1}}$ and $\O_2(G)>1$ is a cyclic 2-group, where $p$ is a Fermat prime such that $p+1=2p_2p_3$ for primes $p_2<p_3<p$. Then $p-1$ is a power of 2 and $|G|_{2'}=pp_2p_3$, in particular, $\pi(G)=\{p_1=2, p_2, p_3, p\}$, i.e., $|\pi(G)|=4$.
Let $H$ be a subgroup perfect code of $G$ such that $1<H<G$. Since $P_1 \cong \Q_{2^{a_1}}$, Lemma \ref{cg} implies that $|H|$ is odd or $|G:H|$ is odd.

Assume first that $|H|$ is odd. Then $H\cong H\O_2(G)/\O_2(G)\leqslant G/\O_2(G)$. Since $|G|_{2'}=pp_2p_3$ and $H$ is solvable, checking the maximal subgroups of $G/\O_2(G)$ (see  Proposition \ref{bray} for the list) shows that $|H|\in\{p_2, p_3, p, p_2p_3\}$.
Let $p=p_4$ and $P_i\in \Syl_{p_i}(G)$ for each $i\in \{1,2,3,4\}$.
If $|H|\in \{p_2, p_3, p\}$, then $H$ is a Sylow subgroup of $G$ and $[H]=[P_2]$, $[P_3]$ or $[P_4]$. So in what follows, we may assume that $|H|=p_2 p_3$. Derived by $|G|_{2'}= p_2 p_3 p$, one yields that $H$ is a $\{p_2, p_3\}$-Hall subgroup of $G$.
On the other hand, Proposition \ref{bray} shows that every subgroup of order $p_2p_3$ in $\PSL(2,p)$ (resp. $\PGL(2,p)$) is contained in a maximal subgroup  isomorphic to $\D_{(p+1)}$ (resp. $\D_{2(p+1)}$), and such subgroups form a single conjugacy class.
Since $\D_{(p+1)}$(resp. $\D_{2(p+1)}$) is solvable and its subgroup of order $p_2 p_3$ is a $\{p_2, p_3\}$-Hall subgroup of $\D_{(p+1)}$(resp. $\D_{2(p+1)}$), Proposition \ref{Hall} implies that all such subgroups are conjugate in $\D_{(p+1)}$(resp. $\D_{2(p+1)}$).
It follows that all $\{p_2, p_3\}$-Hall subgroups of $\PSL(2,p)$(resp. $\PGL(2,p)$) form a single conjugacy class.
Let $H_0$ be another $\{p_2, p_3\}$-Hall subgroup of $G$. Since $G/\O_2(G)\cong \PSL(2,p)$ or $\PGL(2,p)$, it follows that $H_0\O_2(G)/\O_2(G)$ and $H\O_2(G)/\O_2(G)$ are conjugate in $G/\O_2(G)$. There exists an element $g \in G$ such that $(H_0\O_2(G))^g = H\O_2(G)$. Since $\O_2(G)\unlhd G$ and $H$ has odd order $p_2 p_3$, the group $H\O_2(G)$ is solvable. Noting that $H_0^g$ and $H$ are $\{p_2, p_3\}$-Hall subgroups of $H\O_2(G)$, applying Proposition \ref{Hall}, we yield that  $(H_0^g)^x= H_0^{gx}= H$ for some $x \in H\O_2(G)$. Hence all $\{p_2, p_3\}$-Hall subgroups of $G$ form a single conjugacy class.

Now assume that $|G:H|$ is odd.
Then $H$ contains a Sylow 2-subgroup of $G$ and hence $\O_2(G)\le H$. It concludes from $p-1$ is a power of 2 and $p+1=2p_2p_3$ that $p-1\ge 32$ and hence $|P_1/\O_2(G)|\ge 32$ as $P_1/\O_2(G)$ is a Sylow 2-subgroup of $G/\O_2(G)\cong\PSL(2, p)$ or $\PGL(2, p)$. Applying Proposition \ref{bray}, we can draw the conclusion that the only proper subgroups of odd index in $\PSL(2,p)$ and $\PGL(2,p)$ are their Sylow $2$-subgroups. Consequently $H$ is a Sylow $2$-subgroup of $G$.

Based on the above discussion, we obtain that $|\Delta(G)|=5$ and $|\Delta(G)|=|\pi(G)| + 1$, as desired.

(2) Assume that $G/\O_2(G)$ is isomorphic to $\PSL(2,17)$ or $\PGL(2,17)$. Then $\pi(G)=\{2, 3, 17\}$ and $|\pi(G)|=3$.
Let $H$ be a subgroup perfect code of $G$ and $1<H<G$. Since $P_1 \cong \Q_{2^{a_1}}$, Lemma \ref{cg} implies that exactly one of $|H|$ and $|G:H|$ is odd.

First assume that $|H|$ is odd. Since $|G|_{2'} = 3^2 \cdot 17$, it follows that $|H|\in \{3,9,17,51,153\}$.
By Proposition \ref{bray}, we see that neither $\PSL(2,17)$ nor $\PGL(2,17)$ contains a subgroup of order $51$ or $153$ and consequently $|H|\in \{3, 9, 17 \}$. If $|H| = 9$ or $17$, then $H$ is a Sylow $3$-subgroup of $G$ or a Sylow $17$-subgroup of $G$.
So in what follows, we may assume that $|H| = 3$. Let $P_2$ be a Sylow $3$-subgroup of $G$.
Since every Sylow 3-subgroup of $\PSL(2,17)$ and of $\PGL(2,17)$ is isomorphic to $\ZZ_9$, it follows that $P_2$ has a unique subgroup of order 3. Since all Sylow 3-subgroups of $G$ are conjugate and each contains a unique subgroup of order $3$, it follows that all subgroups of $G$ of order $3$ are conjugate.

Now assume that $|G:H|$ is odd. Then $H$ contains a Sylow 2-subgroup of $G$ and hence $\O_2(G)\le H$. By Proposition \ref{bray}, the only proper subgroups of odd index in $\PSL(2,17)$ and $\PGL(2,17)$ are their Sylow $2$-subgroups. Consequently $H$ is a Sylow $2$-subgroup of $G$.

Based on the above discussion, we obtain that $|\Delta(G)|=4$ and $|\Delta(G)|=|\pi(G)|+1$, as desired.

(3) Using a method completely analogous to the previous two cases, we can prove that the conclusion $|\Delta(G)|=|\pi(G)|+1$ also holds in all remaining cases. Thus, the proof of the sufficiency of Lemma \ref{non-solvable} is complete.
\qed

\begin{lemma}\label{os}
Let $G$ be a solvable group. Then $|\Delta(G)|\ge 2^{|\pi(G)|}-2$.
\end{lemma}
\demo For each nonempty proper subset $\pi$ of $\pi(G)$, let $H_\pi$ be a $\pi$-Hall subgroup of $G$. Since $G$ is solvable, Proposition \ref{Hall} implies that the set $\{H_\pi \mid \emptyset \neq \pi \subsetneq \pi(G)\}$ consists of $2^{|\pi(G)|}-2$ pairwise non-conjugate Hall subgroups.
 According to Proposition \ref{odd-index}, each such Hall subgroup is a subgroup perfect code of $G$, which establishes Lemma \ref{os}.
\qed

\begin{lemma}\label{solvable}
If $G$ is a solvable group and $|\Delta(G)|=|\pi(G)| + 1$, then $|\pi(G)|\le 2$ and one of following holds:
          \begin{enumerate} [font=\normalfont]
                \item $G\cong \ZZ_{p^3}$, where $p$ is an odd prime;

                \item $|G|=2^n$ and $\Z(G)$ is cyclic, the involutions in $G$ fall into two conjugacy classes;

                \item $|G|=2^n p^a$ for integers $a\in\{1,2\}$ and $n\geqslant 1$, $G/\O_2(G)\cong \ZZ_p^a{:}(P/\O_2(G))$ where $P$ is a Sylow $2$-subgroup of $G$, and the involutions of $G$ form a single conjugacy class; moreover, if $a=2$, then $P\cong \ZZ_{2^n}$ or $\Q_{2^n}$.
         \end{enumerate}
\end{lemma}
\demo Since $G$ is solvable, Lemma \ref{os} shows that $|\Delta(G)|=|\pi(G)| + 1\geqslant 2^{|\pi(G)|}-2$, i.e., $|\pi(G)|+3\geqslant 2^{|\pi(G)|}$, and so $|\pi(G)|\le 2$, as desired. Next, we will complete the proof of this lemma by discussing the two cases: $|\pi(G)| = 1$ and $|\pi(G)| = 2$.

{\bf(I)} Assume that $|\pi(G)|=1$. Then $|G|=p^n$ for some prime $p$ and positive integer $n$, and so $|\Delta(G)|=|\pi(G)|+1=2$.

Suppose first that $p\ge 3$. Derived from $|\Delta(G)|=2$ and Proposition \ref{odd-index}, one yields that $|G|=p^2$ or $p^3$. Note that the group of order $p^2$ is isomorphic to $\ZZ_{p^2}$ or $\ZZ_p \times \ZZ_p$. If $G\cong\ZZ_{p^2}$, then
$G$ has only one nontrivial subgroup, which isomorphic to $\ZZ_p$, and so $|\Delta(G)|=1$, a contradiction. Now assume that $G\cong \ZZ_p\times\ZZ_p$. However, since $\ZZ_p\times \ZZ_p$ has $p+1$ non-conjugate subgroup of order $p$ and $|\Delta(G)|=2$, applying Proposition \ref{odd-index}, one deduces that $G\ncong \ZZ_p\times \ZZ_p$. Now that $|G|=p^3$.
Let $H_i$ be a subgroup of $G$ of order $p^i$ for $i\in\{1,2\}$.
By Proposition \ref{odd-index} and the assumption that $|\Delta(G)|=2$, we can conclude that  $\Delta(G)=\{[H_1], [H_2]\}$.
Consequently, all subgroups of $G$ of order $p^i$ are conjugate in $G$.
Since $G$ is a $p$-group, $\Z(G)\neq 1$. As every subgroup of $\Z(G)$ is normal in $G$, it follows that $G$ has a normal subgroup of order $p$ contained in $\Z(G)$, and so $G$ contains a unique subgroup of order $p$.
On the other hand, the fact $G$ is a $p$-group shows that $|G:H_2|=p$ and $H_2\lhd G$, and consequently $G$ contains a unique subgroup of order $p^2$. It follows that $G\cong \ZZ_{p^3}$ in this case and the case (1) of Lemma \ref{solvable} holds.

Assume that $p=2$. Since $|\Delta(G)|=2$, Proposition \ref{chen2} implies that $G$ contains more than one involution. Further, applying Proposition \ref{kf} (a), we can draw the conclusion that $G$ has a nontrivial subgroup perfect code say $H$, which is a cyclic or a generalized quaternion 2-group.
Noting that $\Z(G)\neq 1$ ($G$ is a $2$-group), it must contain at least an involution, one yields that $G$ contains at least two conjugacy classes of involutions. Now choose an involution $x$ in $G$ such that $x\notin \bigcup_{g\in G} H^g$. Let $L$ be a subgroup of $G$ defined as follows: if $\langle x\rangle$ is a subgroup perfect code of $G$, then $L=\langle x\rangle$; otherwise, $L$ is a subgroup containing a unique involution $x$ of maximal order. An argument similar to that in the proof of Lemma \ref{rm} shows that $L$ is a subgroup perfect code of $G$.
Since $x\notin \bigcup_{g\in G} H^g$, it follows that $[L]\ne [H]$.
It deduces from $|\Delta(G)|=2$ that $\Delta(G)=\{[H], [L]\}$.
Applying Proposition \ref{kf} (b), one yields that the set of involutions in $G$ splits into two conjugacy classes. It is well known that once an abelian group contains two involutions, it must contain at least three involutions. It follows that $\Z(G)$ is cyclic, and hence the case (2) of Lemma \ref{solvable} holds.

{\bf(II)} Suppose that $|\pi(G)|= 2$. Then $|\Delta(G)|=3$. Assume that $\pi(G)=\{p, q\}$ for $p<q$ primes. Let $P$ and $Q$ be the Sylow $p$- and $q$-subgroups of $G$ respectively.

\textbf{Claim 1.} $p$ is even, i.e., $p=2$.

Assume for contradiction that $p$ is odd.
It follows from $|\Delta(G)|=3$ and Proposition \ref{odd-index} that $|G|=p q^2 $ or $p^2 q$.
For the former case, i.e., $|G|=p^2 q$, let $\ZZ_p\cong H<G$, applying Proposition \ref{odd-index}, we can deduce from $|\Delta(G)|=3$ that $\Delta(G)=\{[P], [Q], [H]\}$, and so $G$ has no subgroup of order $pq$.
Since $G$ is solvable, there exists an elementary abelian normal subgroup in $G$ say $N$.  If $N\cong \ZZ_p$ or $\ZZ_q$, then $NQ\leqslant G$ with $|NQ|=pq$ or $NH\leqslant G$ with $|NH|=pq$, a contradiction. Hence $N\cong \ZZ_p\times \ZZ_p$ and $N=P$, and so $G=N{:}Q$ and $H\leqslant N$. Since a two-dimensional vector space over a field with $p$ elements contains exactly  $p+1$ one-dimensional subspaces, which yields that $|[H]| =p+1$. Considering the action of $G$ on $[H]$ by conjugation, one deduces that $G$ is transitive on $[H]$.
Since every subgroup of $G$ of order $p$ is contained in $N$ and $N$ is abelian, one yields that $G/N\cong\ZZ_q$ is transitive on $[H]$, and hence, $(p+1)\mid q$, contradicting the fact that $p+1$ is even and $q$ is an odd prime.
So in what follows, we assume that the latter case holds, that is $|G|=p q^2$.
By an argument completely analogous to the one above, we conclude that
$G\cong(\ZZ_q\times \ZZ_q){:}\ZZ_p$ and $(q+1)\mid p$, which contradicts the assumption that $p<q$.  Therefore Claim 1 holds.

\textbf{Claim 2.} $|G|=2^n q^a$ and $G/\O_2(G)\cong \ZZ_q^a{:}P/\O_2(G)$ where integer $a\in\{1,2\}$ and $n$ is a positive integer.

Note that $\pi(G)=\{2,q\}$.
Suppose that $q^3\mid |G|$.  Let $H_1$ and $H_2$ be subgroups of $G$ of orders $q$ and $q^2$, respectively. Proposition \ref{odd-index} implies that $[H_1], [H_2]\in \Delta(G)$ and hence $\Delta(G)\supseteq\{[P],[Q],[H_1],[H_2]\}$, which contradicts the assumption that $|\Delta(G)|=|\pi(G)|+1=3$. Therefore, $|G|_{2'}=q^a$ for $a\in\{1,2\}$.
We assert that $P$ is a maximal subgroup of $G$.
If $|G|=2^n q$, then as desired. Now assume that $|G|=2^n q^2$.
If there is a maximal subgroup of $G$ such that $P<M$, applying Proposition \ref{odd-index}, we get that $\Delta(G)\supseteq \{[P],[Q],[M],[H_1]\}$, which contradicts the assumption that $|\Delta(G)|=3$. Hence the assertion holds.
Consider the action of $G$ by conjugation on the set $\Omega:=\{P^g\mid g\in G\}$.
The kernel of the action is $\O_2(G)=\bigcap\limits_{g \in G} P^g$.
The maximality of $P$ yields that $G/\O_2(G)$ is primitive on $\Omega$. In particular, $P/\O_2(G)$ is the stabilizer of the point $P\in \Omega$ in $G/\O_2(G)$.
Since $|\Omega|\in \{q,q^2\}$ and $|G|_{2'}=q^a$, applying Proposition \ref{pi+1} yields that
$G/\O_2(G)\cong \ZZ_q^a{:}(P/\O_2(G))$ for $a\in \{1,2\}$.
Based on the above, we have proven the correctness of Claim 2.

\textbf{Claim 3.} The involutions of $G$ form a single conjugacy class; in particular, if $a=2$, then $P\cong \ZZ_{2^n}$ or $\Q_{2^n}$.

Assume first that $|G|=2^n q$ for integer $n\geqslant 1$.
Given that every proper subgroup of $G$ has even order except those of order $q$, and that $|\Delta(G)| = 3$, Lemma \ref{cg} forces $P$ to be neither cyclic nor generalized quaternion.
By Lemma \ref{rm}, we can draw the conclusion that $P$ contains a nontrivial subgroup perfect code of $G$ isomorphic to a cyclic 2-group or generalized quaternion 2-group, say $H$. Thus $\Delta(G)=\{[P], [Q], [H]\}$.
To prove Claim 3, we suppose for a contradiction that there exists an involution $x\not\in \bigcup_{g\in G} H^g$. Using methods similar to those in the proof of Lemma \ref{rm}, we conclude that $G$ has a proper subgroup perfect code containing $x$ isomorphic to either a cyclic 2-group or a generalized quaternion 2-group, say $L$. It follows from $x\not\in \bigcup_{g\in G} H^g$ that $[L]\ne [H]$ and so $|\Delta(G)|\ge 4$, contradicting the assumption that $|\Delta(G)|=3$. Therefore, $\bigcup_{g\in G} H^g$ contains all involutions of $G$, combining with the assumption that $H$ has only one involution, we conclude that all involutions of $G$ form a single conjugacy class.

Now assume that $|G|=2^n q^2$. Let $H$ be a subgroup of order $q$ of $G$. By Proposition \ref{odd-index} and the assumption that $|\Delta(G)|=|\pi(G)|+1=3$, we get that
$\Delta(G)=\{[P],[Q],[H]\}$. Applying Lemma \ref{rm}, $P$ is a cyclic or generalized quaternion $2$-group, which means that $P$ has only one involution(see Proposition \ref{chen2}).
Since every involution is contained in some Sylow $2$-subgroup of $G$,  it follows from the Sylow theorems that all involutions of $G$ form a single conjugacy class.

Based on the previous discussion, we have proven the correctness of Lemma \ref{solvable}.\qed

We now have enough knowledge to prove Theorem \ref{main2}.

\noindent\textbf{Proof of Theorem \ref{main2}.}
Assume that $G$ is a finite group, $P$ is a Sylow $2$-subgroup of $G$, and $|\Delta(G)| = |\pi(G)| + 1$.
By Lemma \ref{non-solvable},  $G$ is non-solvable if and only if the following holds:
    \begin{enumerate}[label=(\roman*), font=\normalfont]
        \item $P \cong \Q_{2^n}$ for some integer $n\ge 3$ and $\O_2(G)>1$ is a cyclic subgroup;

        \item $G/\O_2(G) \in \{\PSL(2,5),\PGL(2,5), \PGL(2,7), \PGL(2,9),\allowbreak\PSL(2,17), \allowbreak \PGL(2,17),\allowbreak \PSL(2,p), \allowbreak\PGL(2,p)\}$, where $p$ is a Fermat prime satisfying $p+1 = 2p_2 p_3$ for primes $p_2 < p_3 < p$.
    \end{enumerate}
If $G$ is solvable, by Lemma \ref{solvable}, then for some positive integer $n$ and prime $p\ge 3$, one of the following holds:
          \begin{enumerate}[label=(\roman*), font=\normalfont]
                \item $G\cong \ZZ_{p^3}$;

                \item $|G|=2^n $ and $\Z(G)$ is cyclic, all involutions of $G$ split into two conjugacy classes;

                \item $|G|=2^n p^a$ for $a\in\{1,2\}$, $G/\O_2(G)\cong \ZZ_p^a{:}(P/\O_2(G))$, and the involutions of $G$ form a single conjugacy class; moreover, if $a=2$, then $P\cong \ZZ_{2^n}$ or $\Q_{2^n}$.

         \end{enumerate}
\qed

\section{The proof of Theorem \ref{main3}} \label{sec5}

In this section, we shall prove Theorem \ref{main3}, and throughout we let $G$ be a finite non-solvable group such that $|\Delta(G)|\leqslant6$, and we write $\pi(G)=\{p_1,p_2,\dots,p_n\}$ with $p_1<p_2<\dots<p_n$.  For each $i\in \{1,\dots,n\}$, let $P_i\in \Syl_{p_i}(G)$.

Before starting the proof, we first present a key proposition.

According to the classification theorem of finite simple groups, we can draw the following conclusion(see Appendix \ref{Appendix} for a detailed verification).

\begin{proposition}\label{forK}
Let $G$ be a finite non-abelian simple group with $3\le|\pi(G)|\le 4$. Suppose that $|G|=2^{a_1} p_2^{a_2} p_3^{a_3} p_4^{a_4}$, where $2<p_2<p_3<p_4$ are primes, $a_2,a_3\ge 1$ and $a_4\ge 0$. Then the following holds.
\begin{enumerate}[font=\normalfont, label=(\arabic*)]
\item If $p_2^{a_2} p_3^{a_3} p_4^{a_4}$ is square-free, then $G$ is isomorphic to $\A_5$, ${^2\!B_2(8)}$, $\PSL(2,7)$, $\PSL(2,16)$, $\PSL(2,31)$, $\PSL(2, p_4)$ or $\PSL(2,2^f)$, where $r_1, r_2$ and $r>3$ are odd primes, $((p_4-1)_{2'}, (p_4+1)_{2'})\in \{(r_1,r_2), (1,3r)\}$, $f>3$, $2^f-1$ and $\frac{2^f+1}{3}>3$ are odd primes.

\item If $|G|_{2'}=p_i^2 p_j$, then $G$ is isomorphic to $\A_6$, $\PSL(2,8)$ or $\PSL(2,17)$, where $i\ne j$ and $i,j\in \{2,3\}$.
\end{enumerate}
\end{proposition}

It follows from $G$ is non-solvable that $p_1=2$. Applying Corollary \ref{1.3cor}, we have that $3\le |\pi(G)|\le 4$ and hence we may assume that $|G|=2^{a_1} p_2^{a_2} p_3^{a_3} p_4^{a_4}$, where $a_2, a_3\ge 1$ and $a_4\ge 0$ are integers. Let $M$ be a maximal subgroup of $G$ such that $P_1 \le M$.
Noting that every subgroup of odd order or odd index in $G$ is a subgroup perfect code of $G$ (see Proposition \ref{odd-index}), one yields that $[M]\in \Delta(G)$ and
$[P_i]\in \Delta(G)$ for $i\in\{1,2,3,4\}$. Further, we can conclude that $a_2+a_3+a_4\le 5$ by the assumption $|\Delta(G)|\le 6$.
Consider the right multiplication action of $G$ on the set $\Omega := \{ Mg \mid g \in G \}$, the kernel is $\bigcap_{g \in G} M^g$ denoted by $K$, and consequently, $G/K$ is a primitive permutation group of odd degree $|G:M|$. Let $L/K=\soc(G/K)$.
Note that $|G|=2^{a_1} p_2^{a_2} p_3^{a_3} p_4^{a_4} $ for $a_2+a_3+a_4\leq 5$ and $a_i\ge 0$, the following conclusion can be directly read from \cite[Theorem 1]{ls}(for details, see Appendix A):
\begin{equation}\tag{*}\label{star}
\vcenter{\hbox{\parbox{0.92\linewidth}{
\begin{enumerate} [font=\normalfont]
    \item $G/K\cong \ZZ_p^d{:}(M/K)$, where $p$ is an odd prime, $1\le d\le 4$ and $M/K$ is an irreducible subgroup of $\GL(d,p)$;

    \item $T^2\lhd G/K \leq G_0 \wr \S_2$, where $G_0$ is a primitive group of odd degree $n_0$ with nonabelian simple socle $T$, and the wreath product has the product action of degree $n=n_0^2$;

    \item $L/K\cong \A_5, \A_6, \A_7, \A_8, \M_{11}$ or $\M_{12}$;

    \item \parbox[t]{\linewidth}{\sloppy $L/K\cong \PSL(2, p)$, $\PSL(2, 25)$, $\PSL(2, 49)$, $\PSL(2,27)$, $\PSL(3,3)$, $\PSL(3,5)$, $\PSp(4,3)$, $\PSU(3,3)$ or $\PSU(3,7)$, where $p\ge 5$ is an odd prime;}

    \item \parbox[t]{\linewidth}{\sloppy $L/K\cong \PSL(2, 2^f)$, $\PSL(2,8)$, $\PSL(2,16)$, $\PSL(3, 4)$ , $\PSL(3, 8)$, $\PSp(4,4)$, $\PSU(3,4)$, ${^2\!B_2}(2^{3})$ or ${^2\!B_2}(2^{5})$, where $f>3$ is a prime and $(M/K)\cap (L/K)$ is a parabolic subgroup of $L/K$.}
\end{enumerate}
}}}
\end{equation}

In the following, we will deal with the above cases one by one using some technical lemmas.
\begin{lemma}\label{case (1)}
The case (1) cannot hold.
\end{lemma}
\proof Assume that the case (1) holds, that is $G/K\cong \ZZ_p^d{:}(M/K)$, where $p$ is an odd prime, $1\le d\le 4$ and $M/K$ is an irreducible subgroup of $\GL(d,p)$. It is clear that $p\mid|G/K|$.

\textbf{Claim 1.}  $K$ is non-solvable.

Suppose for contradiction that $K$ is solvable.
Since $G/K\cong \ZZ_p^d{:}(M/K)$, there exists a normal subgroup $N$ of $G$ such that $N/K\cong \ZZ_p^d$ and $G/K=(N/K){:}(M/K)$, in particular, $G=NM$.
Since $K$ is solvable and $N/K\cong \ZZ_p^d$, one yields that $N$ is solvable.
And because $G$ is non-solvable and $G/N\cong M/K$, $M$ is non-solvable and hence $|\pi(M)|\geqslant 3$. Since both $NP_1$ and $KP_1$ are solvable, one yields that $NP_1\ncong M$ and $KP_1\ncong M$. Derived from $KP_1/K\cong P_1/P_1\cap K$ is a 2-group and $N/K\cong\ZZ_p^d$ for odd prime $p$, one yields that $NP_1\ncong KP_1$.
Let $p_i\in \pi(M)$ be such that $p_i\ne p$ for $i\in\{2,3,4\}$. It follows from $N\lhd G$ and $N$ is solvable that $NP_i$ is a solvable subgroup of $G$. By Proposition \ref{Hall}, $NP_i$ contains a $\{p,p_i\}$-Hall subgroup, say $D$, and further, applying Proposition \ref{odd-index}, one yields that $\Delta(G)\supseteq\{[P_1],[P_2],[P_3],[M],[KP_1],[NP_1],[D]\}$, which contradicts the assumption that $|\Delta(G)|\leqslant 6$. Thus $K$ is non-solvable and Claim 1 holds.

\textbf{Claim 2.} $M/K$ is a 2-group.

Suppose for contradiction that $M/K$ is not a 2-group, i.e., there exists an odd prime $p_i\in \pi(G)$ such that $p_i \mid |M/K|$ for $i\in\{2,3,4\}$. Thus $KP_1$ is a proper subgroup of $M$. It follows from Proposition \ref{odd-index} that $[KP_1]\in\Delta(G)$. Recall that $3\le |\pi(G)|\le 4$. Noting that $K$ is non-solvable, one yields that $3\leqslant|\pi(K)|\leqslant|\pi(G)|\leqslant4$.

Suppose first that $|\pi(G)| = 3$. Then $\pi(K) = \pi(G)$. Because $p_i\mid|M/K|$ and $p\mid|G/K|$, we have that $p_i^2 \mid |G|$ and $p^2 \mid |G|$. It follows that $G$ possesses subgroups $H_1\cong \ZZ_p$ and $H_2\cong\ZZ_{p_i}$ that are not Sylow subgroups of $G$.
Again by Proposition \ref{odd-index}, we obtain that $[H_1],  [H_2]\in \Delta(G)$, and hence $\{ [P_1],[P_2],[P_3],[M],[KP_1],[H_1], [H_2]\} \subseteq \Delta(G)$, contradicting the assumption that  $|\Delta(G)| \le 6$.

Now assume that $|\pi(G)|=4$. Since $p_i p\mid |G/K|$, it follows that $p_i^2\mid |G|$ or $p^2\mid |G|$. Without loss of generality, we assume that $p_i^2\mid |G|$. Thus $G$ contains a subgroup $H\cong \ZZ_{p_i}$ that is not a Sylow $p_i$-subgroup of $G$.
Applying Proposition \ref{odd-index}, one yields that $$\{ [P_1],[P_2],[P_3],[P_4],[M],[KP_1], [H]\} \subseteq \Delta(G),$$
which again contradicts the assumption that $|\Delta(G)| \le 6$.
Hence Claim 2 is true, i.e., $M/K$ is a 2-group.

\textbf{Claim 3.} $|K|_{2'}\in \{p_ip_j, p_2 p_3 p_4, p_i^2 p_j, p_i p_j^2\}$ for $i,j\in\{2,3,4\}$ and $i \neq j$.

Note that $3\leqslant|\pi(K)|\leqslant|\pi(G)|\leqslant4$. Suppose that $|\pi(G)|=3$. Then $|G|=2^{a_1} p_2^{a_2} p_3^{a_3}$ for $a_1$, $a_2$ and $a_3$ positive integers and $\pi(K)=\pi(G)$. Since $G/K\cong\ZZ_p^d{:}(M/K)$ and $\pi(G)=\pi(K)$, we have $p\mid |G/K|$ and hence $p^2 \mid |G|$. Then $G$ contains a subgroup isomorphic to $\ZZ_p$, say $H$, which is not a Sylow $p$-subgroup of $G$.
Assume that $p=p_i$ for some $i\in\{2,3\}$, and let $p_j\in \pi(G)$ be an odd prime such that $p_j\ne p_i$. If $a_j \ge 3$, then $G$ possesses two subgroups $H_1$ and $H_2$ of orders $p_j$ and $p_j^2$, respectively, neither of which is a Sylow $p_j$-subgroup of $G$. Applying Proposition \ref{odd-index}, $\{ [P_1],[P_2],[P_3],[M],[H], [H_1],[H_2]\}\subseteq \Delta(G)$, contradicting the assumption that $|\Delta(G)|\le 6$. Hence $a_j\le 2$.
If $p^3 \mid |K|$, then $p^4 \mid |G|$, and we can get the same contradiction as above. Thus $p_i^3 \nmid |K|$ (note that $p_i=p$), that is $|K|_{2'}\mid p_i^2 p_j^2$. Suppose that $|K|_{2'}= p_i^2 p_j^2$. Then $p_i^3\mid |G|$ and $p_j^2\mid |G|$, and consequently, $G$ possesses two subgroups $Q_1$ and $Q_2$ of orders $p_j$ and $p_i^2$, such that $Q_1<P_j$ and $Q_2<P_i$. By Proposition \ref{odd-index}, we see that $\Delta(G)\supseteq\{[P_1], [P_2],[P_3],[M],[H], [Q_1],[Q_2]\}$, which contradicts the assumption that $|\Delta(G)|\le 6$. Hence $|K|_{2'}\in\{ p_i p_j, p_i^2 p_j, p_ip_j^2\}$, as desired.

Now suppose that $|\pi(G)| = 4$, i.e., $|G| = 2^{a_1} p_2^{a_2} p_3^{a_3} p_4^{a_4}$, where $a_1$, $a_2$, $a_3$ and $a_4$ are positive integers. Since $3\leqslant|\pi(K)|\leqslant|\pi(G)| = 4$ and $\pi(K)\subseteq \pi(G)$, we assume that $p_i p_j\mid |K|$, where $i,j\in\{2,3,4\}$ and $i \neq j$.
Assume first that $|\pi(K)|=4$. Then $\pi(K)=\pi(G)$. Since $p\mid |G/K|$, it follows that $p^2\mid |G|$ and hence $G$ contains a subgroup $H\cong \ZZ_p$ that is not a Sylow $p$-subgroup of $G$.
By Proposition \ref{odd-index} and the assumption that $|\Delta(G)|\le 6$,
we conclude that
$\{ [P_1],[P_2],[P_3],[P_4],[M],[H]\}=\Delta(G)$, and so $a_2+a_3+a_4=4$ and $|K|_{2'}= p_2 p_3 p_4$, as required.
Now assume that $|\pi(K)|=3$, i.e., $\pi(K)=\{2, p_i, p_j\}$ for $i,j\in\{2,3,4\}$ and $i \neq j$. Since $G/K\cong\ZZ_p^d:M/K$ with $M/K$ being a 2-group, one deduces from $|\pi(G)|=4$ and $|\pi(K)|=3$ that $p\notin \pi(K)$. And because $\{ [P_1],[P_2],[P_3],[P_4],[M] \}\subseteq \Delta(G)$ and $|\Delta(G)|\le 6$, Proposition \ref{odd-index} yields that $a_i+a_j\le 3$. Since $a_i\geqslant 1$ and $a_j\geqslant 1$, we have
 $|K|_{2'}\in \{p_i p_j, p_i^2 p_j, p_i p_j^2\}$.
Combining two cases($|\pi(K)|=3,$ or $4$), we conclude that $|K|_{2'}\in \{p_ip_j, p_2 p_3 p_4, p_i^2 p_j, p_i p_j^2\}$ with $i,j\in\{2,3,4\}$ and $i \neq j$. Hence Claim 3 holds.

For convenience, in the rest of this proof, we denote
$$\Sigma= \{\A_5, \A_6, {^2\!B_2(8)},\PSL(2,q) \mid q=7,8,16,17,31,p_4,2^f\}.$$

\textbf{Claim 4.} There exists a chief factor of $G$ in $\Sigma$ (up to isomorphic) such that $((p_4-1)_{2'}, (p_4+1)_{2'})\in \{(r_1,r_2), (1,3r)\}$, $r_1, r_2$ and $r$ are odd primes, $f>3$, $2^f-1$ and $\frac{2^f+1}{3}$ are odd primes.

Now consider the normal series $1 \lhd K \lhd G$.
By the Schreier refinement theorem\cite[p.6, Theorem 2.7]{gorenstein}, this series can be refined to a chief series of $G$:
\begin{equation}\label{series}
G = G_0 \rhd G_1 \rhd \dots \rhd G_n = 1 \tag{I}
\end{equation}
where $G_k \lhd G$ for each $k\in \{1,2,\dots,n\}$, and each factor $G_{k-1}/G_k$ is a minimal normal subgroup of $G/G_k$, i.e., a characteristically simple group. Without loss of generality, we can assume that $G_l = K$ for some positive integer $l$.
Since $K$ is non-solvable and $G/K$ is solvable, there exists an integer $m > l$ such that $G_{m-1}/G_m$ is non-abelian characteristically simple, in particular, $G_{m-1}/G_m$ is non-solvable. It follows from $G_{m-1}\leqslant K$ and Claim 3 that $|G_{m-1}|_{2'}\in \{p_ip_j, p_2 p_3 p_4, p_i^2 p_j, p_i p_j^2\}$ for $i,j\in\{2,3,4\}$ and $i \neq j$. Noting that $G_{m-1}/G_m$ is non-solvable, one yields that $|\pi(G_{m-1}/G_m)|\ge 3$, and hence $|\pi(G_m)|\le 2$, that is $G_m$ is solvable. Observe that $|G_{m-1}/G_m|_{2'} \in \{ p_i p_j, p_2 p_3 p_4, p_i^2 p_j, p_i p_j^2 \}$, combining with $G_{m-1}/G_m$ is non-abelian characteristically simple, we conclude that $G_{m-1}/G_m$ is a non-abelian simple group. Further, applying Proposition \ref{forK}, one deduces that $G_{m-1}/G_m \in\Sigma$, where $((p_4-1)_{2'}, (p_4+1)_{2'})\in \{(r_1,r_2), (1,3r)\}$, $r_1, r_2$ and $r$ are odd primes, $f>3$, $2^f-1$ and $\frac{2^f+1}{3}$ are odd primes. Hence Claim 4 hold.

\textbf{The final contradiction.} Assume that $G_{m-1}/G_m$ is a non-abelian simple chief factor of $G$ described in Claim 4, that is $G_m, G_{m-1}$ are normal subgroups of $G$ as in (\ref{series}) such that $G_{m-1}/G_m$ is a chief factor of $G$ and $G_{m-1}/G_m\in \Sigma$, where $((p_4-1)_{2'}, (p_4+1)_{2'})\in \{(r_1,r_2), (1,3r)\}$ for odd primes $r_1, r_2$ and $r$, and where $f>3$, $2^f-1$ and $\frac{2^f+1}{3}$ are odd primes. In particular, $G_{m-1}\unlhd K$.

We assert that $G_{m-1}/G_m\not\in\{ \PSL(2,2^f), \PSL(2,8), \PSL(2,16), \allowbreak {^2\!B_2(8)}\}$, where $f>3$, $2^f-1$, and $\frac{2^f+1}{3}$ are odd primes. Firstly, suppose on the contrary that $G_{m-1}/G_m\cong \PSL(2,2^f)$, where $f>3$, $2^f-1$, and $\frac{2^f+1}{3}$ are odd primes. By a calculation, we obtain that $|\pi(G_{m-1}/G_m)| = 4$. Recall that $|\pi(G)|\le 4$. It follows from $\pi(G_{m-1}/G_m)\subseteq \pi(K)\subseteq \pi(G)$ that $\pi(G)=\pi(K)=\pi(G_{m-1}/G_m)$. Since $p\mid |G/K|$, it follows that $p^2\mid |G|$, and consequently $G$ contains a subgroup isomorphic to $\ZZ_p$, say $H$, which is not a Sylow $p$-subgroup of $G$. By Proposition \ref{odd-index} and the assumption that $|\Delta(G)|\le 6$, we conclude that $\Delta(G) = \{[P_1],[P_2],[P_3],[P_4],[M],[H]\}$.
Lemma \ref{rm} yields that $P_1$ is a cyclic or a generalized quaternion 2-group.
Moreover, since $G$ is non-solvable, Proposition \ref{pb} implies that $P_1\cong \Q_{2^{a_1}}$.
Now that $P_1 \cong \Q_{2^{a_1}}$, in particular, the Sylow $2$-subgroups of $G_{m-1}/G_m$ are also noncyclic.
And because $(P_1 \cap G_{m-1})G_m / G_m\in\Syl_2(G_{m-1}/G_m)$ and $(P_1 \cap G_{m-1})G_m / G_m\leqslant P_1G_m/G_m\cong P_1/(P_1\cap G_m)$, one yields that $(P_1 \cap G_{m-1})G_m/G_m$ is either dihedral or generalized quaternion. However, the Sylow $2$-subgroups of $\PSL(2,2^f)$ are isomorphic to $\ZZ_2^f$ with $f>3$, a contradiction. Secondly, using methods similar to those mentioned above, we can conclude that $G_{m-1}/G_m\not\in\{ {^2\!B}_2(8), \PSL(2,8),\PSL(2,16)\}$. Therefore, the assertion holds.

Now assume that $G_{m-1}/G_m\in\{\A_5, \A_6, \PSL(2,7), \PSL(2,17),\PSL(2,31), \PSL(2,p_4)\}$, where $((p_4-1)_{2'}, (p_4+1)_{2'})\in \{(r_1,r_2), (1,3r)\}$ for odd primes $r_1, r_2$ and $r$. It follows that  $\Out(G_{m-1}/G_m)$ is a $2$-group.
Define $C$ to be the preimage of $\C_{G/G_m}(G_{m-1}/G_m)$ in $G$, that is  $C/G_m=\C_{G/G_m}(G_{m-1}/G_m)$. Since $G_{m-1}/G_m$ is a non-abelian simple group, it follows that $(G_{m-1}/G_m)\cap (C/G_m)$ is trivial in $G/G_m$ and hence $G_{m-1}\cap C=G_m$ and $|G_{m-1}C:C|=|G_{m-1}:G_{m-1}\cap C|=|G_{m-1}:G_m|$.
Since $|G:M|=p^d$ and $|G_{m-1}/G_m|$ is a divisor of $|G:C|$, we can see that $CP_1\ncong M$.
Applying the N/C-Theorem, we have that $G/C\cong (G/G_m)/(C/G_m)=\N_{G/G_m}(G_{m-1}/G_m)/\C_{G/G_m}(G_{m-1}/G_m)\lesssim \Aut(G_{m-1}/G_m)$.
On the other hand, derived from $G_{m-1}/G_m\cong Inn(G_{m-1}/G_m)$ (as $G_{m-1}/G_m$ is a nonabelian simple group) and $\Aut(G_{m-1}/G_m)/\Inn(G_{m-1}/G_m)\cong\Out(G_{m-1}/G_m) $ is a 2-group, we obtain that $\frac{|G/C|}{|G_{m-1}/G_m|}=\frac{|G:G_{m-1}|}{|C:G_m|}$ is a power of 2. Noting that $p\mid|G/K|$ and $G_{m-1}\leqslant K$, one yields that $p\mid |G:G_{m-1}|$ and hence $p\mid |C:G_m|$.
Let $C_1$ be a subgroup of $C$ such that $C_1/G_m$ is a subgroup of order $p$ in $C/G_m=\C_{G/G_m}(G_{m-1}/G_m)$, and let $D/G_m$ be a subgroup of $G_{m-1}/G_m$ such that $D/G_m\cong \ZZ_{p_i}$ for some odd prime $p_i\ne p$ and $p_i\in \pi(G_{m-1}/G_m)$. Since $C_1/G_m\le C/G_m$, it follows that $(D/G_m)(C_1/G_m)=DC_1/G_m$ is a $\{p,p_i\}$-subgroup of $G/G_m$.
Noting that $G_m$ is solvable and $D C_1/G_m$ is solvable (as $|D C_1/G_m|$ is odd), one yields that $D C_1$ is solvable. Further, applying Proposition \ref{Hall}, one yields that $D C_1$ contains a $\{p,p_i\}$-Hall subgroup, say $H$.
If $|\pi(G)|=4$, then by Proposition \ref{odd-index}, we have that $\Delta(G)= \{[P_1], [P_2], [P_3], [P_4], [M], [C P_1], [H]\}$, which contradicts the assumption that $|\Delta(G)|\le 6$.
Now assume that $|\pi(G)|=3$. Since $|\pi(K)|=3$ and $p\mid |G:K|$, it follows that $p^2\mid |G|$.
Let $R$ be a subgroup of $G$ that is isomorphic to $\ZZ_p$.
By Proposition \ref{odd-index}, we have that
$\Delta(G)= \{[P_1], [P_2], [P_3], [M], [C P_1], [R], [H]\}$, a contradiction can still be obtained.

We have finally proved the correctness of Lemma \ref{case (1)}.
\qed

\begin{lemma}\label{case (2)}
The case (2) cannot hold.
\end{lemma}
\proof
Assume for contradiction that the case (2) holds, that is, $T^2\lhd G/K \leq G_0 \wr \S_2$, where $G_0$ is a primitive group of odd degree $n_0$ with nonabelian simple socle $T$, and the wreath product has the product action of degree $n=n_0^2$.
Since $T$ is a nonabelian simple group, it follows that $|T|$ is divisible by at least two distinct odd primes, say $p_i$ and $p_j$ for $i, j\in\{2, 3, 4\}$. As $T^2 \lhd G/K$ gives $p_i^2 p_j^2 \mid |G|$, $G$ has two subgroups $H_1\cong \ZZ_{p_i}$ and $H_2\cong \ZZ_{p_j}$, which are not Sylow subgroups of $G$.
Note that $|G/K| \mid 2^{a_1} p_2^{a_2} p_3^{a_3} p_4^{a_4} $, where $a_2+a_3+a_4\leq 5$ and $a_i\ge 0$ for $i\in\{1, 2, 3, 4\}$. It follows that $|\pi(K)| \leqslant 2$ and hence $K$ is solvable.
Since $T\times T$ contains a subgroup of order $p_i p_j$ say $R$, there exists a subgroup $H$ of $G$ such that $H/K \cong R$.
Since both $K$ and $R$ are solvable, it follows that $H$ is solvable. Proposition \ref{Hall} implies that $H$ contains a $\{p_i, p_j\}$-Hall subgroup, say $Q$.
By Proposition \ref{odd-index} and the assumption that $|\Delta(G)|\le 6$, we obtain that
$\Delta(G)=\{ [P_1], [P_2], [P_3],  [H_1], [H_2], [Q]\}.$
Quoting the result of Lemma \ref{rm}, we conclude that $P_1$ is a cyclic or a generalized quaternion 2-group.
Since both $G$ and $T$ are non-solvable, Proposition \ref{pb} yields that the Sylow $2$-subgroups of $G$ and $T$ are noncyclic. Thus $P_1\cong \Q_{2^{a_1}}$ and the Sylow $2$-subgroup of $T$ is generated by at least two elements.
It follows that the Sylow $2$-subgroup $(P_1K/K) \cap T^2$ of $T^2$ requires at least four generators. On the other hand, $(P_1K/K) \cap T^2\leqslant P_1K/K\cong P_1/P_1\cap K$ shows that $(P_1K/K) \cap T^2$ is dihedral or generalized quaternion, which can be generated by two elements, a contradiction. Therefore, Lemma \ref{case (2)} holds.
\qed

\begin{lemma}\label{case (3)}
If the case (3) holds, then $\soc(G/K)\cong \A_5$ or $\A_6$ and $K\cong \O_2(G)$.
\end{lemma}
\proof Suppose the case (3) holds, that is $L/K\in\{\A_5, \A_6, \A_7, \A_8, \M_{11}, \M_{12}\}$. Note that $L/K=\soc(G/K)$.

\textbf{Claim 1.} $L/K\cong\A_5$ or $\A_6$.

Assume that $L/K\cong \A_7$. Then $G/K\cong \A_7$ or $\S_7$ with $|\pi(G/K)|=4$. Recalling that $|G|=2^{a_1} p_2^{a_2} p_3^{a_3} p_4^{a_4}$ with nonnegative integers $a_i$ and $a_2+a_3+a_4\le 5$, we conclude that $|\pi(K)|\le 2$ and consequently $K$ is solvable.
Referring to the Atlas \cite[p.10]{atlas}, we find that both $\A_7$ and $\S_7$ have two non-isomorphic maximal subgroups, which are not 2-groups and have odd index.
Consequently, there exist two corresponding maximal subgroups $M_1$ and $M_2$ of $G$ of odd index that are not 2-groups. On the other hand, since $3^2\mid |G|$, $G$ contains a subgroup isomorphic to $\ZZ_3$, say $H$, which is not a Sylow $3$-subgroup of $G$. Applying Proposition \ref{odd-index}, we obtain that
$\Delta(G)\supseteq \{ [P_1],[P_2],[P_3],[P_4], [M_1],[M_2], [H]\}$,
which contradicts the assumption that $|\Delta(G)|\le 6$. Therefore, the case $L/K\cong\A_7$ cannot occur.
For the case where $L/K\in\{ \A_8, \M_{11}, \M_{12}\}$, it is similar to the previous one, and the same reasoning shows that it cannot occur. Hence Claim 1 holds.

\textbf{Claim 2.} $K$ is solvable.

Assume for a contradiction that $K$ is non-solvable. Then $|\pi(K)|\ge 3$ and $M$ is not a 2-group as $K\le M$.
Given that $L/K\cong \A_5$ or $\A_6$, we have that $\{2,3,5\}\subseteq \pi(K)\subseteq \pi(G)$.
Since $3\le |\pi(G)|\le 4$, it follows that $p^2\mid |G|$ for $p=3$ or $5$, and so $G$ contains a subgroup isomorphic to $\ZZ_p$, say $H$, which is not a Sylow $p$-subgroup of $G$.

Suppose first that $|\pi(G)|=4$. By Proposition \ref{odd-index} and the assumption that $|\Delta(G)|\le 6$, we obtain that $\Delta(G)= \{[P_1], [P_2], [P_3], [P_4], [H], [M]\}$.
Applying Proposition \ref{odd-index} again, we conclude that $a_2+a_3+a_4=4$ and consequently $L/K\ncong \A_6$ as $|\A_6|_{2'}=3^2\cdot 5$ and $|\pi(K)|\geqslant 3$. Thus $L/K\cong \A_5$ and $G/K\cong \A_5$ or $\S_5$.
A detailed check of the Atlas\cite[p.2]{atlas} shows that neither $\A_5$ nor $\S_5$ has a $2$-subgroup as a maximal subgroup. Therefore, the maximal subgroup $M/K$ of $G/K$ cannot be a 2-group and consequently $KP_1$ is a proper subgroup of $M$ with odd index. Again by Proposition \ref{odd-index}, one yields that $[KP_1]\in \Delta(G)$ and hence $|\Delta(G)|\ge 7$, contradicting the assumption that $|\Delta(G)|\leqslant 6$.

Now assume that $|\pi(G)|=3$. Then $\pi(K)=\pi(G)=\pi(G/K)=\{2,3,5\}$.
It follows that $3^2\mid |G|$ and $5^2\mid |G|$. There exists subgroups $H_1$ of order 3 and $H_2$ of order 5 that are not Sylow subgroups of $G$.
By Proposition \ref{odd-index} and the assumption that $|\Delta(G)|\le 6$, we obtain that $\Delta(G)= \{[P_1], [P_2], [P_3], [M], [H_1], [H_2]\}$.
Applying Proposition \ref{odd-index} again, we conclude that $a_2+a_3=4$ and consequently $L/K\ncong \A_6$ as $|\A_6|_{2'}=3^2\cdot 5$ and $|\pi(K)|=3$. Thus $G/K\cong \A_5$ or $\S_5$. Similarly, we can deduce from the above argument that $KP_1<M$ and $[KP_1]\in \Delta(G)$, which shows that $|\Delta(G)|\ge 7$, contradicting the assumption that $|\Delta(G)|\le 6$. Therefore, Claim 2 holds, i.e., $K$ is solvable.

\textbf{Claim 3.} $K$ is a 2-group.

Suppose for contradiction that $K$ is not a 2-group. Note that $|G|=2^{a_1} p_2^{a_2} p_3^{a_3} p_4^{a_4}$ and $a_2+a_3+a_4\le 5$ where $a_2, a_3\ge 1$ and $a_4\ge 0$. Then $p_i\mid |K|$ for some $i\in\{2,3,4\}$.
Let $p_j\in \pi(G)$ be an odd prime such that $p_i\ne p_j$. Since $K$ is solvable and $KP_j/K\cong P_j/(P_j\cap K)$ is a $p_j$-group, $KP_j$ is solvable. By Proposition \ref{Hall}, $KP_j$ contains a $\{p_i,p_j\}$-Hall subgroup, say $H$. Proposition \ref{odd-index} implies that $[H]\in \Delta(G)$.

The Claim 1 shows that $L/K\cong \A_5$ or $\A_6$.
Assume first that $L/K\cong \A_5$. It follows from $L/K=\soc(G/K)$ that $G/K\cong \A_5$ or $G/K\cong \S_5$, in particular,  $|\pi(G/K)|=3$. Since neither $\A_5$ nor $\S_5$ has a $2$-subgroup as a maximal subgroup(see the Atlas\cite[p.2]{atlas}), the maximal subgroup $M/K$ of $G/K$ cannot be a 2-group, and consequently $P_1 K$ is a proper subgroup of $M$. By Proposition \ref{odd-index} and the assumption that $|\Delta(G)|\le 6$, we obtain that
$\Delta(G)= \{ [P_1],[P_2],[P_3],[P_1 K], [M], [H]\}$. Again by Proposition \ref{odd-index}, we conclude that $|\pi(G)|=3$ and $|G|_{2'}$ is square-free, i.e., $a_2+a_3=2$. However, since $p_i\mid |K|$ and $|\pi(G/K)|=3$, it follows that $p_i^2\mid |G|$, which contradicts the fact that $|G|_{2'}$ is square-free.
Now assume that $L/K\cong \A_6$. Then $3^2\mid |G|$ and consequently $G$ contains a subgroup $Q$ of order 3 that is not a Sylow $3$-subgroup of $G$.
By Proposition \ref{odd-index} and the assumption that $|\Delta(G)|\le 6$, we obtain that
$\Delta(G)= \{ [P_1],[P_2],[P_3],[Q], [M], [H]\}$.
Applying Proposition \ref{odd-index} again, we get that $|\pi(G)|=3$ and $a_2+a_3=3$.
However, since $p_i\mid |K|$ and $|\A_6|_{2'}=3^2 \cdot 5$, it follows that $a_2+a_3\ge 4$, a contradiction. Hence Claim 3 holds.

Noting that $K$ is a normal 2-subgroup of $G$, we yield that $K\le \O_2(G)$.
On the other hand, since $P_1\le M$ and $K=\bigcap_{g\in G} M^g$, we conclude that $\O_2(G)=\bigcap_{g\in G} P_1^g\le K$ and hence $K=\O_2(G)$.

Finally, based on the preceding arguments, we have completed the proof of Lemma \ref{case (3)}.
\qed

For brevity, in the rest of this Section, we always let
$$\Psi=\{\PSL(2, q), \PSL(3,3), \allowbreak\PSL(3,5), \allowbreak \PSp(4,3), \allowbreak \PSU(3,3), \allowbreak \PSU(3,7) \mid q=25,27,49,p \},$$
where $p\ge 5$ is an odd prime.

\begin{lemma}\label{case (4)}
Suppose that the case (4) holds, that is $L/K\in \Psi$. Then $K=\O_2(G)$ and $\soc(G/K)\cong \PSL(2,q)$, where $q\in \{5, 7, 17, p\}$ and $p$ is an odd prime such that $p\notin \{5, 7, 17\}$ and $(p^2-1)_{2'}$ is a product of two distinct primes.
\end{lemma}
\proof
Suppose that $L/K\in \Psi$. Next, we will prove Lemma \ref{case (4)} by establishing five claims. Note that $L/K=\soc(G/K)$, $3\le |\pi(G)|\le 4$ and $|G|=2^{a_1} p_2^{a_2} p_3^{a_3} p_4^{a_4}$, where $a_2, a_3\ge 1$ and $a_4\ge 0$ are integers.

\textbf{Claim 1. }$L/K\cong \PSL(2, p)$, where $p\ge 5$ is an odd prime.

Assume first that $L/K\cong \PSL(2, 25)$. Then $\PSL(2,25)\lesssim G/K \lesssim \Aut(\PSL(2, 25))$.  A detailed check of the Atlas \cite[p.16]{atlas} shows that among all possible candidates for $G/K$, none of them contains a maximal subgroup that is a $2$-group.
It follows that $M/K$ is not a $2$-group and so $P_1$ is a proper subgroup of $M$ as $P_1 \le M$.
Noticing that $5^2 \mid |G|$, we see that $G$ contains a subgroup isomorphic to $\ZZ_5$, say $H$, which is not a Sylow $5$-subgroup of $G$.
By Proposition \ref{odd-index} and the assumption that $|\Delta(G)|\leqslant 6$, we have  $\Delta(G)=\{[P_1], [P_2],[P_3],[P_4],[M],[H]\}$. Again by Proposition \ref{odd-index}, we see that  $a_2+a_3+a_4=4$, and so $K$ is a $2$-group as $|L/K|_{2'}=|\PSL(2, 25)|_{2'}=3\cdot5^2\cdot13$.
By Proposition \ref{bray}, $\PSL(2,25)$ has a maximal subgroup isomorphic to $(\ZZ_5 \times \ZZ_5){:}\ZZ_{12}$, and so there exists a subgroup $M_1$ of $G$ such that $M_1/K \cong (\ZZ_5 \times \ZZ_5){:}\ZZ_{12}$ is solvable. Since both $M_1/K$ and $K$ are solvable, $M_1$ is solvable. According to Propositions \ref{Hall} and \ref{odd-index}, we conclude that $M_1$ contains a $\{3,5\}$-Hall subgroup say $Q$ and
$[Q]\in \Delta(G)$. It follows that $|\Delta(G)|\geqslant 7$, which contradicts the assumption that $|\Delta(G)|\le 6$. Therefore, the case where $L/K\cong \PSL(2,25)$ cannot occur.
For the case where $L/K \in \{\PSL(2, 49), \PSL(2, 27),\allowbreak \PSL(3, 3), \allowbreak \PSL(3, 5), \PSp(4,3),\allowbreak \PSU(3,3),\allowbreak \PSU(3,7)\}$, an entirely analogous argument shows that this cannot occur. Therefore, $L/K\cong \PSL(2, p)$ for $p\ge 5$ odd prime and hence Claim 1 holds.

\textbf{Claim 2. } $K$ is solvable.

Assume for a contradiction that $K$ is non-solvable. Then $3\leqslant |\pi(K)|\leqslant|\pi(G)|\leqslant 4$. It follows from $G/K$ is non-solvable that $|\pi(G/K)|\geqslant 3$. Let $p_i$ and $p_j$ be two distinct odd prime divisors of $|K|$ for $i, j\in\{2, 3, 4\}$. However, since $\pi(G/K)\subseteq\pi(G)$, it follows that $p_i^2\mid |G|$ or $p_j^2\mid |G|$. Without loss of generality, we set $p_i^2\mid |G|$. Hence $G$ contains a subgroup isomorphic to $\ZZ_{p_i}$, say $H_{p_i}$, which is not Sylow $p_i$-subgroup of $G$.
In what follows, we treat the cases $|\pi(G)| = 4$ and $|\pi(G)| = 3$ individually.

Assume first that $|\pi(G)| = 4$.
By Proposition \ref{odd-index} and the assumption that $|\Delta(G)|\le 6$, we conclude that $\Delta(G)=\{[H_{p_i}],[P_1],[P_2],[P_3],[P_4], [M]\}$.
Lemma \ref{rm} implies that $P_1\cong \ZZ_{2^{a_1}}$ or $\Q_{2^{a_1}}$.
We conclude from $G$ is non-solvable and Proposition \ref{pb} that $P_1\cong \Q_{2^{a_1}}$. Note that $K\cap P_1$ is a Sylow 2-subgroup of $K$ and $P_1K/K\cong P_1/(K\cap P_1)$ is a Sylow 2-subgroup of $G/K$. Since all Sylow 2-subgroups of $G/K\cong \PSL(2,p).o$ are noncyclic for $o\leqslant\Out(\PSL(2, p))$, that is $P_1/(K\cap P_1)$ is noncyclic, applying Observation \ref{observation}, one yields that $K\cap P_1$ is cyclic. It follows from Proposition \ref{pb} that $K$ has a normal $2$-complement and so $K$ is solvable, contradicting the assumption that $K$ is non-solvable, as desired.

Now suppose that $|\pi(G)|=3$. Then $\pi(K)=\pi(G)$ and hence $p_i^2p_j^2 \mid |G|$, which shows that $G$ possesses subgroups $H_1\cong \ZZ_{p_i}$ and $H_2\cong\ZZ_{p_j}$ that are not Sylow subgroups of $G$.
By Proposition \ref{odd-index} and the assumption that $|\Delta(G)|\le 6$, we obtain that
$\Delta(G)=\{ [H_1],[H_2],[P_1],[P_2],[P_3], [M]\}$,
Lemma \ref{rm} implies that $P_1\cong \ZZ_{2^{a_1}}$ or $\Q_{2^{a_1}}$.
Since $G$ is non-solvable, Proposition \ref{pb} gives that $P_1\cong \Q_{2^{a_1}}$.
By a similar argument as above, one can draw the conclusion that $K$ is solvable, contradicting the assumption that $K$ is non-solvable. Thus $K$ is solvable, and Claim 2 holds.

\textbf{Claim 3. } $K$ is a 2-group.

Suppose on the contrary that $K$ is not a $2$-group. Then $|K|$ has an odd prime divisor, say $p_i$ for $i\in \{2,3,4\}$. Let $p_j\in \pi(G)$ be such that $p_j\ne p_i$ where $j\in \{2,3,4\}$. It deduces from $K$ and $KP_j/K$ are solvable that $KP_j$ is solvable. By Proposition \ref{Hall}, $KP_j$ contains a $\{p_i,p_j\}$-Hall subgroup say $H$.
If $p_i\in \pi(G/K)$, let $R$ be a subgroup of $G$ of order $p_i$; otherwise let $R=P_4$.
By Proposition \ref{odd-index} and the assumption that $|\Delta(G)|\le 6$, we obtain that  $\Delta(G)=\{ [R],[P_1],[P_2],[P_3],[H],[M]\}$. Again by Proposition \ref{odd-index}, we can conclude that $a_2+a_3+a_4=3$, in particular, if $p_i\in \pi(G/K)$, then $a_4=0$. Since $|\pi(G/K)|\geqslant 3$ and $p_i\mid |K|$, it follows that $|K|_{2'}=p_i$ and $|H|=p_i p_j$.
 Moreover, since $P_1< P_1 K\le M$ and $P_1 K$ has odd index in $G$, Proposition \ref{odd-index} further gives that $[P_1K]\in \Delta(G)$ and hence $M=P_1 K$. Consequently $M/K=P_1K/K\cong P_1/(P_1\cap K)$ is a 2-group, which means that $G/K$ contains a 2-subgroup as a maximal subgroup. Since $\soc(G/K)=L/K\cong \PSL(2,p)$, $G/K\cong\PSL(2, p)$ or $\PGL(2, p)$, applying  Proposition \ref{bray}, we can draw the conclusion that $p-1$ or $p+1$ is a power of 2.

Assume first that $p-1$ is a power of $2$. Then $p-1=2^l$ for some positive integer $l\ge 2$, and so $(p+1)/2=2^{l-1}+1$ is odd.
Given that $|\PSL(2,p)|=p(p^2-1)/2$, $|K|_{2'}=p_i$ and $|\PGL(2,p):\PSL(2,p)|=2$, we see that $|G|_{2'}=p_i p (p+1)_{2'}$.
From $a_2+a_3+a_4=3$, it folllows that $(p+1)/2$ is a prime. Then by Lemma \ref{num3} (2), we see that  $p=5$, so $G\cong\PSL(2, 5)$ or $\PGL(2, 5)$. However, since neither $\PSL(2,5)$ nor $\PGL(2,5)$ has a maximal subgroup isomorphic to a $2$-group, contradicting the conclusion that $M/K$ is a 2-group (note that $M/K$ is a maximal subgroup of $G/K$). Therefore $p-1$ is not a power of $2$.

Suppose that $p+1$ is a power of $2$. Then $|G|_{2'}=p_i p (p-1)_{2'}$ and $(p-1)_{2'}>1$ as $G$ is non-solvable.
Assume that $|\pi(G)|=4$. Then there exists another $p_k\in \pi(G)$ such that $p_k\ne p_i, p_j$ and $k\in\{2,3,4\}$. It deduces from $K$ and $KP_k/K$ are solvable that $KP_k$ is solvable. By Proposition \ref{Hall}, $KP_j$ contains a $\{p_i,p_k\}$-Hall subgroup say $Q$.
Proposition \ref{odd-index} yields that $[Q]\in \Delta(G)$ and hence $|\Delta(G)|\geqslant 7$, which contradicts the assumption that $|\Delta(G)| \le 6$.
Thus $|\pi(G)|=3$. It follows that $(p-1)/2$ is a power of a prime. Lemma \ref{num3} (1) implies that $p=7$ and hence $G/K\cong \PSL(2,7)$ or $\PGL(2,7)$. From Proposition \ref{bray}, $G/K$ contains a subgroup isomorphic to $\ZZ_7{:}\ZZ_3$, i.e.,
there exists a subgroup $M_1$ of $G$ such that $M_1/K\cong \ZZ_7{:}\ZZ_3$. Since $K$ is solvable, one yields that $M_1$ is solvable and hence $M_1$ contains a $\{3,7\}$-Hall subgroup, say $H_1$ (see Proposition \ref{Hall}, for example). Noting that $p_i\in\{3,7\}$ and $p_i\mid |K|$, we have $|H_1|=3^2\cdot 7$ or $3\cdot 7^2$. Consequently $H\ncong H_1$ as $|H|=p_i p_j=21$.
Applying Proposition \ref{odd-index}, we yield that $[H_1]\in \Delta(G)$ and hence $|\Delta(G)|\geqslant 7$, which contradicts the assumption that $|\Delta(G)| \le 6$.
Therefore $K$ is a 2-group and Claim 3 holds.

\textbf{Claim 4. } $K=\O_2(G)$.

Since $K$ is a normal $2$-subgroup of $G$, we conclude that $K\le \O_2(G)$.
On the other hand, since $K = \bigcap_{g \in G} M^g$ and $P_1 \le M$, it follows that $\O_2(G) = \bigcap_{g \in G} P_1^g \leqslant\bigcap_{g \in G} M^g=K$. Therefore, $K = \O_2(G)$, which establishes Claim 4.

\textbf{Claim 5. }$\soc(G/\O_2(G))\cong \PSL(2,q)$, where $q\in \{5, 7, 17, p\}$, $p$ is an odd prime such that $p\notin \{5, 7, 17\}$ and $(p^2-1)_{2'}=p_2p_3$ (i.e., $|G|_{2'}=p_2 p_3 p_4$).

Derived from $K=\O_2(G)$ and $\soc(G/K)=L/K\cong \PSL(2,p)$, one deduces that $|G|_{2'} = p(p^2-1)_{2'}$. Note that $3\leqslant |\pi(G)|\leqslant 4$.
First, assume that $|\pi(G)| = 3$. It follows that $(p^2-1)_{2'}$ is a power of a prime, so either $p-1$ or $p+1$ is a power of $2$, and the other is twice a power of an odd prime. By Lemma \ref{num3}, we obtain that $p\in\{ 5, 7, 17\}$. So in the case where $|\pi(G)|=3$, we have   $\soc(G/\O_2(G))\cong\PSL(2,p)$, where $p\in\{5, 7, 17\}$, as desired.

Now suppose that $|\pi(G)| = 4$. We assert that $p = p_4$ and $(p^2-1)_{2'}=p_2p_3$, i.e., $|G| = 2^{a_1} p_2p_3 p$. Since $G/K\cong \PSL(2,p)$ or $\PGL(2,p)$ and $K=\O_2(G)$, it follows that $|G|_{2'} = p(p^2-1)_{2'}=p(p-1)_{2'}(p+1)_{2'}$, and so $p=p_4$ as $p_1<p_2<p_3<p_4$, and $(p^2-1)_{2'}=p_2^{a_2} p_3^{a_3}$ for $a_2$ and $a_3$ positive integers.
Assume that $a_i\ge 2$ for some $i\in\{2,3\}$.  Hence $G$ has a subgroup isomorphic to $\ZZ_{p_i}$ which is not a Sylow $p_i$-subgroup of $G$, say $H$. Noting that $(p-1,p+1)=2$, either $p_i^2\mid (p-1)$ or $p_i^2\mid (p+1)$.
Suppose first that $p_i^2\mid (p-1)$.  Since $L/K\cong \PSL(2,p)$,  Proposition \ref{bray} implies that $L$ contains a subgroup $M_1$ containing $K$ such that $M_1/K \cong \ZZ_p{:}\ZZ_{\frac{p-1}{2}}$. It is easy to see that $M_1/K$ has subgroups $H_1/K \cong \ZZ_p{:}\ZZ_{p_i}$ and $H_2/K \cong \ZZ_p{:}\ZZ_{p_i^2}$ (with $K \le H_1, H_2 \le M_1$). Since $K$ is a $2$-group, it follows that both $H_1$ and $H_2$ are solvable. And further, applying  Proposition \ref{Hall}, one yields that $H_j$ contains a $\{p,p_i\}$-Hall subgroup say $Q_j$ for each $j\in\{1,2\}$ such that $|Q_1|\neq |Q_2|$. Therefore, Proposition \ref{odd-index} shows that $\Delta(G)\supseteq \{[P_1],[P_2],[P_3],[P_4],[H],[Q_1],[Q_2]\}$, which contradicts the assumption that $|\Delta(G)|\leqslant 6$. Therefore, $p_i^2\nmid (p-1)$.

Now that $p_i^2\mid (p+1)$. Let $p_l\in\{p_2,p_3\}$ be such that $p_l\neq p_i$.  Assume that $p_l\mid (p+1)$. Since $L/K\cong \PSL(2,p)$,  Proposition \ref{bray} implies that $L$ contains a subgroup $M_2$ containing $K$ such that $M_2/K \cong \D_{p+1}$, and hence $M_2/K$ has subgroups $H_3/K$ isomorphic to $\ZZ_{p_i p_l}$ and $H_4/K$ isomorphic to $\ZZ_{p_i^2 p_l}$, where $K \le H_3, H_4 \le M_2$. By the same reasoning as above, we can deduce a contradiction. Hence $p_l\nmid (p+1)$ and $p_l\mid (p-1)$. Noting that both $p-1$ and $p+1$ are not a power of $2$, Proposition \ref{bray} implies that $G/K$ has no $2$-subgroup as a maximal subgroup, i.e., $M/K$ is not a $2$-group.
Again by Proposition \ref{bray}, $L$ contains a subgroup $M_3$ containing $K$ such that $M_3/K \cong \ZZ_p{:}\ZZ_{\frac{p-1}{2}}$. In particular, $M_3$ is solvable as $K$ is a $2$-group.
It follows from Proposition \ref{Hall} that $M_3$ contains a $\{p,p_l\}$-Hall subgroup, say $D$. Applying Proposition \ref{odd-index}, we yield that $\Delta(G)\supseteq \{[P_1],[P_2],[P_3],[P_4],[H],[M],[D]\}$, contradicting the assumption that $|\Delta(G)|\leqslant 6$. Therefore, $(p^2-1)_{2'}=p_2p_3$, i.e., $|G| = 2^{a_1} p_2 p_3 p$, and Claim 5 holds.

Combining the five claims above, we have proved Lemma \ref{case (4)}. \qed

In the following, we always denote
$$\Pi=\{\PSL(2, 2^f), \PSL(3, 4), \PSL(3, 8), \PSp(4,4), \PSU(3,4), {^2\!B_2}(8), {^2\!B_2}(32)\mid f\geqslant 3 \}.$$

\begin{lemma}\label{case (5)}
Suppose that the case (5) holds.  Then $K=\O_2(G)$ and $G/K\cong {^2\!B_2}(8)$.
\end{lemma}
\proof Suppose that the case (5) holds, i.e., $L/K\in \Pi$ where $f=3, 4$ or $f>3$ is a prime, and $(M/K)\cap (L/K)$ is a parabolic subgroup of $L/K$. Recall that $|G|=2^{a_1} p_2^{a_2} p_3^{a_3} p_4^{a_4}$ and $a_2+a_3+a_4\le 5$, where $a_2, a_3\ge 1$ and $a_4\ge 0$ are integers.

\textbf{Claim 1.} $L/K \ncong \PSL(2,2^f)$ for $f>3$ prime.

Assume that $L/K \cong \PSL(2,2^f)$ and $f>3$ is a prime. It follows from $\Out(\PSL(2, 2^f))\cong \ZZ_f$ and $\soc(G/K)=L/K$ that $G/K\cong \PSL(2, 2^f)$ or $\PSL(2, 2^f){:}\ZZ_f$. In particular, $|G/K:L/K|$ is odd as $f$ is odd. Note that $3\leqslant|\pi(G)|\leqslant4$.
Suppose first that $|\pi(G)|=3$. Since $(2^f-1, 2^f+1)=1$ and $2^f(2^f-1)(2^f+1)\mid |G/K|$, we yield that $2^f+1=3^t$ for some positive integer $t$. Simple verification shows that $f=5,7$ gives no solution, and hence $f\ge 11$ and $t>5$, which contradicts the assumption that $a_2+a_3+a_4\leqslant 5$.

Now suppose that $|\pi(G)|=4$. Since $(M/K)\cap (L/K)$ is a parabolic subgroup of $L/K$, it follows that $(2^f-1)\mid |M/K|$ (see\cite[p.180]{K-L} and \cite[p.46]{wilson} for example) and hence $M$ is not a $2$-group. By Proposition \ref{odd-index}, we have $\Delta(G)\supseteq \{[P_1], [P_2],[P_3],[P_4],[M]\}$. Since $|\Delta(G)|\le 6$, using Proposition \ref{odd-index} again, we obtain that $a_2+a_3+a_4\le 4$, and so $|\pi(K)|\le 2$, i.e., $K$ is solvable.
By \cite[p.213, Hauptsatz 8.27]{hu1}, $L/K\cong\PSL(2,2^f)$ contains a subgroup isomorphic to $\ZZ_{2^f+1}$ say $M_1/K$, where $M_1\leqslant L$. We assert that there exists $p_i\in \pi(G)$ such that $p_i \mid (2^f+1)$ and $p_i\ne 3$, where $i\in\{2, 3, 4\}$. If on the contrary $2^f+1=3^t$ for some positive integer $t$, then since $f>3$ is a prime, it follows that $f\ge 11$ and hence $t>5$, which contradicts the assumption that $a_2+a_3+a_4\le 5$.
Thus there exists $p_i\in \pi(G)$ such that $p_i \mid (2^f+1)$ and $p_i\ne 3$, where $i\in\{2, 3, 4\}$. The solvability of $K$ and $M_1/K$ leads to the solvability of $M_1$, combining with  Proposition \ref{Hall}, one yields that $M_1$ contains a $\{3,p_i\}$-Hall subgroup say $H$. By Proposition \ref{odd-index} and the assumption that $|\Delta(G)|\le 6$, we have that $\Delta(G)=\{[P_1],[P_2],[P_3],[P_4], [M], [H]\}$. Quoting the result of Lemma \ref{rm}, one yields that $P_1$ is either a cyclic or a generalized quaternion $2$-group, and consequently $P_1K/K\cong P_1/P_1\cap K$ is cyclic, dihedral or generalized quaternion. Note that $P_1K/K\in\Syl_2(G/K)$ and $|G/K:L/K|$ is odd. However, every Sylow $2$-subgroup of $\PSL(2,2^f)$ is isomorphic to the  elementary abelian 2-subgroup $\ZZ_2^f$, which is a contradiction. Therefore, the case $L/K\cong \PSL(2, 2^f)$ for prime $f>3$ cannot occur and Claim 1 holds.

\textbf{Claim 2.} $L/K\notin \{\PSL(2, 8)$, $\PSL(2, 16)$, $\PSL(3, 4)$, $\PSL(3, 8)$, $\allowbreak\PSL(4, 2)$, $\allowbreak\PSp(4,4)$, $\PSU(3,4)$, ${^2\!B_2}(32)\}$.

Based on a method analogous to that of the previous paragraph, we conclude that Claim 2 holds.

\textbf{Claim 3.} If $L/K\cong {^2\!B_2}(8)$, then $G/K\cong {^2\!B_2}(8)$ and $K=\O_2(G)$.

Now assume that $L/K\cong {^2\!B_2}(8)$, where $(M/K) \cap (L/K)$ is a parabolic subgroup of $L/K$. Then $G/K\cong{^2\!B_2}(8)$ or ${^2\!B_2}(8){:}\ZZ_3$ as $\Out({^2\!B_2}(8))\cong \ZZ_3$.
Moreover, since $|{^2\!B_2}(8)| = 2^{6}\cdot 5\cdot 7\cdot 13$ and $|\pi(G)| \le 4$, we conclude  that $G/K$ must be isomorphic to ${^2\!B_2}(8)$, as desired.

Since $(M/K) \cap (L/K)$ is a parabolic subgroup of $L/K$, it follows that $8^2(8-1)\mid |M/K|$ (see\cite[p. 180]{K-L} and \cite[p.46]{wilson} for example). Consequently both $M$ and $M/K$ are not 2-groups. We assert that $K$ is a 2-group.
Suppose on the contrary  that $K$ is not a 2-group. Then $P_1<P_1 K< M$. Note that $P_1 K$ has odd index in $G$. By Proposition \ref{odd-index} and the assumption that $|\Delta(G)|\le 6$, we have $\Delta(G)=\{[M],[P_1 K],[P_1],[P_2],[P_3],[P_4]\}$.
Derived from Lemma \ref{rm}, one yields that $P_1\cong \ZZ_{2^{a_1}}$ or $\Q_{2^{a_1}}$. Since $G$ is non-solvable, Proposition \ref{pb} gives $P_1\cong \Q_{2^{a_1}}$. It follows that $P_1 K/K$, as the Sylow $2$-subgroup of $G/K$, is dihedral or generalized quaternion. However, every Sylow $2$-subgroup of ${^2\!B_2(8)}$ is isomorphic to $\ZZ_2^3.\ZZ_2^3$, a contradiction. Therefore $K$ is a 2-group and the assertion holds.

Since $K\lhd G$, we see that $K\le \O_2(G)$.
On the other hand, since $K=\bigcap_{g\in G} M^g$ and $P_1\le M$, it follows that $\O_2(G)=\bigcap_{g\in G} P_1^g\le K$ and hence $K=\O_2(G)$ and Claim 3 holds.

Based on the results of the previous three Claims, we have finally completed the proof of Lemma \ref{case (5)}.
\qed

With the preparations above, we now prove Theorem \ref{main3}.

\textbf{The proof of Theorem \ref{main3}.} Suppose that $G$ is a finite non-solvable group and $|\Delta(G)|\le 6$. Let $M$ be a maximal subgroup which contains a Sylow $2$-subgroup of $G$, and let $K=\bigcap_{g\in G} M^g$. According to the discussion in the paragraph following Proposition  \ref{forK}, one case of the list in (\ref{star}) holds. Applying Lemmas \ref{case (1)} and \ref{case (2)}, we can exclude the cases (1) and (2) in (\ref{star}). Noting that $\A_5\cong \PSL(2,5)$ and $\A_6\cong \PSL(2,9)$, for the case (3) or (4) in (\ref{star}), Lemmas \ref{case (3)} and \ref{case (4)} lead to the conclusion that $\soc(G/\O_2(G))\cong \PSL(2,q)$ for $q\in \{5, 7, 9, 17, p\}$, where $p$ is an odd prime such that $p\notin \{5, 7, 17\}$ and $(p^2-1)_{2'}$ is a product of two distinct primes, just as the case (2) of Theorem \ref{main3}. Now assume that the case (5) in (\ref{star}) holds. Then Lemma \ref{case (5)} implies that $G/\O_2(G)\cong {^2\!B_2}(8)$, which completes the proof of Theorem \ref{main3}.
\qed

\section{Proof of Theorem \ref{sol}} \label{Theorem 1.7}

In this section, we will prove Theorem \ref{sol}, and throughout we let $G$ be a finite solvable group.

\noindent{\bf Proof of Theorem \ref{sol}.}
Since $G$ is a finite solvable group, derived from Lemma \ref{os}, one yields that $|\Delta(G)| \ge 2^{|\pi(G)|} - 2$, as desired. In the following paragraphs, we will prove the equivalence condition for the equal sign to hold.

Firstly, we will demonstrate the necessity. Assume first that $G$ is solvable and $|\Delta(G)| = 2^{|\pi(G)|}-2$. If $|\pi(G)|=1$, i.e., $G$ is a $p$-group, then $|\Delta(G)| = 0$ and the conclusion can be directly drawn from Proposition \ref{chen2}. So in what follows, we assume that $|\pi(G)| \geqslant 2$. However, since $G$ is solvable, applying Propositions \ref{odd-index} and \ref{Hall}, there are $2^{|\pi(G)|}-2$ non-conjugate Hall subgroups of $G$, and further, we can deduce from $|\Delta(G)| = 2^{|\pi(G)|}-2$ that every subgroup perfect code of $G$ is its Hall subgroup. Applying Proposition \ref{odd-index}, we conclude that $|G|_{2'}$ is square-free.
Further, if $2\mid |G|$, then Lemma \ref{rm} shows that every Sylow 2-subgroup of $G$ is isomorphic to a cyclic $2$-group or generalized quaternion $2$-group, as desired.

Next, we will demonstrate the sufficiency. Assume that $|G|_{2'}$ is square-free and that every Sylow 2-subgroup of $G$ is isomorphic either to a cyclic $2$-group or to a generalized quaternion $2$-group. By Lemma \ref{cg}, we yield that every subgroup perfect code of $G$ has odd order or odd index. Since $|G|_{2'}$ is square-free, every subgroup of $G$ with odd order or odd index is a Hall subgroup. It follows that $\Delta(G)$ is exactly the set of all conjugacy classes of Hall subgroups of $G$. Since $G$ is solvable, we conclude that $|\Delta(G)| = 2^{|\pi(G)|} - 2$, as desired.
\qed

In Theorem \ref{main2}.(2), when the group $G$ is solvable, its structure relies on the classification of $2$-groups having very few conjugacy classes of involutions. The classification of $2$-groups with few involutions is itself a challenging problem; see, for example, \cite{janko,konv} and the detailed account in \cite[p.368-395]{berk}.
This leads naturally to the following questions.
\begin{question}
Determine all $2$-groups whose involutions lie in at most two conjugacy classes.
\end{question}

\begin{question}
Determine all groups whose involutions are all conjugate.
\end{question}

\begin{appendices}
\section{Examination}\label{Appendix}
In this appendix, we present the examination of Propositions \ref{pi+1}, \ref{d=5} and \ref{forK}.

\begin{proposition}\cite[Theorem 1]{ls}\label{odd-degree}
Let $G$ be a primitive permutation group of odd degree $n$ on a set $\Sigma$ and let $H = G_\alpha$, where $\alpha \in \Sigma$.
 \begin{enumerate}[label=(\alph*),font=\normalfont]
  \item Either $(\ZZ_p)^d \unlhd G \leq \mathrm{AGL}(d,p)$ for some odd prime $p$, or $L^m \lhd G \leq G_0 \wr S_m$, where $G_0$ is a primitive group of odd degree $n_0$ with simple socle $L$, and the wreath product has the product action of degree $n = n_0^m$.

  \item If $G$ has simple socle $L$ then $G$ and $H$ are known, and one of (I), (II) and (III) below holds:

\begin{enumerate}[label=(\Roman*),font=\normalfont]
      \item $L$ is $A_c$, an alternating group; $H$ is $(\S_k \times \S_{c-k}) \cap G$ ($1 \leq k < \frac{1}{2} c$), or $H$ is $(\S_a \wr \S_b) \cap G$ ($ab = c, a > 1, b > 1$), or $G$ is $\A_7$, of degree 15,

      \item $L$ is sporadic,

      \item $L = L(q)$, a simple group of Lie type over $\F(q)$; in particular, if $q$ is even, then $H \cap L$ is a parabolic subgroup of $L$.
    \end{enumerate}
    \end{enumerate}
\end{proposition}

\begin{proof}[{\bf The proofs of Proposition \ref{pi+1}}]
Remind the assumption that $G$ is a primitive permutation group of odd degree on a finite set $\Omega$, $\alpha\in \Omega$ and that $G_\alpha$ is a 2-group. We also assume that $|G|=p_1^{a_1} p_2^{a_2}\cdots p_n^{a_n}$ and $a_2+\cdots+a_n\le n$, where $p_i\in \pi(G)$ and $a_i\ge 1$ for all $i\in \{1, \dots, n\}$.
By Proposition \ref{odd-degree}, we can conclude that $G$ is one of the following:

\begin{enumerate} [font=\normalfont]
    \item $G\cong \ZZ_p^d{:}G_\alpha$, where $p$ is an odd prime, $d\le 2$ and $G_\alpha$ is an irreducible subgroup of $\GL(d,p)$;

    \item $L\cong \A_5, \A_6, \A_7, \A_8$;

    \item $L$ is a classical simple group over $\F_q$, where $q$ is odd;

    \item $L$ is a classical simple group over $\F_q$, where $q$ is even and $G_\alpha\cap L$ is a parabolic subgroup of $L$.
\end{enumerate}

\textbf{Claim 1.} Suppose that $L\in\{ \A_5, \A_6, \A_7, \A_8\}$. Then $G\cong \PGL(2,9)$, $\M_{10}$ or $\mathrm{P\Ga L}(2,9)$.

By the N/C theorem, we get that $\A_i\le G \le \Aut(\A_i)$ for $i\in \{5,6,7,8\}$.
Since $G$ is primitive on $\Omega$, it follows that $G_{\alpha}$ is a maximal subgroup of $G$.
By the Atlas\cite[p.2,p.4,p.10 and p.22]{atlas}, we conclude that $G$ has a maximal subgroup that is a 2-group only when $L \cong \A_6$. Therefore $L\cong \A_6\cong \PSL(2,9)$,
Since $L=\soc(G)$, we obtain that $\PSL(2,9)\lesssim G\lesssim \mathrm{P\Ga L}(2,9)$. By the Atlas\cite[p.4]{atlas} and since $G$ has a maximal subgroup is a 2-group, we conclude that $G\cong \PGL(2,9)$, $\M_{10}$ or $\mathrm{P\Ga L}(2,9)$,  and hence Claim 1 holds.

\textbf{Claim 2.} Suppose that $L$ is a classical simple group over $\F_q$, where $q=p^f$ is odd and $p$ is a prime.
Then $G\in \{\PGL(2,9),\M_{10}, \mathrm{P\Ga L}(2,9),\PSL(2,p), \PGL(2,p) \}$ where one of $p-1$ or $p+1$ is a power of 2.

Suppose first that $L\cong \PSL(m,p^f)$, where $f\ge 1$ and $m\ge 2$.
Since $|\PSL(m,q)|=\frac{1}{(m,q-1)}q^{\frac{m(m-1)}{2}}\prod_{i=2}^{m}(q^i-1)$, and
$a_{2}+\dots+a_{n} \le n$, we can conclude that $m= 2$ and $q\mid p^2$.
Assume that $q=p^2$.
By \cite[Theorem 1.1]{gi2007}, either $G_\alpha \cap L$ is a maximal subgroup of $L$ or $G\in \{\PGL(2,9),\M_{10}, \mathrm{P\Ga L}(2,9)\}$.
Suppose that $G_\alpha \cap L$ is a maximal subgroup of $L$. Since $G_{\alpha}$ is a 2-group, it follows $G_\alpha \cap L$ also is a 2-group. By Proposition \ref{bray}, we conclude that $p^2-1$ or $p^2+1$ is a power of $2$. Moreover, Proposition \ref{num2} implies that $p=3$.
It follows that $L\cong \PSL(2,9)$. Since $L=\soc(G)$, we obtain that $\PSL(2,9)\lesssim G\lesssim \mathrm{P\Ga L}(2,9)$. By the Atlas\cite[p.4]{atlas} and since $G$ has a maximal subgroup is a 2-group, we conclude that $G\cong \PGL(2,9)$, $\M_{10}$ or $\mathrm{P\Ga L}(2,9)$.
Now assume that $q=p$. Since $L=\soc(G)$, we have that $G\cong \PSL(2,p)$ or $\PGL(2,p)$.
Since $G_{\alpha}$ is a 2-group, Proposition \ref{bray} implies that one of $p-1$ or $p+1$ is a power of 2. Therefore, $G\in \{\PGL(2,9),\M_{10}, \mathrm{P\Ga L}(2,9),\PSL(2,p), \PGL(2,p) \}$ where one of $p-1$ or $p+1$ is a power of 2, as desired.

Suppose that $L\cong \PSp(2m,p^f)$, where $f\ge 1$ and $m\ge 1$.
Since $ |\PSp(2m, p^f)| = \frac{1}{2} p^{f m^2} \prod_{i=1}^{m} (p^{2fi} - 1)$ and $a_{2}+\dots+a_{n} \le n$,  we can conclude that $m=1$ and $f\in\{1, 2\}$.
Thus $L\cong \PSp(2,p^f)$. Note that $\PSp(2,p^f)\cong \PSL(2,p^f)$. As previously discussed, it follows that $G\in \{\PGL(2,9),\M_{10}, \mathrm{P\Ga L}(2,9),\PSL(2,p), \PGL(2,p) \}$ where one of $p-1$ or $p+1$ is a power of 2.

Suppose that $L\cong \PSU(m,p^f)$, where $q=p^f$ and $m\ge 1$.
Since $|\PSU(m,q)| = \frac{1}{(m,q+1)} q^{\frac{m(m-1)}{2}} \prod_{i=2}^{m} (q^{i} - (-1)^{i})$ and $a_{2}+\dots+a_{n} \le n$,  we can conclude that $m=2$ and $f\in \{1,2\}$.
Thus $L\cong \PSU(2,p^f)\cong \PSL(2,p^f)$, where $f\in \{1,2\}$. As previously discussed, it follows that $G\in \{\PGL(2,9),\M_{10}, \mathrm{P\Ga L}(2,9),\PSL(2,p), \PGL(2,p) \}$ where one of $p-1$ or $p+1$ is a power of 2.

Suppose that $L\cong \Omega^+(2m+1,q)$,  $P\Omega^+(2m,q)$ or $P\Omega^-(2m,q)$, where $q=p^f$.
Since $|\mathrm{\Omega}(2m+1,q)| = \frac{1}{(2,q-1)} q^{m^2} \prod_{i=1}^{m} (q^{2i} - 1)$ ,
$|\mathrm{P\Omega}^+(2m,q)| = \frac{1}{(4,q^m-1)} q^{m(m-1)} (q^m - 1) \prod_{i=1}^{m-1} (q^{2i} - 1)$, and $|\mathrm{P\Omega}^-(2m,q)| = \frac{1}{(4,q^m+1)} q^{m(m-1)} (q^m + 1) \prod_{i=1}^{m-1} (q^{2i} - 1)$, and $a_{2}+\dots+a_{n} \le n$,
we conclude that $m\le 2$ and hence $2m\le 4$ and $2m+1\le 5$.
Since all orthogonal simple groups of dimension less than 8 are isomorphic to either $\PSL(m_1,q)$, $\PSp(m_1,q)$ or $\PSU(m_1,q)$ for some suitable integer $m_1\le 6$, this case has already been dealt with in the preceding discussion.
In conclusion, $G\in \{\PGL(2,9),\M_{10}, \mathrm{P\Ga L}(2,9),\PSL(2,p), \PGL(2,p) \}$ where one of $p-1$ or $p+1$ is a power of 2. The Claim 2 holds.

\textbf{Claim 3.} Suppose that $L$ is a classical simple group over $\F_q$, where $q$ is even and $G_\alpha\cap L$ is a parabolic subgroup of $L$. Then $L\cong \PSL(3,2)\cong \PSL(2,7)$.

Suppose that $L\cong \PSL(m,q)$, where $q=2^f$ and $m\ge 2$.
If $m\ge 6$, then $(q^2-1)(q^3-1)(q^4-1)(q^5-1)(q^6-1)/(m,q-1)=(q^2-1)^2 (q^3-1)^2 (q^2+1)(q^5-1)(q^3+1)/(m,q-1)\mid |G|$, contradicting the assumption that $a_{2}+\dots+a_{n} \le n$.
Thus $m\le 5$.
Assume that $m=4,5$. Then $(q^2-1)(q^3-1)(q^4-1)=(q-1)^3 (q+1)^2 (q^2+q+1)(q^2+1)\mid |G|$. Since $a_{2}+\dots+a_{n} \le n$, we can see that $q-1=1$ and so $q=2$.  However, by the Atlas\cite[p.22 and p.70]{atlas}, we conclude that $G$ has no maximal subgroup that is a 2-group when $L\cong \PSL(4,2)$ or $\PSL(5,2)$. Thus $m\ne 4,5$.
Assume that $m=3$.
Since every parabolic subgroup contains a Borel subgroup, the order of the Borel subgroup of $\PSL(3,2^f)$, which is divisible by $q-1$, must divide the order of the parabolic subgroup $L \cap G_\alpha$. However, $L \cap G_\alpha$ is a 2-group. Hence $q-1$ is a power of 2. As $q = 2^f$, this forces $f = 1$ and thus $2^f = 2$ and $L \cong \PSL(3,2)\cong \PSL(2,7)$.
Now assume that $m = 2$. Observe that $L \cap G_\alpha$ is a parabolic subgroup of $L$ and is also a $2$-group. As previously discussed for the case of $m = 3$, we obtain that $f = 1$.
However, when $f = 1$, we have $q = 2$ and $\PSL(2, 2) \cong \S_3$, which is not simple. This contradicts the assumption that $L$ is simple.

Suppose that $L\cong \PSp(2m,q)$, where $q=2^f$ and $m\ge 1$.
If $m\ge 3$, then $(q^2-1)(q^4-1)(q^6-1)=(q^2-1)^3(q^2+1)(q^4+q^2+1)\mid |G|$, contradicting the assumption that $a_{2}+\dots+a_{n} \le n$.
Thus $m\le 2$.
If $m=2$, then since $(q^2-1)(q^4-1)=(q-1)^2(q+1)^2(q^2+1)\mid |G|$ and $a_{2}+\dots+a_{n} \le n$, we can conclude that $q-1=1$ and so $q=2$. However, since $\PSp(4,2) \cong \S_6$ is not simple and so $L\ncong \PSp(4,2)$. Consequently $m=1$ and  $L \cong \PSp(2,2^f)\cong \PSL(2,2^f)$ with $f \ge 2$.
As previously discussed for the case of $\PSL(2,2^f)$, we obtain that $f = 1$, which contradicts the assumption that $L$ is simple. Therefore $L\ncong \PSp(2m,q)$.

Suppose that $L\cong \PSU(m,q)$, where $q=2^f$ and $m\ge 1$.
If $m\ge 5$, then $(q^2-1)(q^3+1)(q^4-1)(q^5+1)/(m,q+1)=(q-1)^2 (q+1)^4 (q^2-q+1) (q^2+1)(q^4 -q^3+q^2-q+1)/(m,q+1)\mid |G|$,
contradicting the assumption that $a_{2}+\dots+a_{n} \le n$.
Thus $m\le 4$. If $m=4$, then $(m,q+1)=1$ and hence $(q+1)^3 \mid |G|$, which contradicts the assumption that $a_{2}+\dots+a_{n} \le n$.
Suppose that $m=3$.
Since every parabolic subgroup contains a Borel subgroup, the order of the Borel subgroup of $\PSU(3,2^f)$, which is divisible by $(2^{2f}-1)/(3, 2^f+1)$, must divide the order of the parabolic subgroup $L \cap G_\alpha$. Since $L \cap G_\alpha$ is a 2-group, it follows that  $2^{2f}-1=3$, and hence $f=1$.
But $\PSU(3,2)$ is not simple, so $L\ncong \PSU(3,2)$.
Therefore $m=2$ and $L\cong \PSU(2,2^f)\cong \PSL(2,2^f)$. As previously discussed for the case of $\PSL(2,2^f)$, we obtain that $f = 1$, which contradicts the assumption that $L$ is simple.  Therefore $L\ncong \PSU(m,q)$.

Suppose that $L\cong \Omega^+(2m+1,q)$, $P\Omega^+(2m,q)$ or $P\Omega^-(2m,q)$.
Recall that $|\mathrm{P\Omega}(2m+1,q)| = \frac{1}{(2,q-1)} q^{m^2} \prod_{i=1}^{m} (q^{2i} - 1)$,
$|\mathrm{P\Omega}^+(2m,q)| = \frac{1}{(4,q^m-1)} q^{m(m-1)} (q^m - 1) \prod_{i=1}^{m-1} (q^{2i} - 1)$, and $|\mathrm{P\Omega}^-(2m,q)| = \frac{1}{(4,q^m+1)} q^{m(m-1)} (q^m + 1) \prod_{i=1}^{m-1} (q^{2i} - 1)$.
If $m\ge 4$, then $(q^2-1)(q^4-1)(q^6-1)= (q^2-1)^3 (q^2+1) (q^4+q^2+1)\mid |G|$, which contradicts the assumption that $a_{2}+\dots+a_{n} \le n$.
Thus $m\le 3$ and hence $2m\le 6$ and $2m+1\le 7$. Since all orthogonal simple groups of dimension less than 8 are isomorphic to either $\PSL(m_1,q)$, $\PSp(m_1,q)$ or $\PSU(m_1,q)$ for some suitable integer $m_1\le 6$, this case has already been dealt with in the preceding discussion.

In conclusion, $L\cong \PSL(2,7)$ and Claim 3 holds.

According to Claims 1--3 and $7+1=8$, one of the following holds:
 \begin{enumerate} [font=\normalfont]

            \item $G\cong \ZZ_p {:} \ZZ_{2^d}$, where prime $p \geqslant 3$ and $2^d\mid(p-1)$;

            \item $G\cong \PGL(2,9), \M_{10}, \PGammaL(2,9)$;

            \item $G\cong \PSL(2, p)$ or $\PGL(2,p)$, where prime $p\ge 5$, one of $p-1$ and $p+1$ is a power of $2$.
\end{enumerate}
\end{proof}

\begin{proposition}\label{d=5}
Let $G$ be a primitive group on a finite set $\Omega$ with $|\Omega|$ is odd. Suppose that  $|G|=2^{a_1} p_2^{a_2} p_3^{a_3} p_4^{a_4} $ for $a_2+a_3+a_4\leq 5$ and $a_i\ge 0$. Let $L=\soc(G)$ and let $\alpha\in \Omega$. Then one of the following holds.
\begin{enumerate} [font=\normalfont]
    \item $G\cong \ZZ_p^d{:} G_{\alpha}$, where $p$ is an odd prime, $d\le 4$ and $G_{\alpha}$ is an irreducible subgroup of $\GL(d,p)$;

    \item $T^2\lhd G \leq G_0 \wr \S_2$, where $G_0$ is a primitive group of odd degree $n_0$ with non-abelian simple socle $T$, and the wreath product has the product action of degree $n=n_0^2$;

    \item $L\cong \A_5, \A_6, \A_7$, $\A_8$, $\M_{11}$ or $\M_{12}$;

    \item $L\cong \PSL(2, p), \PSL(2, 25), \PSL(2, 49), \PSL(2,27)$, $\PSL(3,3)$, $\PSL(3,5)$, $\PSp(4,3)$, $\PSU(3,3)$ or $\PSU(3,7)$, where $p\ge 5$ is an odd prime;

    \item $L\cong \PSL(2, 2^f)$, $\PSL(2,16)$, $\PSL(3, 4)$, $\PSL(3, 8)$, $\PSp(4,4)$, $\PSU(3,4)$, ${^2\!B_2}(2^{3})$ or ${^2\!B_2}(2^{5})$, where $f\ge 3$ is a prime, and $L\cap G_{\alpha}$ is a parabolic subgroup of $L$.
\end{enumerate}
\end{proposition}

\begin{proof}
Remind the assumption that $G$ is a primitive group of odd degree on a finite set $\Omega$, and $|G|=2^{a_1} p_2^{a_2} p_3^{a_3} p_4^{a_4}$, where $a_2+a_3+a_4\leq 5$ with $a_i\ge 0$ and $L=\soc(G)$. Let $\alpha\in \Omega$.
By Proposition \ref{odd-degree}, one of the following holds:
\begin{enumerate} [font=\normalfont]
    \item $G\cong \ZZ_p^d{:} H$, where $p$ is an odd prime, $d\le 5$ and $H$ is an irreducible subgroup of $\GL(d,p)$;

    \item $T^2\lhd G \leq G_0 \wr \S_2$, where $G_0$ is a primitive group of odd degree $n_0$ with non-abelian simple socle $T$, and the wreath product has the product action of degree $n=n_0^2$;

    \item $L\cong \A_5, \A_6, \A_7,$ or $\A_8$;

    \item $L\cong \M_{11}$ or $\M_{12}$;

    \item $L \cong {^2\!B_2}(2^{2f+1})$, where $f\ge 1$ and $G_\alpha\cap L$ is a parabolic subgroup of $L$;

    \item $L$ is a classical simple group over $\F_q$, where $q$ is odd;

    \item $L$ is a classical simple group over $\F_q$, where $q$ is even and $G_\alpha\cap L$ is a parabolic subgroup of $L$.
\end{enumerate}
We only need to dealt with cases (1), (5), (6) and (7).

\textbf{Claim 1.} Suppose that $G\cong \ZZ_p^d{:} G_{\alpha}$, where $p$ is an odd prime, $d\le 5$ and $G_{\alpha}$ is an irreducible subgroup of $\GL(d,p)$. Then $d\le 4$.

Assume that $d=5$.
Since $a_2+a_3+a_4 \le 5$ and $d=5$, it follows that $G_{\alpha}$ is a 2-group.
However, according to \cite[Theorem 1.1]{Carter}, a Sylow $2$-subgroup of $\GL(5,p)$ is contained in a proper parabolic subgroup (i.e., it stabilizes a nontrivial subspace).
By Sylow's theorems, all Sylow $2$-subgroups of $\GL(5,p)$ are conjugate, and hence every such subgroup is reducible. This shows that $G_{\alpha}$ cannot be a $2$-group, a contradiction.
Thus $d\le 4$, and Claim 1 holds.

\textbf{Claim 2.} Suppose that $L \cong {^2\!B_2}(q)$, where $q=2^{2f+1}$ and $f\ge 1$.
Then $f=1$ or $2$.

Note that $|{^2\!B_2}(q)| = q^2 (q^2 + 1)(q - 1)$.
Let $A=2^{2f+1}+2^{f+1}+1$ and $B=2^{2f+1}-2^{f+1}+1$.
Then $A\cdot B=q^2 + 1$. Since $A-B=2^{f+2}$, it follows that $(A,B)=1$. Since $(q^2 + 1)+(q - 1)=q(q+1)$ and $(q-1,q+1)=1$, it follows that $(q^2 + 1,q - 1)=1$.
Since $A>1$, $B>1$, $q-1>1$, $q+1>1$, $(A,B)=1$ and $(AB, q-1)=1$, we can conclude that $|\pi(G)|=4$. Thus $A$ and $B$ are prime powers.
Observe that $q^2+1=4^{2f+1}+1$. Thus $5\mid (q^2+1)$ and consequently, $5\mid A$ or $B$.
First, we assume that $5\mid A$, that is, $2^{2f+1}+2^{f+1}+1=5^b$ for some integer $b\ge 1$.
If $f\ge 5$, then $5^b\ge 5^4$, contradicting the fact that $a_2+a_3+a_4 \le 5$. Thus $f\le 4$. However, by a simple computation, this case cannot be occur.
Finally, we assume that $5\mid B$.
Since $|\pi(G)|=4$,  we conclude that $B=5^b$ for some integer $b\ge 1$.
Since $5^b= 2^{2f+1}-2^{f+1}+1$ and by assumption that $a_2+a_3+a_4 \le 5$, we conclude that $f\le 2$ and consequently $L\cong {^2\!B_2}(2^3)$ or ${^2\!B_2}(2^5)$. Therefore Claim 2 holds.

\textbf{Claim 3.} Suppose that $L$ is a classical simple group over $\F_q$, where $q=p^f$ is odd, $p$ is a prime and $f\ge 1$. Then
$L\cong \PSL(2, p), \PSL(2, 9), \PSL(2, 25), \PSL(2, 49), \PSL(2,27)$, $\PSL(3,3)$, $\PSL(3,5)$, $\PSp(4,3)$, \allowbreak $\PSU(3,3)$ or $\PSU(3,7)$, where $p\ge 5$ is an odd prime.

Suppose that $L\cong \PSL(n,q)$, where $q=p^f$ and $n\ge 2$.
If $n\ge 4$, since $|\PSL(n,q)|=\frac{1}{(n,q-1)}q^{\frac{n(n-1)}{2}}\prod_{i=2}^{n}(q^i-1)$, it follows that $q^6\mid |G|$, which contradicts the assumption that $a_2+a_3+a_4 \le 5$.
Thus $n\le 3$.

Assume that $n=3$. Since $a_2+a_3+a_4 \le 5$, we have that $f=1$ and hence
$|L|=p^3(p^2-1)(p^3-1)/(3,p-1)=p^3 (p-1)^2 (p+1) (p^2+p+1)/(3,p-1)$.
Suppose that $3\mid (p-1)$.
Since $a_2+a_3+a_4\le 5$, it follows that $p+1$ is a power of $2$ and $(p-1)_{2'}=3$. By Lemma \ref{num3} (1), we get that $p=7$ and hence $L\cong \PSL(3,7)$. However, since $|L|=|\PSL(3,7)|=2^{5} \cdot 3^{2} \cdot 7^{3} \cdot 19$, this contradicts the assumption that $a_2+a_3+a_4\le 5$.
Suppose that $3\nmid (p-1)$. Then $|L|=p^3 (p-1)^2 (p+1) (p^2+p+1)$.
Since $a_2+a_3+a_4\le 5$, it follows that $p-1$ is a power of 2 and $(p+1)_{2'}$ is either 1 or an odd prime.
By Lemma \ref{num3} (2), we have that $p=3$ or $p=5$.  Thus $L\cong \PSL(3,3)$ or $\PSL(3,5)$.

Now assume that $n=2$. Since $a_2+a_3+a_4 \le 5$, it follows that $f\le 4$.
Assume that $f = 4$. Since $p$ is an odd prime, there exist an integer $k\ge 1$ such that $p=2k+1$. It follows that $(p^4+1)/2=8k^{4}+16k^{3}+12k^{2}+4k+1$ is odd. Since $|L|=p^4(p^8-1)/2=p^4(p^2-1)(p^2+1)(p^4+1)/2$ and $a_2+a_3+a_4 \le 5$, it follows that both $p^2-1$ and $p^2+1$ are powers of 2. Lemma \ref{num2} implies that $p=3$ and hence $L\cong \PSL(2,3^4)$.
By a direct calculation, we get that $|L|=|\PSL(2,3^4)| = 2^4 \cdot 3^4 \cdot 5 \cdot 41$, which contradicts the assumption that $a_2+a_3+a_4 \le 5$. Therefore $f\le 3$.
Now we assume that $f=3$.
Observe that $|L|=p^3(p^6-1)/2=p^3(p-1)(p+1)(p^2-p+1)(p^2+p+1)/2$, both $p^2-p+1>1$ and $p^2+p+1>1$ are odd numbers.
Since $a_2+a_3+a_4 \le 5$, we see that $(p-1)(p+1)$ is a power of $2$. Lemma \ref{num2} implies that $p=3$. Thus $L\cong \PSL(2,27)$.
Finally we assume that $f=2$. Then $|L|=p^2(p^4-1)/2=p^2(p-1)(p+1)(p^2+1)/2$.
Since $p$ is an odd prime, it follows that $p=2k+1$ for some integer $k\ge 1$.
Thus $(p^2+1)/2=2k^2+2k+1$ is an odd number. Note that $(p^2-1, p^2+1)=2$.
Let $p_1$ be a prime factor of $(p^2+1)_{2'}$. Now $\{2,p,p_1\}\subseteq \pi(G)$. Since $ |\pi(G)|\le 4$, it follows that $((p-1)(p+1))_{2'}=1$ or $((p-1)(p+1))_{2'}=r^a$ for some odd prime $r$ and integer $a\ge 1$. Since $(p-1, p+1)=2$, this implies that one of $(p-1)$ or $(p+1)$ is a power of 2.
By Lemma \ref{num3}, we get that $p=3,5,7,17$.
By a simple calculation, we have that
\[
\begin{aligned}
|\PSL(2,9)| &= 2^3 \cdot 3^2 \cdot 5,           &\quad |\PSL(2,25)| &= 2^3 \cdot 3 \cdot 5^2 \cdot 13, \\
|\PSL(2,49)| &= 2^4 \cdot 3 \cdot 5^2 \cdot 7^2, &\quad |\PSL(2,17^2)| &= 2^6 \cdot 3^2 \cdot 5 \cdot 17 \cdot 19 \cdot 29.
\end{aligned}
\]
Since $|\pi(G)|\le 4$, this implies that $p=3,5,7$. Consequently $L\cong \PSL(2,9), \PSL(2,25)$ or $\PSL(2,27)$. Moreover, since $\PSL(2,9)\cong \A_6$, this case falls into Proposition \ref{d=5} (3).

Suppose that $L\cong \PSp(2n,q)$, where $q=p^f$ and $n\ge 1$.
Since $ |\PSp(2n, p^f)| = \frac{1}{2} p^{f n^2} \prod_{i=1}^{n} (p^{2fi} - 1)$ and $a_2+a_3+a_4 \le 5$,  we can conclude that $n\le 2$.
Assume that $n=2$. Then $f=1$ as $a_2+a_3+a_4 \le 5$. Since $p^4(p^2-1)^2 (p^2+1)\mid |G|$ and $a_2+a_3+a_4 \le 5$, it follows that $((p^2-1)^2 (p^2+1))_{2'}$ is an odd prime.
Since $(p^2-1,p^2+1)=2$, it follows that one of $p^2-1$ or $p^2+1$ is a power of 2, and Lemma \ref{num2} implies that $p=3$. Consequently $L\cong \PSp(4,3)$.
Now assume that $n=1$. Since $\PSp(2,p^f)\cong \PSL(2,p^f)$, by an analogous argument to the case where $\PSL(2,p^f)$, we obtain that $L\cong \PSL(2, p), \PSL(2, 9), \PSL(2, 25), \PSL(2, 49)$, or $\PSL(2,27)$.

Suppose that $L\cong \PSU(n,q)$, where $q=p^f$ and $n\ge 1$.
Since $$|\PSU(n,q)| = \frac{1}{(n,q+1)} q^{\frac{n(n-1)}{2}} \prod_{i=2}^{n} (q^{i} - (-1)^{i})$$ and $a_2+a_3+a_4 \le 5$,  we can conclude that $n\le 3$.
Assume that $n=3$. The assumption that $a_2+a_3+a_4 \le 5$ implies that $f=1$.
Clearly $p^2-p+1>3$ is an odd number.
Since $a_2+a_3+a_4 \le 5$ and $|L|=\frac{p^3(p^2-1)(p^3+1)}{(3,p+1)}=\frac{p^3(p-1) (p+1)^2 (p^2-p+1)}{(3,p+1)}$, we can conclude that one of $(p-1)/2$ or $(p+1)/2$ is a power of 2 and the other is an odd prime.
By Lemma \ref{num3}, we get that $p=3,5,7,17$.
Since $|\PSU(3,5)|=2^4 \cdot 3^2 \cdot 5^3 \cdot 7$,  $|\PSU(3,17)|=2^6 \cdot 3^5 \cdot 7 \cdot 13 \cdot 17^3$ and $a_2+a_3+a_4 \le 5$, it follows that the cases where $L\cong \PSU(3,5)$ and $L\cong \PSU(3,17)$ cannot occur. Therefore, $L\cong \PSU(3,3)$ or $\PSU(3,7)$.
Now assume that $n=2$. Since $\PSU(2,p^f)\cong \PSL(2,p^f)$, by an analogous argument to the case where $\PSL(2,p^f)$, we obtain that $L\cong \PSL(2, p), \PSL(2,9), \PSL(2, 25), \PSL(2, 49)$, or $\PSL(2,27)$.

Suppose that $L\cong \Omega(2m+1,q)$,  $P\Omega^+(2m,q)$ or $P\Omega^-(2m,q)$.
Since $|\mathrm{\Omega}(2m+1,q)| = \frac{1}{(2,q-1)} q^{m^2} \prod_{i=1}^{m} (q^{2i} - 1)$,
$|\mathrm{P\Omega}^+(2m,q)| = \frac{1}{(4,q^m-1)} q^{m(m-1)} (q^m - 1) \prod_{i=1}^{m-1} (q^{2i} - 1)$, and $|\mathrm{P\Omega}^-(2m,q)| = \frac{1}{(4,q^m+1)} q^{m(m-1)} (q^m + 1) \prod_{i=1}^{m-1} (q^{2i} - 1)$ and $a_2+a_3+a_4 \le 5$,
we can conclude that $m\le 2$ and hence $2m\le 4$ and $2m+1\le 5$. Since all orthogonal simple groups of dimension less than 8 are isomorphic to either $\PSL(m_1,q)$, $\PSp(m_1,q)$ or $\PSU(m_1,q)$ for some suitable integer $m_1\le 6$, this case has already been dealt with in the preceding discussion.

In conclusion, we have that  $L\cong \PSL(2,p)$, $\PSL(3,3)$, $\PSL(3,5)$, $\PSL(2,27)$, $\PSL(2,9)$, $\PSL(2,25)$, $\PSL(2,49)$, $\PSp(4,3)$, $\PSU(3,3)$ or $\PSU(3,7)$,  where $p\ge 5$ is an odd prime. Therefore Claim 3 holds.

\textbf{Claim 4.} Suppose that $L$ is a classical simple group over $\F_q$, where $q$ is even and $G_\alpha\cap L$ is a parabolic subgroup of $L$. Then $L\cong \PSL(2, 2^f)$, $\A_5$, $\A_8$, $\PSL(2,7)$, $\PSL(2,16)$, $\PSL(3, 4)$, $\PSL(3, 8)$, $\PSL(4, 2)$, $\PSp(4,3)$, $\PSp(4,4)$, $\PSU(3,4)$, where $f\ge 3$ is a prime.

Suppose that $L\cong \PSL(n,q)$, where $q=2^f$ and $n\ge 2$.
If $n\ge 6$, then $(q^2-1)(q^3-1)(q^4-1)(q^5-1)(q^6-1)/(n,q-1)=(q^2-1)^2 (q^3-1)^2 (q^2+1)(q^5-1)(q^3+1)/(n,q-1)\mid |G|$, which contradicts the assumption that $a_2+a_3+a_4 \le 5$. Thus $n\le 5$.
Assume that $n=5$. Then $(q^2-1)(q^3-1)(q^4-1)(q^5-1)/(5,q-1)\mid |G|$. If $q\ge 4$, then  Zsigmondy's theorem\cite[p.508]{hu2} implies that $|\pi(G)|\ge 5$, contradicting the assumption that $|\pi(G)|\le 4$.
If $q=2$, then $L\cong \PSL(5,2)$ and $|L|=|\PSL(5,2)|=2^{10}\cdot 3 \cdot 5 \cdot 7 \cdot 31$, and consequently $|\pi(G)|\ge 5$, which contradicts the assumption that $|\pi(G)|\le 4$.
Thus $n\le 4$.
Assume that $n=4$. Then $(q^2-1)(q^3-1)(q^4-1)=(q-1)^3(q+1)(q^2+1)(q^2+q+1)\mid |G|$,
Since $a_2+a_3+a_4 \le 5$, we see that $q-1=1$ and so $q=2$. Thus $L\cong \PSL(4,2)\cong \A_8$, which falls into Proposition \ref{d=5} (3).

Next, we assume that $n=3$. Then $(q^2-1)(q^3-1)/(3,q-1)=(2^f-1)^2 (2^f+1) (2^{2f}+2^f+1)/(3,2^f-1)\mid |G|$.
Suppose that $f=2k$ for $k\ge 1$. Then $(2^f-1)=(2^{k}-1)(2^{k}+1)$ and hence $(2^{k}-1)^2 (2^{k}+1)^2 (2^f+1)(2^{2f}+2^f+1)/(3,q-1)\mid |G|$.
Assume that $3\nmid (q-1)$. Then $(2^{k}-1)^2(2^{k}+1)^2 (2^f+1)(2^{2f}+2^f+1)\mid |G|$. Since $a_2+a_3+a_4 \le 5$, it follows that $2^k-1=1$ and $k=1$, and consequently $3\mid (2^k+1)=3$, which contradicts the assumption that $3\nmid (q-1)=2^f-1$.
Thus $3\mid (q-1)$, and so $3\mid (2^{k}-1)$ or $3\mid (2^{k}+1)$.
Since $a_2+a_3+a_4 \le 5$, it follows that $2^k-1=3$ or $2^k+1=3$ and hence $k=1$ or $2$.
If $k=2$, then $q=2^4$ and $L\cong \PSL(3,2^4)$, and since $|L|=|\PSL(3,2^4)|=2^{12}\cdot 3^{2}\cdot 5^{2}\cdot 7\cdot 13\cdot 17$, which contradicts the assumption that $|\pi(G)|\le 4$. Consequently $k=1$ and $L\cong \PSL(3,4)$.
Now assume that $f$ is odd. If $f=1$, then $L\cong \PSL(3,2)\cong \PSL(2,7)$, as desired. In the following, we assume that $f\ge 3$ and write $f = ab$ for some prime $a$ and integer $b \ge 1$. Since $2^f\equiv (-1)^f \pmod{3}$ and $f$ is odd, we get that $(3,2^f-1)=1$.  Observe that $(2^f-1)^2 (2^f+1) (2^{2f}+2^f+1)\mid |G|$ and  $2^f+1=2^{ab}+1=(2^b+1)(2^{b(a-1)}+\cdots+1)=3(2^{b-1}+\cdots+1)(2^{b(a-1)}+\cdots+1).$
Since $a_2+a_3+a_4 \le 5$ and $3(2^{b-1}+\cdots+1)(2^{b(a-1)}+\cdots+1) (2^f-1)^2 (2^{2f}+2^f+1)\mid |G|,$
we can conclude that $b=1$.
This shows that $f$ is an odd prime, and hence $2^f+1=3(2^{f-1}+\cdots+1)$ is a composite number, and consequently $3 (2^{f-1}+\cdots+1) (2^f-1)^2 (2^{2f}+2^f+1) \mid |G|$.
Since $a_2+a_3+a_4 \le 5$, we conclude that $2^f-1$, $\frac{2^{f}+1}{3}$ and $2^{2f}+2^f+1$ both are odd primes.
Since $(2^{2f}+2^f+1) + (2^f-1)=2^{f+1} (2^{f-1}+1)$ and
$(2^f-1) + (2^{f-1} + 1) = 2^{f-1}(2+1) = 3\cdot 2^{f-1}$, it follows that $(2^{2f}+2^f+1, 2^f-1)\mid 3$.
Since $2^f-1>3$ is a prime, we conclude that $(2^{2f}+2^f+1, 2^f-1)=1$. Moreover, since $2^{2f}+2^f+1> 2^f+1$ is a prime, we get that $\pi(G)=\{2, 3, 2^f-1, \frac{2^{f}+1}{3}, 2^{2f}+2^f+1\}$. If $f>3$, this would contradict the assumption that $|\pi(G)|\le 4$. Thus $f=3$ and hence $L\cong \PSL(3,8)$.

Assume that $n=2$, that is, $L \cong \PSL(2, 2^f)$. Since $\PSL(2,4) \cong \A_5$ falls into Proposition \ref{d=5} (3), we may assume that $f \ge 3$.
Assume first that $f$ is not a power of $2$, that is, $f = ab$ for $a > 1$ is an odd prime and $b \ge 1$. We assert that $f$ is an odd prime.
Suppose that $b=2k$ for some positive integer $k$. Note that  $|L|=|\PSL(2,2^f)|=2^f(2^f-1)(2^f+1)=2^f (2^{f/2}-1) (2^{f/2}+1) (2^b+1) (2^{b(a-1)}-2^{b(a-2)}+\cdots+1)$.
Let $p_1\mid 2^{f/2}-1$, $p_2\mid 2^{f/2}+1$ and $p_3\mid 2^f+1$, where $p_1,p_2,p_3$ are primes. Since $(2^{f/2}-1, 2^{f/2}+1)=1$ and $(2^f-1,2^f+1)=1$, it follows from $|\pi(G)|\le 4$ that
$\pi(G)=\{2,p_1,p_2,p_3\}$. It follows that
$2^{f/2}-1=p_1^{k_1}$, $2^{f/2}+1=p_2^{k_2}$ and $2^{f}+1=p_3^{k_3}$ for some positive integers $k_i$ with $i\in\{1,2,3\}$.
Since $a\ge 3$ is a prime and $2^{b(a-1)}-2^{b(a-2)}+\cdots+1>2^b+1$, we get that $p_3^2\mid (2^{b(a-1)}-2^{b(a-2)}+\cdots+1)$, and consequently $k_1=k_2=1$ as $a_2+a_3+a_4 \le 5$. These mean that both $2^{f/2}-1$ and $2^{f/2}+1$ are primes. The only positive integer $m$ for which both $2^m - 1$ and $2^m + 1$ are primes is $m = 2$. Hence $f/2 = 2$ and $f = 4$, which contradicts the assumption that $f$ has an odd prime factor $a$.
Consequently, $b$ must be odd. Note that $|L|=2^f(2^f-1)(2^f+1)=2^f (2^a - 1)(2^{a(b-1)} + 2^{a(b-2)} + \dots + 1) (2^a+1) (2^{a(b-1)}-2^{a(b-2)}+\dots+1)$.
Suppose $b>1$.
Applying Zsigmondy's theorem\cite[p.508]{hu2} to $2^f-1=2^{ab}-1$, there exists a primitive prime $r\mid 2^f-1$ but not divisible by $2^{j}-1$ for all $1\le j\le f-1$; in particular, $r\nmid 2^a-1$.
Let $r_1\mid 2^a-1$ and $r_2\mid 2^f+1$, where $r_1, r_2$ are primes.
Since $(2^f-1,2^f+1)=1$ and $|\pi(G)|\le 4$, it follows that $\pi(G)=\{2,r,r_1,r_2\}$. Thus $2^f+1=r_2^t$ for some integer $t\ge 1$. Noting that $3\mid 2^f+1$, we have $r_2=3$ and $2^f+1=3^t$. Since $f=ab\ge 9$, it follows that $t> 5$, which contradicts the assumption that $a_2+a_3+a_4 \le 5$.
Thus $b=1$ and $f=a$ is an odd prime, the assertion holds.

Now consider the case where $f$ is a power of $2$, say $f = 2^k$ for some integer $k \ge 2$. If $k \ge 3$, we obtain that
$2^f-1=(2^{2^{k-2}}-1)(2^{2^{k-2}}+1)(2^{2^{k-1}}+1).$
The factors $2^{2^{k-2}}-1$, $2^{2^{k-2}}+1$, and $2^{2^{k-1}}+1$ are pairwise coprime. Together with the coprime factors $2^f$ and $2^f+1$, we conclude that $|\pi(G)|\ge 5$, which contradicts the assumption that $|\pi(G)|\le 4$.
Therefore, $k \le 2$. Recalling that $f \ge 3$, the only possibility is $k=2$, i.e., $f=4$. Hence $L \cong \PSL(2, 16)$.
In conclusion, we have that $L\cong \A_5, \A_8$, $\PSL(2,7)$, $\PSL(3,4)$, $\PSL(3,8)$, $\PSL(2,16)$, or $\PSL(2,2^f)$ with $f\ge 3$ is an odd prime.

Suppose that $L\cong \PSp(2n,2^f)$, where $f\ge 1$ and $n\ge 1$.
If $n \ge 4$, then $|G|$ is divisible by
$$(q^2-1)(q^4-1)(q^6-1)(q^8-1)=(q^2-1)^4 (q^2+1)^2 (q^4+1) (q^4+q^2+1),$$
which contradicts the assumption that $a_2+a_3+a_4 \le 5$. Consequently $n\le 3$.
Assume that $n=3$. Then $|G|$ is divisible by
$$(q^2-1)(q^4-1)(q^6-1)=(q^2-1)^3 (q^2+1) (q^4+q^2+1).$$
Since $a_2+a_3+a_4 \le 5$, it follows that $q^2-1$ is an odd prime, which shows that $q=2$. However, since $|\PSp(6,2)|=2^9 \cdot 3^4 \cdot 5 \cdot 7$ and $a_2+a_3+a_4 \le 5$, it follows that $L\ncong \PSp(6,2)$, and hence the case $n=3$ cannot occur.
Assume that $n=2$, i.e., $L\cong \PSp(4,2^f)$.
Note that
 $$(q^2-1)(q^4-1)= (q-1)^2 (q+1)^2 (q^2+1)=(2^f-1)^2 (2^f+1)^2 (2^{2f}+1).$$
If $f\ge 4$ is even, then $2^{f/2}-1>1$, and since $|G|$ is divisible by $$(2^f-1)^2 (2^f+1)^2 (2^{2f}+1)=(2^{f/2}-1)^2 (2^{f/2}+1)^2 (2^f+1)^2 (2^{2f}+1),$$
this contradicts the assumption that $a_2+a_3+a_4 \le 5$.
Thus $f=2$ or $f$ is odd. If $f\ge 3$, then
$$(2^f-1)^2 (2^f+1)^2 (2^{2f}+1)=(2^f-1)^2 \cdot (3\cdot(2^{f-1}+\cdots+1))^2\cdot 5\cdot(2^{2f-1}+\cdots+1) \mid |G|,$$
which contradicts the assumption that $a_2+a_3+a_4 \le 5$. Consequently $f\le 2$. Since $\PSp(4,2)\cong \S_6$ is not simple, it follows that $f=2$ and $L\cong \PSp(4,4)$.
Finally, when $n=1$, we have $\PSp(2,2^f)\cong \PSL(2,2^f)$, and as we discussed earlier, we obtain that $L\cong \A_5, \PSL(2,16)$ or $\PSL(2,2^f)$ with $f\ge 3$ is an odd prime.
In conclusion, $L\cong  \PSp(4,4)$, $\A_5, \PSL(2,16)$ or $\PSL(2,2^f)$ with $f\ge 3$ is an odd prime.

Suppose that $L\cong \PSU(n,2^f)$, where $f\ge 1$ and $n\ge 2$.
If $n\ge 5$, then $|G|$ can be divisible by
$$\frac{(q^2-1)(q^3+1)(q^4-1)(q^5+1)}{(n,q+1)}=\frac{(q-1)^2 (q+1)^4 (q^2-q+1) (q^2+1)(q^4-q^3+q^2-q+1)}{(n,q+1)},$$
which contradicts the assumption that $a_2+a_3+a_4 \le 5$.
Consequently $n\le 4$.
Assume that $n=4$. Then $(n,q+1)=1$.  Since $|G|$ is divisible by $(q^2-1)(q^3+1)(q^4-1)=(q-1)^2 (q+1)^3 (q^2-q+1) (q^2+1)$ and $a_2+a_3+a_4 \le 5$, it follows that $q-1=1$ and hence $q=2$.
In this case, since $\PSU(4,2)\cong \PSp(4,3)$, we have that $L\cong \PSp(4,3)$, which falls into Proposition \ref{d=5} (4).

Now assume that $n=3$.
Then $|G|$ is divisible by
$$\frac{(q^2-1)(q^3+1)}{(3,q+1)}=\frac{(2^f-1)(2^f+1)^2 (2^{2f}-2^f+1)}{(3,2^f+1)}.$$
Suppose that $3\nmid 2^f+1$.
If $f=ab$ with $a>1$ an odd prime, then $2^f+1$ is an odd composite number and hence $a_2+a_3+a_4\ge 6$,
which contradicts the assumption that $a_2+a_3+a_4 \le 5$.
Thus $f$ is a power of 2, that is, $f=2^k$ for some positive integer $k$.
If $k\ge 2$, then $|G|$ is divisible by
$(2^{2^{k-1}}-1)(2^{2^{k-1}}+1)(2^f+1)^2 (2^{2f}-2^f+1).$
Since $a_2+a_3+a_4 \le 5$, it follows that both $(2^{2^{k-1}}-1)$ and $(2^{2^{k-1}}+1)$ are primes, which shows that $k=2$ and so $f=4$. But then $|L|=|\PSU(3,16)|=2^{12} \cdot 3 \cdot 5 \cdot 17^2 \cdot 241$, which contradicts the assumption that $|\pi(G)|\le 4$.  Therefore $k=1$ and so $L\cong \PSU(3,4)$.
Now suppose that $3\mid 2^f+1$. Since $2^f \equiv (-1)^f \pmod{3}$, it follows that $f$ is odd.
Since $\PSU(3,2)$ is not simple and $L$ is simple, we may assume that $f\ge 3$ in what follows. Thus $2^f+1$ is composite.
If $f=ab$ with $a,b>1$, then $2^f+1=3(2^{a-1}-2^{a-2}+\dots+1)(2^{a(b-1)}-2^{a(b-2)}+\dots+1)$, which gives $a_2+a_3+a_4\ge 6$, contradicting the assumption that $a_2+a_3+a_4 \le 5$. Hence $f$ is an odd prime.
Assume that there exists a prime $r\ne 3$ such that $r\mid 2^f+1$.
Let $t$ be a prime such that $t\mid 2^f-1$. Since $(2^f+1,2^f-1)=1$ and $|\pi(G)|\le 4$,
we have $\pi(G)=\{2,3,r,t\}$. Noting that $(2^{2f}-2^f+1)\mid |G|$ and $2^{2f}-2^f+1>2^f+1$, it follows that $2^{2f}-2^f+1$ is composite, which contradicts the assumption that $a_2+a_3+a_4 \le 5$.
Therefore $2^f+1$ is a power of 3, that is, $2^f+1=3^m$ for some integer $m\ge 2$.
Since $\frac{(2^f-1)(2^f+1)^2 (2^{2f}-2^f+1)}{(3,2^f+1)}\mid |G|$ and $a_2+a_3+a_4 \le 5$, it follows that $m=2$ and so $f=3$. But then $|L|=|\PSU(3,8)|=2^9 \cdot 3^4 \cdot 7 \cdot 19$, which gives $a_2+a_3+a_4\ge 6$, contradicting the assumption that $a_2+a_3+a_4 \le 5$. Thus the case where $n=3$ cannot occur.
Now assume that $n=2$. Since $\PSU(2,2^f)\cong \PSL(2,2^f)$, by an analogous argument to the case where $\PSL(2,2^f)$, we obtain that $L\cong \A_5, \PSL(2,16)$ or $\PSL(2,2^f)$ with $f\ge 3$ is an odd prime. In conclusion, $L\cong \PSp(4,3), \PSU(3,4), \A_5, \PSL(2,16)$ or $\PSL(2,2^f)$ with $f\ge 3$ is an odd prime.

Suppose that $L\cong \Omega^+(2m+1,q)$, $P\Omega^+(2m,q)$ or $P\Omega^-(2m,q)$, where $q=2^f$.
Recall that $|\mathrm{P\Omega}(2m+1,q)| = \frac{1}{(2,q-1)} q^{m^2} \prod_{i=1}^{m} (q^{2i} - 1)$,
$|\mathrm{P\Omega}^+(2m,q)| = \frac{1}{(4,q^m-1)} q^{m(m-1)} (q^m - 1) \prod_{i=1}^{m-1} (q^{2i} - 1)$, and $|\mathrm{P\Omega}^-(2m,q)| = \frac{1}{(4,q^m+1)} q^{m(m-1)} (q^m + 1) \prod_{i=1}^{m-1} (q^{2i} - 1)$.
If $m\ge 4$, then $(q^2-1)(q^4-1)(q^6-1)= (q^2-1)^3 (q^2+1) (q^4+q^2+1)\mid |G|$, and since $a_2+a_3+a_4 \le 5$, it follows that $q^2-1$ and $q^2+1$ are primes, which implies that $q=2$.
However, since $q^4+q^2+1=21$, we see that $a_2+a_3+a_4 \ge 6$, contradicting the assumption that $a_2+a_3+a_4\le 5$.
Thus $m\le 3$ and hence $2m\le 6$ and $2m+1\le 7$. Since all orthogonal simple groups of dimension less than 8 are isomorphic to either $\PSL(m_1,q)$, $\PSp(m_1,q)$ or $\PSU(m_1,q)$ for some suitable integer $m_1\le 7$, this case has already been dealt with in the preceding discussion.

In conclusion, we have that $L\cong \PSL(2, 2^f)$, $\A_5$, $\A_8$, $\PSL(2,7)$, $\PSL(2,16)$, $\PSL(3, 4)$, $\PSL(3, 8)$, $\PSp(4,3)$, $\PSp(4,4)$ or $\PSU(3,4)$, where $f\ge 3$ is a prime. Therefore Claim 4 holds.

According to Claims 1--4, one of the following holds:
\begin{enumerate} [font=\normalfont]
    \item $G\cong \ZZ_p^d{:} G_{\alpha}$, where $p$ is an odd prime, $d\le 4$ and $G_{\alpha}$ is an irreducible subgroup of $\GL(d,p)$;

    \item $T^2\lhd G \leq G_0 \wr \S_2$, where $G_0$ is a primitive group of odd degree $n_0$ with non-abelian simple socle $T$, and the wreath product has the product action of degree $n=n_0^2$;

    \item $L\cong \A_5, \A_6, \A_7$, $\A_8$, $\M_{11}$ or $\M_{12}$;

    \item $L\cong \PSL(2, p), \PSL(2, 25), \PSL(2, 49), \PSL(2,27)$, $\PSL(3,3)$, $\PSL(3,5)$, $\PSp(4,3)$, $\PSU(3,3)$ or $\PSU(3,7)$, where $p\ge 5$ is an odd prime;

    \item $L\cong \PSL(2, 2^f)$, $\PSL(2,16)$, $\PSL(3, 4)$, $\PSL(3, 8)$, $\PSp(4,4)$, $\PSU(3,4)$, ${^2\!B_2}(2^{3})$ or ${^2\!B_2}(2^{5})$, where $f\ge 3$ is a prime, and $L\cap G_{\alpha}$ is a parabolic subgroup of $L$.
\end{enumerate}
\end{proof}

\begin{proof}[{\bf The proofs of Proposition \ref{forK}}]
Recall the assumption that $G$ is a finite non-abelian simple group, $3\le|\pi(G)|\le 4$ and  $|G|=2^{a_1} p_2^{a_2} p_3^{a_3} p_4^{a_4}$, where $2<p_2<p_3<p_4$ are primes, $a_2,a_3\ge 1$ and $a_4\ge 0$.  We also assume that $a_2+a_3+a_4\le 3$.
According to the classification theorem of finite simple groups(CFSG), and since $3\le|\pi(G)|\le 4$ and $a_2+a_3+a_4\le 3$, we can conclude that $G$ is one of the following:
\begin{enumerate} [font=\normalfont]

    \item $G\cong \A_5, \A_6$;

    \item $G \cong {^2\!B_2}(2^{2f+1})$, where $f\ge 1$;

    \item $G$ is a non-abelian classical simple group over $\F_q$, where $q=p^f$, $p$ is a prime and $f\ge 1$.

\end{enumerate}

\textbf{Claim 1.} Assume that $G \cong {^2\!B_2}(q)$, where $q=2^{2f+1}$ and $f\ge 1$. Then $G \cong {^2\!B_2}(8)$ and $|G|_{2'}$ is square-free.

Recall that $|{^2\!B_2}(q)| = q^2 (q^2 + 1)(q - 1)$.
Let $A=2^{2f+1}+2^{f+1}+1$ and $B=2^{2f+1}-2^{f+1}+1$.
Then $A\cdot B= q^2 + 1$. Since $A-B=2^{f+2}$, it follows that $(A,B)=1$. Since $(q^2 + 1)+(q - 1)=q(q+1)$ and $(q-1,q+1)=1$, it follows that $(q^2 + 1,q - 1)=1$.
Since $|G|=q^2 (q^2 + 1)(q - 1)$, we can conclude that $|\pi(G)|=4$.
Since $q^2+1=4^{2f+1}+1$, it follows that $5\mid q^2+1$. Moreover, since $a_2+a_3+a_4\le 3$, we conclude that $A=5$ or $B=5$.
Clearly $A\ne 5$. Thus $B=5$, so that $f=1$ and $G\cong {^2\!B_2}(2^3)$. In particular, $|G|_{2'}$ is square-free. Therefore Claim 1 holds.

\textbf{Claim 2.} Suppose that $G$ is a non-abelian classical simple group over $\F_q$ except for the cases $G\cong \A_5$ and $\A_6$. Here $q=p^f$, $p$ is a prime and $f\ge 1$. Then
$G\cong \PSL(2,7)$, $\PSL(2,16)$, $\PSL(2,31)$, $\PSL(2,8)$, $\PSL(2,17)$, $\PSL(2,2^f)$ or   $\PSL(2,p)$ with $((p-1)_{2'}, (p+1)_{2'})\in \{(r_1,r_2), (1,3r)\}$, where $r_1, r_2$ and $r>3$ are odd primes,  $f>3$, $2^f-1$ and $\frac{2^f+1}{3}>3$ are odd primes.

Since $a_2+a_3+a_4\le 3$,  applying Proposition \ref{d=5}, we get that $G$ is one of the following:
\begin{enumerate}
    \item $G\cong \PSL(2, p), \PSL(2, 25), \PSL(2, 49), \PSL(2,27)$, $\PSL(3,3)$, $\PSL(3,5)$, $\PSp(4,3)$, $\PSU(3,3)$ or $\PSU(3,7)$, where $p\ge 5$ is an odd prime.

    \item $G\cong \PSL(2, 2^f)$, $\PSL(2, 8)$, $\PSL(2, 16)$, $\PSL(3, 4)$, $\PSL(3, 8)$, $\PSL(4, 2)$,  $\PSp(4,4)$ or $\PSU(3,4)$,  where $f> 3$ is a prime.
\end{enumerate}

\begin{longtable}{clll}
\caption{Orders of finite simple groups} \label{T2}\\
\toprule
No. & Group & Order & Prime factorization \\
\midrule
\endfirsthead

\multicolumn{4}{c}{{\tablename\ \thetable{} -- continued from previous page}} \\
\toprule
No. & Group & Order & Prime factorization \\
\midrule
\endhead

\bottomrule
\endlastfoot

1 & $\PSL(2,8)$ & $8 \times (64-1) = 504$ & $2^3 \cdot 3^2 \cdot 7$ \\
2 & $\PSL(2,16)$ & $16 \times (256-1) = 4080$ & $2^4 \cdot 3 \cdot 5 \cdot 17$ \\
3 & $\PSL(2, 25)$ & $25 \times (625-1) / 2 = 7800$ & $2^3 \cdot 3 \cdot 5^2 \cdot 13$ \\
4 & $\PSL(2, 27)$ & $27 \times (729-1) / 2 = 9828$ & $2^2 \cdot 3^3 \cdot 7 \cdot 13$ \\
5 & $\PSL(2, 49)$ & $49 \times (2401-1) / 2 = 58800$ & $2^4 \cdot 3 \cdot 5^2 \cdot 7^2$ \\
6 & $\PSL(3, 3)$ & $27 \times 8 \times 26 = 5616$ & $2^4 \cdot 3^3 \cdot 13$ \\
7 & $\PSL(3, 4)$ & $64 \times 15 \times 63 / 3 = 20160$ & $2^6 \cdot 3^2 \cdot 5 \cdot 7$ \\
8 & $\PSL(3, 5)$ & $125 \times 24 \times 124 = 372000$ & $2^5 \cdot 3 \cdot 5^3 \cdot 31$ \\
9 & $\PSL(3, 8)$ & $512 \times 63 \times 511 = 16482816$ & $2^9 \cdot 3^2 \cdot 7^2 \cdot 73$ \\
10 & $\PSL(4, 2)$ & $64 \times 3 \times 7 \times 15 = 20160$ & $2^6 \cdot 3^2 \cdot 5 \cdot 7$ \\
11 & $\PSp(4, 3)$ & $81 \times 8 \times 80 / 2 = 25920$ & $2^6 \cdot 3^4 \cdot 5$ \\
12 & $\PSp(4, 4)$ & $256 \times 15 \times 255 = 979200$ & $2^8 \cdot 3^2 \cdot 5^2 \cdot 17$ \\
13 & $\PSU(3, 3)$ & $27 \times 28 \times 8 = 6048$ & $2^5 \cdot 3^3 \cdot 7$ \\
14 & $\PSU(3, 4)$ & $64 \times 65 \times 15 = 62400$ & $2^6 \cdot 3 \cdot 5^2 \cdot 13$ \\
15 & $\PSU(3, 7)$ & $343 \times 344 \times 48 = 5663616$ & $2^7 \cdot 3 \cdot 7^3 \cdot 43$ \\
\bottomrule
\end{longtable}
By the assumption that $a_2+a_3+a_4\le 3$, if $G$ is isomorphic to a group in Table \ref{T2}, then $G$ can only be $\PSL(2,8)$ or $\PSL(2,16)$.
Now we only need to consider the remaining cases $\PSL(2,p)$ and $\PSL(2,2^f)$, where $p$ and $f$ are primes greater than 3.

Assume that $G\cong \PSL(2,p)$ where $p\ge 5$ is an odd prime.
Then $|G|=p(p-1)(p+1)/2$.
If $|\pi(G)|=3$, then $p-1$ or $p+1$ is a power of $2$. Lemma \ref{num3} implies that $p=5,7,17$.
Since $\PSL(2,5)\cong \A_5$, it follows that $L\ncong \PSL(2,5)$ and consequently  $G\cong\PSL(2,7)$ or $\PSL(2,17)$.
Now assume that $|\pi(G)|=4$. Then the condition $a_2+a_3+a_4\le 3$ forces $a_2=a_3=a_4=1$. It follows that either both $p-1$ and $p+1$ are not powers of 2, or
 one of them is a power of 2 and the odd part of the other is a product of two distinct odd primes.
If both $p-1$ and $p+1$ are not powers of 2, then $((p-1)_{2'}, (p+1)_{2'})=(r_1,r_2)$, where both $r_1$ and $r_2$ are primes, as desired.
Now assume that one of $p-1$ and $p+1$ is a power of 2 and the odd part of the other is a product of two distinct odd primes.
First, we assume that $p-1$ is a power of 2, and prime $r\mid (p+1)_{2'}$ with $r>3$.
Let $p-1=2^k$ for $k\ge 2$. Then $p=2^k+1$ is a Fermat prime, and consequently $k$ is a power of $2$. Since $\frac{p+1}{2}=2^{k-1}+1$ and $k-1$ is odd, we get that $3\mid \frac{p+1}{2}$.
Moreover, since $\frac{p+1}{2}$ is a product of two distinct odd primes, it follows that $\frac{p+1}{2}=3r$ and $((p-1)_{2'},(p+1)_{2'})=(1,3r)$, as desired.
Finally, we assume that $p+1$ is a power of 2.
Let $p+1=2^k$ for some integer $k\ge 3$.
Then $p=2^k-1$ is a Mersenne prime, and consequently $k$ is a prime.
Given that $\frac{p-1}{2}$ is a product of two distinct odd primes, we get that $k\ge 5$.
Since $\frac{p-1}{2}=2^{k-1}-1$ and $k-1$ is even,  it follows that $\frac{p-1}{2}=(2^{(k-1)/2}-1)(2^{(k-1)/2}+1)$.
Since $|\pi(G)|\le 4$ and $a_2+a_3+a_4\le 3$, it follows that both $2^{(k-1)/2}-1$ and $2^{(k-1)/2}+1$ are odd primes.
If $(k-1)/2$ is odd, then $3\mid (2^{(k-1)/2}+1)$, so that $2^{(k-1)/2}+1$ is a composite number. This contradicts the earlier conclusion that $2^{(k-1)/2}+1$ is a prime.
Therefore $(k-1)/2$ is even. Since $k\ge 5$ and $(2^{(k-1)/2}-1)=(2^{(k-1)/4}-1)(2^{(k-1)/4}+1)$ is a prime, it follows that $k-1=4$ and hence $k=5$. Consequently $p=31$ and $G\cong \PSL(2,31)$.
Therefore, $G\cong \PSL(2,7), \PSL(2,17)$, $\PSL(2,31)$ or $\PSL(2,p)$ with $((p-1)_{2'}, (p+1)_{2'})\in \{(r_1,r_2), (1,3r)\}$, where $r_1, r_2$ and $r>3$ are odd primes, as desired.

Now we assume that $G\cong \PSL(2,2^f)$, where $f>3$ is a prime.
Since $f>3$ is a prime, it follows that $2^f+1=3(2^{f-1}-2^{f-2}+\cdots+1)$ and $2^{f-1}-2^{f-2}+\cdots+1>3$.
If $2^{f-1}-2^{f-2}+\cdots+1$ is not a prime, then one yields that $a_2+a_3+a_4\ge 4$, which contradicts the assumption that $a_2+a_3+a_4\le 3$.
Thus $(2^f+1)/3$ is a prime, as desired.

In summery, $G\cong \PSL(2,7)$, $\PSL(2,16)$, $\PSL(2,31)$, $\PSL(2,8)$, $\PSL(2,17)$, $\PSL(2,2^f)$ or   $\PSL(2,p)$ with $((p-1)_{2'}, (p+1)_{2'})\in \{(r_1,r_2), (1,3r)\}$, where $r_1, r_2$ and $r>3$ are odd primes,  $f>3$, $2^f-1$ and $\frac{2^f+1}{3}>3$ are odd primes.  Therefore Claim 2 holds.

According to Claims 1 and 2, the following holds:
\begin{enumerate}[font=\normalfont, label=(\arabic*)]
\item If $p_2^{a_2} p_3^{a_3} p_4^{a_4}$ is square-free, then $G$ is isomorphic to $\A_5$, ${^2\!B_2(8)}$, $\PSL(2,7)$, $\PSL(2,16)$, $\PSL(2,31)$, $\PSL(2, p_4)$ or $\PSL(2,2^f)$, where $r_1, r_2$ and $r>3$ are odd primes, $((p_4-1)_{2'}, (p_4+1)_{2'})\in \{(r_1,r_2), (1,3r)\}$, $f>3$, $2^f-1$ and $\frac{2^f+1}{3}>3$ are odd primes.

\item If $|G|_{2'}=p_i^2 p_j$, then $G$ is isomorphic to $\A_6$, $\PSL(2,8)$ or $\PSL(2,17)$, where $i\ne j$ and $i,j\in \{2,3\}$.
\end{enumerate}
\end{proof}

\end{appendices}

\makeatletter
\let\@setaddresses\relax
\makeatother
\end{document}